\documentclass[12pt]{amsart}

\renewcommand{\labelenumi}{\theenumi.}

\usepackage{amsmath}
\usepackage{amssymb}
\usepackage{amsthm}
\usepackage[latin1]{inputenc}
\usepackage{eurosym}
\usepackage[dvips]{graphics}
\usepackage{graphicx}
\usepackage{epsfig}
\usepackage{hyperref}
\usepackage{dsfont}
\usepackage{epstopdf}
\usepackage{color}

\usepackage{ifthen}

\usepackage{graphicx}
\usepackage{caption}
\usepackage{subcaption}
\newtheorem{theorem}{Theorem}[section]

\newtheorem{lemma}[theorem]{Lemma}
\newtheorem{proposition}[theorem]{Proposition}

\theoremstyle{definition}
\newtheorem{definition}[theorem]{Definition}

\newtheorem{remark}[theorem]{Remark}

\newcommand{\eps}{\varepsilon}					   % Epsilon
\newcommand{\Tt}{{\mathbb{T}}}

\newcommand{\bx}{{\bf x}}
\newcommand{\bu}{{\bf u}}

\newcommand{\lb}{\left(}
\newcommand{\rb}{\right)}

\def\appendixname{\empty}

\def\refname{References}

\def\figurename{Fig.}
\def\abstractname{Abstract}

\usepackage[usenames,dvipsnames,svgnames,table]{xcolor}
\definecolor{darkgreen}{rgb}{0.,.66,0.}
\hypersetup{%
  breaklinks,
  bookmarks=true,
  bookmarksnumbered=true,
  bookmarksopen=true,
  colorlinks=true,
  linkcolor=blue!60!black,
  urlcolor=blue!60!black,
  citecolor=blue!60!black
}
\newboolean{include-notes}
\setboolean{include-notes}{true}
\newcommand{\js}[1]{\ifthenelse{\boolean{include-notes}}%
 {\textcolor{cyan}{\textbf{[ #1  --JS]}}}{}}
\newcommand{\dg}[1]{\ifthenelse{\boolean{include-notes}}%
 {\textcolor{darkgreen}{\textbf{[ #1  --DG]}}}{}}
  \newcommand{\lv}[1]{\ifthenelse{\boolean{include-notes}}%
 {\textcolor{magenta}{\textbf{[ #1  --LN]}}}{}}

\theoremstyle{definition}

\begin{document}

\title[Fourier Approximation Methods for Nonlocal First-Order Mean-Field Games]{Fourier Approximation Methods for First-Order Nonlocal Mean-Field Games}

\author{L. Nurbekyan}
\address[L. Nurbekyan]{Department of Mathematics and Statistics, McGill University, 805 Sherbrooke St W, Montreal, H3A 0B9 QC, Canada}
\email{levon.nurbekyan@mcgill.ca}
\author{J. Sa\'ude}
\address[J. Saude]{Electrical and Computer Engineering department, Carnegie Mellon University. 5000 Forbes Avenue Pittsburgh, PA 15213-3890 USA.}
\email{jsaude@cmu.edu}


\date{\today}

\begin{abstract}
In this note, we develop Fourier approximation methods for the solutions of first-order nonlocal mean-field games (MFG) systems. Using Fourier expansion techniques, we approximate a given MFG system by a simpler one that is equivalent to a convex optimization problem over a finite-dimensional subspace of continuous curves. Furthermore, we perform a time-discretization for this optimization problem and arrive at a finite-dimensional saddle point problem. Finally, we solve this saddle-point problem by a variant of a primal dual hybrid gradient method.
\end{abstract}

\maketitle
\tableofcontents

% \begin{classification}
% Primary 35Q91, 35Q93; Secondary 35A01.
% \end{classification}

%35Q91 PDEs in connection with game theory, economics, social and behavioral sciences
%35Q93 PDEs in connection with control and optimization
%35A01 Existence problems: global existence, local existence, non-existence

% \begin{keywords}
% 	Infinite-dimensional differential games, Mean-field games, Nonlocal interactions, Fourier expansions.
% \end{keywords}

\section{Introduction}
The mean-field game (MFG) framework \cite{LL061,LL062,LL07,HCM06,HCM07} models systems with a huge number of small identical rational players (agents) that play non-cooperative differential games. In this framework, a generic player aims at minimizing a cost functional that takes the distribution of the whole population as a parameter. Consequently, the problem is to find a Nash equilibrium where a generic player cannot unilaterally improve his position. For a detailed account on MFG systems we refer the reader to \cite{LCDF,CardaNotes,gll'11,befreyam'13,Gomes2014,GoBook,cardela'18,notebari}.

In this note, we introduce Fourier approximation techniques for first-order nonlocal MFG models. More precisely, we consider the system
\begin{equation}\label{eq:maingeneral}
\begin{cases}
-\partial_t u + H(x,\nabla u) = F[x,m], \\
\partial_t m - \mathrm{div}\left(m \nabla_pH(x,\nabla u)\right)=0,~(x,t) \in \mathbb{T}^d \times [0,1],\\
m(x,0)=M(x),~u(x,1)=U(x),~x\in \mathbb{T}^d.
\end{cases}
\end{equation}
Here, $u: \mathbb{T}^d \times [0,1]\to \mathbb{R}$ and $m:\mathbb{T}^d \times [0,1]\to \mathbb{R}_+$ are the unknown functions. Furthermore, $H\in C^2(\mathbb{T}^d \times \mathbb{R}^d)$ is a Hamiltonian, and $F:\mathbb{T}^d \times \mathcal{P}(\mathbb{T}^d) \to \mathbb{R}$ is a nonlocal coupling term between the Hamilton-Jacobi and Fokker-Planck equations. Above, $\mathbb{T}^d$ is the d-dimensional flat torus, and $\mathcal{P}(\mathbb{T}^d)$ is the space of Borel probability measures on $\mathbb{T}^d$. Next, $U\in C^2(\mathbb{T}^d)$ and $M\in L^\infty(\mathbb{T}^d)\cap \mathcal{P}(\mathbb{T}^d)$ (with a slight abuse of notation we identify the absolutely continuous measures with their densities) are terminal-initial conditions for $u$ and $m$, respectively.

In \eqref{eq:maingeneral}, $u$ represents the value function of a generic agent from a continuum population of players, whereas $m$ represents the density of this population. Each agent aims at solving the optimization problem
\begin{equation}\label{eq:generic_optim}
u(x,t)=\inf_{\gamma \in H^1([t,1]),\gamma(t)=x} \int_t^1 L(\gamma(s),\dot{\gamma}(s))+F(\gamma(s),m(\cdot,s))ds+U(\gamma(1)),
\end{equation}
where $L$ is the Legendre transform of $H$; that is,
\begin{equation*}
L(x,v)=\sup_p -v\cdot p - H(x,p),~(x,v)\in \mathbb{T}^d \times \mathbb{R}^d.
\end{equation*}
Hence, $U$ is a terminal cost function. Since a generic agent is small and her actions on the population distribution can be neglected, we assume that $m$ is fixed, but unknown, in \eqref{eq:generic_optim}. Consequently, $u$ must solve a Hamilton-Jacobi equation; that is, the first PDE in \eqref{eq:maingeneral} with terminal data $U$.

Furthermore, given $u$, optimal trajectories of agents are determined by
\begin{equation*}
\dot{\gamma}(s)=-\nabla_p H(\gamma(s),\nabla u(\gamma(s),s)).
\end{equation*}
Therefore, $m$, being the population density, must satisfy the Fokker-Planck equation; that is, the second PDE in \eqref{eq:maingeneral} with initial data $M$. Hence, $M$ is the population density at time $t=0$. 

The existence, uniqueness and stability theories for \eqref{eq:maingeneral} are well understood \cite{LL07,CardaNotes,carda'13}. Here, we are specifically interested in approximation methods for the solutions of \eqref{eq:maingeneral} that can be useful for numerical solution and modeling purposes.

Currently, there are number of efficient approximation methods for solutions of MFG systems. We refer to \cite{achdou'13,achdcamdolcetta'13,achporretta'16,achddolcetta'10} for finite-difference schemes, \cite{carsilva'14,carsilva'15} for semi-Lagrangian methods, \cite{bencar'15,andreev'17,bencarsan'17,bencarmarnen'18,silva'18} for convex optimization techniques, \cite{alfego'17,GJS2} for monotone flows, and \cite{osher'18b} for infinite-dimensional Hamilton-Jacobi equations. Although general, the aforementioned methods are particularly advantageous when $F$ in \eqref{eq:maingeneral} depends locally on $m$. The reason is that local $F$ yield analytic pointwise formulas for infinite-dimensional operators involved in the algorithms. Instead, nonlocal $F$ do not yield such formulas. Additionally, fixed-grid methods suffer from dimensionality issues. Also, the number of inter-nodal couplings grows significantly for nonlocal $F$ which leads to an increased complexity of such schemes. Hence, we are interested in developing approximation methods that specifically suit nonlocal $F$ and are grid-free.

Our approach is based on a Fourier approximation of $F$ and is inspired by the methods in \cite{nurbe'18}. Here, we use the classical trigonometric polynomials as an approximation basis. Nevertheless, our method is flexible and allows more general bases. For instance, one may consider \eqref{eq:maingeneral} on different domains and boundary conditions and choose a basis accordingly.

Additionally, our approach yields a mesh-free numerical approximation of $u$ and $m$. More precisely, we directly recover the optimal trajectories of the agents rather than the values of $u$ and $m$ on a given mesh. In particular, our methods may blend well with recently developed ideas for fast and curse-of-the-dimensionality-resistant solution approach for first-order Hamilton-Jacobi equations \cite{osher'17,osher'18,Yegorov2018}. Hence, our techniques may lead to numerical schemes for nonlocal MFG that are efficient in high dimensions.

To avoid technicalities, we consider a linear $F$. More precisely, we assume that
\begin{equation*}%\label{eq:F_lin}
	F(x,m)=\int_{\mathbb{T}^d} K(x,y) m(y,t)dy,~x\in \mathbb{T}^d,~m\in \mathcal{P}(\mathbb{T}^d),
\end{equation*}
where $K\in C^2(\mathbb{T}^d \times \mathbb{T}^d)$. Thus, here we deal with the system
\begin{equation} \label{eq:main}
\begin{cases}
-\partial_t u + H(x,\nabla u) = \int_{\mathbb{T}^d} K(x,y) m(y,t) dy, \\
\partial_t m - \mathrm{div}(m \nabla_p H(x,\nabla u))=0,~(x,t) \in \mathbb{T}^d \times [0,1],\\
m(x,0)=M(x),~u(x,1)=U(x),~x\in \mathbb{T}^d.
\end{cases}
\end{equation}	

Our basic idea is to show that when $K$ is a generalized polynomial in a given basis then \eqref{eq:main} is equivalent to a fixed point problem, in a space of continuous curves, that has nice structural properties. In particular, when $K$ is symmetric and positive semi-definite, \eqref{eq:main} is equivalent to a convex optimization problem in the space of continuous curves.

Furthermore, we discuss how to construct generalized polynomial kernels that approximate a given $K$. Additionally, we observe that for translation invariant $K$ the approximating kernels have a particularly simple structure. Consequently, for such $K$ the aforementioned optimization problem is much simpler to solve.

The paper is organized as follows. In Section \ref{sec:prelim}, we present standing assumptions and some preliminary results. In Section \ref{sec:optim}, we prove the equivalence of \eqref{eq:main} to a fixed point problem over the space of continuous curves when $K$ is a generalized polynomial. Next, in Section \ref{sec:approx}, we discuss approximation methods for a general kernel. Furthermore, in Section \ref{sec:a_numerical_method}, we construct a discretization for the optimization problem from Section \ref{sec:optim} and devise a variant of a primal dual hybrid gradient algorithm for the discrete problem. Finally, in Section \ref{sec:numerical_examples}, we study several numerical examples.

\section{Assumptions and preliminary results}\label{sec:prelim}

We denote by $\mathbb{T}^d$ the $d$-dimensional flat torus. Furthermore, throughout the paper, we assume that $H \in C^2(\mathbb{T}^d \times \mathbb{R}^d)$, and
\begin{equation}\label{eq:H_hyp}
\begin{split}
\frac{1}{C}I_d \leq& \nabla^2_{pp} H(x,p)\leq C I_d,~\forall (x,p)\in \mathbb{T}^d \times \mathbb{R}^d,\\ 
-C(1+|p|^2) \leq& \nabla_x H(x,p) \cdot p,~\forall (x,p)\in \mathbb{T}^d \times \mathbb{R}^d,
\end{split}
\end{equation}
for some constant $C>0$. Next, we assume that $M\in L^\infty(\mathbb{T}^d)\cap \mathcal{P}(\mathbb{T}^d),~U\in C^2(\mathbb{T}^d),~K\in C^2(\mathbb{T}^d \times \mathbb{T}^d)$, and
\begin{equation}\label{eq:MUbounds}
\|M\|_{L^\infty(\mathbb{T}^d)},~\|U\|_{C^2(\mathbb{T}^d)},~\|K\|_{C^2(\mathbb{T}^d\times \mathbb{T}^d)} \leq C.
\end{equation}
Additionally, we suppose that $K$ is positive semi-definite; that is,
\begin{equation}\label{eq:K_mon}
\int_{\mathbb{T}^d\times \mathbb{T}^d} K(x,y)f(x)f(y)dxdy \geq 0,~\forall f\in  L^{\infty}(\mathbb{T}^d).
\end{equation}
We call $K$ symmetric if
\begin{equation}\label{eq:K_sym}
K(x,y)=K(y,x),~\forall x,y \in \mathbb{T}^d.
\end{equation}
Next, we denote by $\mathcal{P}(\mathbb{T}^d)$ the space of Borel probability measures on $\mathbb{T}^d$. We equip $\mathcal{P}(\mathbb{T}^d)$ with the Monge-Kantorovich distance that is given by
\begin{equation}\label{eq:MK}
\|m_2-m_1\|_{MK} =\sup \left\{\int_{\mathbb{T}^d}\phi(x) (m_2(x)-m_1(x))dx~\mbox{s.t.}~\|\phi\|_{\mathrm{Lip}}\leq 1\right\}.
\end{equation}

In the rest of this section, we present some preliminary results and formulas. For the optimal control and related Hamilton-Jacobi equations theory we refer to \cite{fleson'93,bardidolcetta'97}. We begin by the definition of a solution for \eqref{eq:main}.
\begin{definition}
A pair $(u,m)$ is a solution of \eqref{eq:main} if $u\in W^{1,\infty}(\mathbb{T}^d\times [0,1])$ is a viscosity solution of
\begin{equation}\label{eq:HJinMFG}
\begin{cases}
-\partial_t u+H(x,\nabla u)= \int_{\mathbb{T}^d} K(x,y) m(y,t) dy,~(x,t)\in \mathbb{T}^d\times[0,1],\\
u(x,1)=U(x),~x\in \mathbb{T}^d,
\end{cases}
\end{equation}
and $m\in L^\infty(\mathbb{T}^d\times [0,1])\cap C\left([0,1];\mathcal{P}(\mathbb{T}^d)\right)$ is a distributional solution of
\begin{equation}\label{eq:FPinMFG}
\begin{cases}
\partial_t m-\mathrm{div}(m \nabla_p H(x,\nabla u))= 0,~(x,t)\in \mathbb{T}^d\times[0,1],\\
m(x,0)=M(x),~x\in \mathbb{T}^d.
\end{cases}
\end{equation}
\end{definition}

The following theorem \cite{LL07,CardaNotes,carda'13} asserts that \eqref{eq:main} is well-posed.
\begin{theorem}\label{thm:MFGwellposed}
\begin{itemize}
\item[i.] Under assumptions \eqref{eq:H_hyp} and \eqref{eq:MUbounds}, system \eqref{eq:main} admits a solution $(u,m)$. Moreover, there exists a constant $C_1(C)>0$ such that
\begin{equation}\label{eq:um_bounds}
\nabla^2_{xx} u,~\|u\|_{W^{1,\infty}},~\|m\|_{L^\infty} \leq C_1,
\end{equation}
for any solution $(u,m)$. Additionally, if \eqref{eq:K_mon} holds then $(u,m)$ is unique.

\item[ii.] Solutions of \eqref{eq:main} are stable with respect to variations of $U,M$ and $K$ in respective norms. Particularly, suppose that $\{K_r\}_{r=1}^\infty \subset C^2(\mathbb{T}^d\times \mathbb{T}^d)$ is such that
\begin{equation}\label{eq:K-K_rC2}
\lim\limits_{r\to \infty}\|K-K_r\|_{C^2(\mathbb{T}^d\times \mathbb{T}^d)}=0,
\end{equation}
and $\{(u_r,m_r)\}_{r=1}^\infty$ are solutions of \eqref{eq:main} corresponding to kernel $K_r$. Then, the sequence $\{(u_r,m_r)\}_{r=1}^\infty$ is precompact in $C(\mathbb{T}^d\times [0,1])\times C\left([0,1];\mathcal{P}(\mathbb{T}^d)\right)$ with all accumulation points being solutions of \eqref{eq:main}. Consequently, if \eqref{eq:K_mon} holds then
\begin{equation}
\begin{split}
\lim\limits_{r\to \infty}u_r(x,t)=u(x,t),~\mbox{uniformly in}~(x,t)\in \mathbb{T}^d \times [0,1],\\
\lim\limits_{r\to \infty} \|m_{r}(\cdot,t)-m(\cdot,t)\|_{MK}=0,~\mbox{uniformly in}~t\in [0,1],
\end{split}
\end{equation}
where $(u,m)$ is the unique solution of \eqref{eq:main}.
\end{itemize}
\end{theorem}

Next, consider an arbitrary basis of smooth functions
\begin{equation}\label{eq:basis}
\Phi=\{\phi_1,\phi_2,\cdots,\phi_r\} \subset C^2(\mathbb{T}^d).
\end{equation}
For $a=(a_1,a_2,\cdots,a_r)\in C\left([0,1];\mathbb{R}^r\right)$ we denote by $u_a$ the viscosity solution of
\begin{equation}\label{eq:u_a}
\begin{cases}
-\partial_t u(x,t)+H(x,\nabla u(x,t))=\sum\limits_{i=1}^r a_i(t)\phi_i(x),~(x,t)\in \mathbb{T}^d\times[0,1]\\
u(x,1)=U(x),~x\in \mathbb{T}^d.
\end{cases}
\end{equation}
From the optimal control theory, we have that
\begin{equation}\label{eq:ua_rep}
u_a(x,t)=\inf_{\gamma \in H^1([t,1]),\gamma(t)=x} \int_t^1 \left(L\left(\gamma(s),\dot{\gamma}(s) \right)+\sum_{i=1}^r a_i(s) \phi_i(\gamma(s))\right) ds+U(\gamma(1)),
\end{equation}
for all $(x,t)\in \mathbb{T}^d\times [0,1]$, where
\begin{equation}\label{eq:L}
L(x,v)=\sup_{p\in \mathbb{R}^d} -v \cdot p -H(x,p).
\end{equation}
Moreover, for all $(x,t)\in \mathbb{T}^d\times [0,1]$ there exists $\gamma_{x,t,a}\in C^2([t,1];\mathbb{T}^d)$ such that
\begin{equation}\label{eq:ua_rep_min}
	u_{a}(x,t)=\int_t^1 \left(L\left(\gamma_{x,t,a}(s),\dot{\gamma}_{x,t,a}(s) \right)+\sum_{i=1}^r a_i(s) \phi_i(\gamma_{x,t,a}(s))\right) ds+U(\gamma_{x,t,a}(1)),
\end{equation}
and
\begin{equation}\label{eq:EL}
\begin{split}
&\frac{d}{ds}\nabla_v L\left(\gamma_{x,t,a}(s),\dot{\gamma}_{x,t,a}(s) \right) \\
=& \nabla_x L\left(\gamma_{x,t,a}(s),\dot{\gamma}_{x,t,a}(s) \right)+\sum_{i=1}^r a_i(t)\nabla \phi_i(\gamma_{x,t,a}(s)),~s\in [t,1].
\end{split}
\end{equation}
Additionally,
\begin{equation}\label{eq:u_agrad}
\begin{split}
-\nabla_v L\left(x,\dot{\gamma}_{x,t,a}(t) \right) \in& \nabla^{+}_x u_a(x,t),\\
-\nabla_v L\left(\gamma_{x,t,a}(s),\dot{\gamma}_{x,t,a}(s) \right) = &\nabla_x u_a(\gamma_{x,t,a}(s),s),~s\in (t,1],\\
-\dot{\gamma}_{x,t,a}(s)  = &\nabla_p H(\gamma_{x,t,a}(s),\nabla_x u_a(\gamma_{x,t,a}(s),s)),~s\in (t,1].
\end{split}
\end{equation}
In fact, this previous equation is also sufficient for \eqref{eq:ua_rep_min} to hold. For lighter notation, we denote $\gamma_{x,0,a}$ by $\gamma_{x,a}$.

In general, $u_a$ is not everywhere differentiable. Nevertheless, $u_a$ is semiconcave and hence $\nabla^+ u_a(x,t) \neq \emptyset$ for all $(x,t)$, and $\nabla^+ u_a(x,t)=\{\nabla u_a(x,t)\}$ for a.e. $(x,t)$. In fact, points $(x,t)$ where $u_a$ is not differentiable are precisely those for which \eqref{eq:ua_rep} admits multiple minimizers. Thus, at points $x\in \mathbb{T}^d$ where $u_a(x,0)$ is not differentiable we choose $\gamma_{x,a}$ in such a way that the map $(x,t)\mapsto \gamma_{x,a}(t)$ is Borel measurable.

Furthermore, we denote by $m_a$ the distributional solution of
\begin{equation}\label{eq:m_a}
\begin{cases}
\partial_t m- \mathrm{div}\left(m \nabla_p H(x,\nabla u_a)\right)=0,~(x,t)\in \mathbb{T}^d\times[0,1],\\
m(x,0)=M(x),~x\in \mathbb{T}^d.
\end{cases}
\end{equation}
One can show that $m_a$ is given by the push-forward of the measure $M$ by the map $\gamma_{\cdot,a}(t)$; that is,
\begin{equation}\label{eq:m_a_rep}
m_a(\cdot,t)=\gamma_{\cdot,a}(t)\sharp M.
\end{equation}
We equip $C([0,1];\mathbb{R}^r)$ with the $L^\infty$ norm
\begin{equation*}
\|a\|_{\infty}=\max_{i}\sup_{t\in [0,1]} |a_i(t)|.
\end{equation*}
Then, one has that
\begin{equation}\label{eq:m_a_stability}
\lim\limits_{n\to \infty} \|m_{a_n}(\cdot,t)-m_a(\cdot,t)\|_{MK}=0,~\mbox{uniformly in}~t\in [0,1],
\end{equation}
if $\lim\limits_{n\to \infty}\|a_n-a\|_{\infty}=0$. For a detailed discussion on $m_a$ see Chapter 4 in \cite{CardaNotes}.

Finally, we denote by
\begin{equation}\label{eq:G}
G(a)=\int_{\mathbb{T}^d} u_a(x,0)M(x)dx,~a\in C\left([0,1];\mathbb{R}^r\right).
\end{equation}

Our first theorem addresses the properties of $G$.
\begin{theorem}\label{thm:G}
The functional $a\mapsto G(a)$ is concave and everywhere Fr\'{e}chet differentiable. Moreover,
\begin{equation}\label{eq:dGa}
\partial_{a_i} G=\int_{\mathbb{T}^d} \phi_i(x) m_a(x,\cdot) dx,~1\leq i \leq r.
\end{equation}
\end{theorem}
\begin{proof}
We denote by
\begin{equation*}
p(a)=\left(\int_{\mathbb{T}^d} \phi_i(x) m_a(x,\cdot) dx\right)_{i=1}^r,~a\in C([0,1];\mathbb{R}^r).
\end{equation*}
We prove that for every $a\in C([0,1];\mathbb{R}^r)$
\begin{equation*}
0\geq G(b)-G(a)-(b-a)\cdot p(a) \geq o\left(\|b-a\|_\infty\right).
\end{equation*}
We have that
\begin{equation*}
\begin{split}
&G(b)-G(a)-(b-a)\cdot p(a)\\
=&\int_{\mathbb{T}^d}M(x)dx\int_0^1 \left(L\left(\gamma_{x,b}(t),\dot{\gamma}_{x,b}(t) \right)+\sum_{i=1}^r b_i(t) \phi_i(\gamma_{x,b}(t))\right) dt+U(\gamma_{x,b}(1))\\
&-\int_{\mathbb{T}^d}M(x)dx\int_0^1 \left(L\left(\gamma_{x,a}(t),\dot{\gamma}_{x,a}(t) \right)+\sum_{i=1}^r a_i(t) \phi_i(\gamma_{x,a}(t))\right) dt+U(\gamma_{x,a}(1))\\
&-\sum_{i=1}^r \int_0^1  (b_i(t)-a_i(t))dt\int_{\mathbb{T}^d} \phi_i(x)m_a(x,t)dx.
\end{split}
\end{equation*}
From \eqref{eq:m_a_rep} we have that
\begin{equation*}
\int_{\mathbb{T}^d} \phi_i(x)m_a(x,t)dx=\int_{\mathbb{T}^d} \phi_i(\gamma_{x,a}(t)) M(x)dx,~t\in [0,1],~1\leq i \leq r.
\end{equation*}
Hence,
\begin{equation*}
\begin{split}
&G(b)-G(a)-(b-a)\cdot p(a)\\
=&\int_{\mathbb{T}^d}M(x)dx\int_0^1 L\left(\gamma_{x,b}(t),\dot{\gamma}_{x,b}(t) \right)-L\left(\gamma_{x,a}(t),\dot{\gamma}_{x,a}(t) \right)dt\\
&+\int_{\mathbb{T}^d}M(x)dx\int_0^1\sum_{i=1}^r b_i(t)(\phi_i(\gamma_{x,b}(t))-\phi_i(\gamma_{x,a}(t))) dt \\
&+\int_{\mathbb{T}^d}M(x)\left(U(\gamma_{x,b}(1))-U(\gamma_{x,a}(1))\right)dx.
\end{split}
\end{equation*}
By definition, we have that
\begin{equation*}
\begin{split}
& \int_0^1 L\left(\gamma_{x,b}(t),\dot{\gamma}_{x,b}(t) \right)+\sum_{i=1}^r b_i(t)\phi_i(\gamma_{x,b}(t)) dt+U(\gamma_{x,b}(1))\\
\leq & \int_0^1 L\left(\gamma_{x,a}(t),\dot{\gamma}_{x,a}(t) \right)+\sum_{i=1}^r b_i(t)\phi_i(\gamma_{x,a}(t)) dt+U(\gamma_{x,a}(1)),~\forall x\in \mathbb{T}^d.
\end{split}
\end{equation*}
Hence,
\begin{equation*}
G(b)-G(a)-(b-a)\cdot p(a) \leq 0,~\forall a,b \in C([0,1];\mathbb{T}^d).
\end{equation*}
This previous inequality yields the concavity of $G$. On the other hand, we have that
\begin{equation*}
\begin{split}
&G(b)-G(a)-(b-a)\cdot p(a)\\
=&\int_{\mathbb{T}^d}M(x)dx\int_0^1 L\left(\gamma_{x,b}(t),\dot{\gamma}_{x,b}(t) \right)-L\left(\gamma_{x,a}(t),\dot{\gamma}_{x,a}(t) \right)dt\\
&+\int_{\mathbb{T}^d}M(x)dx\int_0^1\sum_{i=1}^r a_i(t)(\phi_i(\gamma_{x,b}(t))-\phi_i(\gamma_{x,a}(t))) dt \\
&+\int_{\mathbb{T}^d}M(x)\left(U(\gamma_{x,b}(1))-U(\gamma_{x,a}(1))\right)dx\\
&+\int_{\mathbb{T}^d}M(x)dx\int_0^1\sum_{i=1}^r (b_i(t)-a_i(t))(\phi_i(\gamma_{x,b}(t))-\phi_i(\gamma_{x,a}(t))) dt.
\end{split}
\end{equation*}
Therefore, again by the definition of $\gamma_{x,a}$ and $\gamma_{x,b}$, we have that
\begin{equation*}
\begin{split}
&G(b)-G(a)-(b-a)\cdot p(a)\\
\geq&\int_{\mathbb{T}^d}M(x)dx\int_0^1\sum_{i=1}^r (b_i(t)-a_i(t))(\phi_i(\gamma_{x,b}(t))-\phi_i(\gamma_{x,a}(t))) dt\\
\geq& -\|b-a\|_{\infty}\sum_{i=1}^r \int_0^1 \left|\int_{\mathbb{T}^d}\phi_i(\gamma_{x,b}(t))M(x)dx-\int_{\mathbb{T}^d}\phi_i(\gamma_{x,a}(t))M(x)dx\right|dt\\
=& -\|b-a\|_{\infty}\sum_{i=1}^r \int_0^1 \left|\int_{\mathbb{T}^d}\phi_i(x)m_b(x,t)dx-\int_{\mathbb{T}^d}\phi_i(x)m_a(x,t)dx\right| dt\\
\geq & -\|b-a\|_{\infty}\sum_{i=1}^r \mathrm{Lip}(\phi_i)\int_0^1 \|m_b(\cdot,t)-m_a(\cdot,t)\|_{MK}dt.
\end{split}
\end{equation*}
Hence, by \eqref{eq:m_a_stability} the proof is complete.
\end{proof}

\section{The optimization problem}\label{sec:optim}

In this section, we assume that $K$ is a generalized polynomial in the basis $\Phi$; that is,
\begin{equation}\label{eq:K_poly}
K(x,y)=\sum_{i,j=1}^{r} k_{ij} \phi_i(x)\phi_j(y),~x,y \in \mathbb{T}^d.
\end{equation}
where $\mathbf{K}=(k_{ij})_{i,j=1}^r\in M_{r,r}(\mathbb{R})$ is a matrix of coefficients. For such $K$, \eqref{eq:main} takes form
\begin{equation} \label{eq:main_r}
\begin{cases}
-\partial_t u + H(x,\nabla u) = \sum\limits_{i=1}^r \phi_i(x) \sum\limits_{j=1}^r k_{ij} \int_{\mathbb{T}^d} \phi_j(y) m(y,t) dy, \\
\partial_t m - \mathrm{div}(m \nabla_p H(x,\nabla u))=0,~(x,t) \in \mathbb{T}^d \times [0,1],\\
m(x,0)=M(x),~u(x,1)=U(x),~x\in \mathbb{T}^d.
\end{cases}
\end{equation}

Our main observation is the following theorem.
\begin{theorem}\label{thm:equivalence}
\begin{itemize}
\item[i.] A pair $(u,m)$ is a solution of \eqref{eq:main_r} if and only if $(u,m)=(u_{a^*},m_{a^*})$ for some $a^*\in C\left([0,1];\mathbb{R}^r\right)$ such that
\begin{equation}\label{eq:a_fixedpoint}
a^*=\mathbf{K} \partial_a G(a^*).
\end{equation}

\item[ii.] If $\mathbf{K}$ is positive-definite then \eqref{eq:a_fixedpoint} is equivalent to finding a $0$ of a monotone operator $a\mapsto \mathbf{K}^{-1} a - \partial_a G(a),~a\in C\left([0,1];\mathbb{R}^r\right)$.

\item[iii.] Additionally, if $\mathbf{K}$ is symmetric, \eqref{eq:a_fixedpoint} is equivalent to the convex optimization problem
\begin{equation}\label{eq:a_equation_optimization}
\begin{split}
&\inf_{a\in C\left([0,1];\mathbb{R}^r\right)}  \frac{1}{2}\langle \mathbf{K}^{-1} a, a \rangle - G(a)\\
=&\inf_{a\in C\left([0,1];\mathbb{R}^r\right)}  \frac{1}{2}\langle \mathbf{K}^{-1} a, a \rangle - \int_{\mathbb{T}^d}u_a(x,0)M(x)dx.
\end{split}
\end{equation}
\end{itemize}
\end{theorem}
\begin{proof}
Items \textit{ii} and \textit{iii} follow immediately from \textit{i} by the concavity of $G$. Thus, we just prove \textit{i}.

By Theorem \ref{thm:MFGwellposed} \eqref{eq:main_r} admits a solution $(u,m)$. Furthermore, define $a^*$ as
\begin{equation}\label{eq:a*}
a_i^*(t)= \sum_{j=1}^r k_{ij} \int_{\mathbb{T}^d} \phi_j(y) m(y,t) dy,~t \in [0,1].
\end{equation}
Then $a^*\in C\left([0,1];\mathbb{R}^r\right)$, and by the definition of $u_a$ and $m_a$ we have that $(u,m)=(u_{a^*},m_{a^*})$. Hence, by Theorem \ref{thm:G}, we have that
\begin{equation*}
\partial_{a_i} G(a^*)=\int_{\mathbb{T}^d} \phi_i(x) m(x,\cdot) dx,~1\leq i \leq r.
\end{equation*}
Consequently, from \eqref{eq:a*} obtain
\begin{equation*}
a_i^*= \sum_{j=1}^r k_{ij} \partial_{a_j} G(a^*).
\end{equation*}
\end{proof}

\begin{remark}
The optimization problem \eqref{eq:a_equation_optimization} is equivalent to the optimal control of Hamilton-Jacobi PDE pointed out in \cite{LL07} (equations (58)-(59) in Section 2.6). One can think of \eqref{eq:a_equation_optimization} as (58)-(59) of \cite{LL07} written in Fourier coordinates.
\end{remark}

\section{Approximating the kernel}\label{sec:approx}

In this section, we show that one can construct suitable approximations for an arbitrary $K$. We begin by a simple lemma.
\begin{lemma}\label{lma:mon_posdef_equivalence}
Suppose that $K$ is given by \eqref{eq:K_poly}. Then $K$ is positive semi-definite if and only if $\mathbf{K}=(k_{ij})_{ij,=1}^r$ is positive semi-definite.
\end{lemma}
\begin{proof}
Fix an arbitrary $(\xi_i)_{i=1}^r \in \mathbb{R}^r$. Then there exists a unique $(\lambda_i)_{i=1}^r \in \mathbb{R}^r$ such that
\begin{equation*}
\xi_i=\sum_{j=1}^r \lambda_j \int_{\mathbb{T}^d} \phi_i(x)\phi_j(x)dx,~1\leq i \leq r,
\end{equation*}
because $\{\phi_i\}$ are linearly independent. Therefore, for
\begin{equation*}
f=\sum_{j=1}^r \lambda_j \phi_j 
\end{equation*}
we have that
\begin{equation*}
\xi_i = \int_{\mathbb{T}^d} f(x)\phi_i(x)dx,~1\leq i \leq r.
\end{equation*}
Hence, 
\begin{equation*}
\int_{\mathbb{T}^d\times \mathbb{T}^d} K(x,y) f(x) f(y)dxdy=\sum_{i,j=1}^r k_{ij} \xi_i \xi_j,
\end{equation*}
that yields the proof.
\end{proof}
Now, we fix our basis to be the trigonometric one:
\begin{equation}\label{eq:phi_k_trig}
\phi_{\alpha}(x)=e^{2i \pi \alpha \cdot x},~x\in \mathbb{T}^d,~\alpha\in \mathbb{Z}^d.
\end{equation}
\begin{remark}
Unlike in \eqref{eq:basis}, here it is more practical to use multi-dimensional indexes to enumerate the trigonometric functions in higher dimensions. Additionally, it is more economical in terms of notation to use the complex-valued trigonometric functions. Nevertheless, our discussion is always about real valued $K$, and the reader can think of the end results as expansions in terms of $\{\cos(2\pi \alpha \cdot x),\sin(2\pi \alpha\cdot x)\}_{\alpha\in \mathbb{Z}^d}$.
\end{remark}
For $\alpha=(\alpha_1,\alpha_2,\cdots,\alpha_d)\in \mathbb{Z}^d$, we denote by
\begin{equation*}
|\alpha|=(|\alpha_1|,|\alpha_2|,\cdots,|\alpha_d|),
\end{equation*}
and for $\alpha,r\in \mathbb{Z}^d$
\begin{equation*}
\alpha \leq r \iff \alpha_j \leq r_j,~1\leq j \leq d.
\end{equation*}
For $r_1,r_2 \in \mathbb{N}_0^d$ we denote by
\begin{equation*}
K_{r_1 r_2}(x,y)=\sum_{|\alpha|\leq r_1,|\beta| \leq r_2} \hat{K}_{\alpha \beta} e^{2i \pi (\alpha \cdot x+\beta \cdot y)},~x,y\in \mathbb{T}^d,
\end{equation*}
where
\begin{equation*}
\hat{K}_{\alpha \beta}=\int_{\mathbb{T}^d} K(x,y) e^{-2i \pi (\alpha \cdot x+\beta \cdot y)}dxdy,~\alpha,\beta \in \mathbb{Z}^d.
\end{equation*}
Furthermore, for $r_1,r_2\in \mathbb{N}_0^d$ we denote by
\begin{equation*}
\Sigma_{r_1 r_2}(x,y)=\frac{1}{\prod_{j=1}^d(1+r_{1j})(1+r_{2j})}\sum_{|\alpha|\leq r_1,|\beta| \leq r_2} K_{r_1 r_2}(x,y),~x,y\in \mathbb{T}^d.
\end{equation*}
\begin{remark}
The function $K_{r_1 r_2}$ is the rectangular partial Fourier sum of $K$. Correspondingly, $\Sigma_{r_1 r_2}$ is the rectangular Fej\'{e}r average of $K$. Additionally, if $K$ is real valued then $K_{r_1 r_2}$ and $\Sigma_{r_1 r_2}$ are real valued for any $r_1,r_2 \in \mathbb{N}_0^d$.
\end{remark}

\begin{proposition}\label{prp:K_approx}
If $K$ is positive semi-definite (symmetric) then, $K_{r r}$ and $\Sigma_{r r}$ are also positive semi-definite (symmetric) for all $r \in \mathbb{N}_0^d$. Moreover,
\begin{equation}\label{eq:Fejerconverges}
\lim\limits_{\min_j r_j \to \infty} \|\Sigma_{r r}-K\|_{C^2(\mathbb{T}^d \times \mathbb{T}^d)}=0,
\end{equation}
Additionally, if $K\in C^3(\mathbb{T}^d \times \mathbb{T}^d)$ then
\begin{equation}\label{eq:partialconverges}
\lim\limits_{\min_j r_j \to \infty} \|K_{r r}-K\|_{C^2(\mathbb{T}^d \times \mathbb{T}^d)}=0.
\end{equation}
\end{proposition}
\begin{proof}
The convergence properties \eqref{eq:Fejerconverges}, \eqref{eq:partialconverges} are classical results in Fourier analysis. Thus, we will just prove that $K_{r r}$ and $\Sigma_{r r}$ are positive semi-definite (symmetric). For that, we use the representation formulas
\begin{equation*}
\begin{split}
K_{rr}(x,y)=&\int_{\mathbb{T}^d\times \mathbb{T}^d} K(z,w) D_{r r}(x-z,y-w)dzdw,\\
\Sigma_{rr}(x,y)=&\int_{\mathbb{T}^d\times \mathbb{T}^d} K(z,w) F_{r r}(x-z,y-w)dzdw,~x,y\in \mathbb{T}^d,\\
\end{split}
\end{equation*}
where $D_{rr}$ and $F_{rr}$ are, respectively, the $2d$-dimensional rectangular Dirichlet and Fej\'{e}r kernels. A crucial feature of $D_{rr}$ and $F_{rr}$ is that they are symmetric and decompose into lower dimensional kernels:
\begin{equation*}
D_{rr}(z,w)=D_{r}(z)D_{r}(w),~ F_{rr}(z,w)=F_{r}(z)F_{r}(w),~z,w\in \mathbb{T}^d,
\end{equation*}
where $D_r$ and $F_r$ are the corresponding $d$-dimensional kernels. In particular, $K_{rr},\Sigma_{rr}$ are symmetric if $K$ is such. Furthermore, for an arbitrary $f \in L^{\infty}(\mathbb{T}^d)$ we have that
\begin{equation*}
\begin{split}
&\int_{\mathbb{T}^d\times \mathbb{T}^d} K_{rr}(x,y)f(x)f(y)dxdy\\
=&\int_{\mathbb{T}^d\times \mathbb{T}^d} K(z,w)dzdw \int_{\mathbb{T}^d\times \mathbb{T}^d} f(x)f(y)D_{r r}(x-z,y-w)dxdy\\
=&\int_{\mathbb{T}^d\times \mathbb{T}^d} K(z,w)dzdw \int_{\mathbb{T}^d\times \mathbb{T}^d} f(x)f(y)D_r(x-z) D_r(y-w)dxdy\\
=&\int_{\mathbb{T}^d\times \mathbb{T}^d} K(z,w) f_r(z)f_r(w)dzdw \geq 0.
\end{split}
\end{equation*}
Thus, $K_{rr}$ is positive semi-definite if $K$ is such. The proof for $\Sigma_{r r}$ is identical.
\end{proof}
\begin{remark}
By Proposition \ref{prp:K_approx}, kernels $K_{rr},\Sigma_{rr}$ are positive semi-definite. Therefore, their coefficients matrices with respect to basis $\{\cos(2\pi \alpha\cdot x),\sin(2\pi \alpha\cdot x)\}$ are also positive semi-definite by Lemma \ref{lma:mon_posdef_equivalence}. Nevertheless, to take full advantage of Theorem \ref{thm:equivalence} one would need these matrices to be positive definite (invertible). To solve this problem one can add $\eps I$ regularization term, where $I$ is the identity matrix of the suitable dimension and $\eps>0$ is a small constant. However, as discussed below, this regularization is not necessary for translation invariant kernels.
\end{remark}

Suppose that
\begin{equation*}%\label{eq:K_trans_invariant}
K(x,y)=\eta(x-y),~x,y \in \mathbb{T}^d,
\end{equation*}
where $\eta$ is a periodic function. Then, we have
\begin{equation}\label{eq:K_cos}
\begin{split}
&\int_{\mathbb{T}^d} K(x,y) \cos (2\pi \alpha \cdot y) dy\\
=& \int_{\mathbb{T}^d} \eta(x-y) \cos (2\pi \alpha\cdot y) dy= \int_{\mathbb{T}^d} \eta(y) \cos (2\pi \alpha\cdot (x-y)) dy\\
=&\cos(2\pi \alpha\cdot x) \int_{\mathbb{T}^d} \eta(y) \cos (2\pi \alpha\cdot y) dy+ \sin(2\pi \alpha\cdot x) \int_{\mathbb{T}^d} \eta(y) \sin (2\pi \alpha\cdot y) dy.
\end{split}
\end{equation}
Similarly, we obtain that
\begin{equation}\label{eq:K_sin}
\begin{split}
&\int_{\mathbb{T}^d} K(x,y) \sin (2\pi \alpha\cdot y) dy\\
=& \int_{\mathbb{T}^d} \eta(x-y) \sin (2\pi \alpha\cdot y) dy= \int_{\mathbb{T}^d} \eta(y) \sin (2\pi \alpha\cdot (x-y)) dy\\
=&\sin(2\pi \alpha\cdot x) \int_{\mathbb{T}^d} \eta(y) \cos (2\pi \alpha\cdot y) dy- \cos(2\pi \alpha\cdot x) \int_{\mathbb{T}^d} \eta(y) \sin (2\pi \alpha\cdot y) dy.
\end{split}
\end{equation}
Therefore, we have that
\begin{equation*}
\begin{split}
\int_{\mathbb{T}^d}K(x,y) \cos(2\pi \alpha \cdot x) \cos(2\pi \alpha \cdot y) dx dy=& \int_{\mathbb{T}^d} \eta(y) \cos (2\pi \alpha\cdot y) dy,\\
\int_{\mathbb{T}^d}K(x,y) \sin(2\pi \alpha \cdot x) \cos(2\pi \alpha \cdot y) dx dy=& \int_{\mathbb{T}^d} \eta(y) \sin (2\pi \alpha\cdot y) dy,\\
\int_{\mathbb{T}^d}K(x,y) \cos(2\pi \alpha \cdot x) \sin(2\pi \alpha \cdot y) dx dy=& -\int_{\mathbb{T}^d} \eta(y) \sin (2\pi \alpha\cdot y) dy,\\
\int_{\mathbb{T}^d}K(x,y) \sin(2\pi \alpha \cdot x) \sin(2\pi \alpha \cdot y) dx dy=& \int_{\mathbb{T}^d} \eta(y) \cos (2\pi \alpha\cdot y) dy.
\end{split}
\end{equation*}
Hence, the coefficients matrices of partial Fourier sums (and their linear combinations) of $K$ consist of $2\times 2$ blocks that correspond to expansion terms with a frequency $\alpha \in \mathbb{Z}^d$; that is,
\begin{equation}\label{eq:Delta_alpha}
\Delta_\alpha=\begin{pmatrix}
\int_{\mathbb{T}^d} \eta(y) \cos (2\pi \alpha\cdot y) dy & \int_{\mathbb{T}^d} \eta(y) \sin (2\pi \alpha\cdot y) dy\\
-\int_{\mathbb{T}^d} \eta(y) \sin (2\pi \alpha\cdot y) dy & \int_{\mathbb{T}^d} \eta(y) \cos (2\pi \alpha\cdot y) dy
\end{pmatrix}.
\end{equation}
Thus, the coefficient matrix will be degenerate if $\det(\Delta_\alpha)=0$ for some $\alpha$. But we have that
\begin{equation*}
\det(\Delta_\alpha)= \left(\int_{\mathbb{T}^d} \eta(y) \cos (2\pi \alpha\cdot y) dy\right)^2 + \left(\int_{\mathbb{T}^d} \eta(y) \sin (2\pi \alpha\cdot y) dy\right)^2.
\end{equation*}
Hence, $\det(\Delta_\alpha)=0$ if and only if $\Delta_\alpha=0$ or, equivalently, there are no expansion terms with frequency $\alpha$. But then, we can simply ignore these terms in our basis and obtain a non-degenerate matrix.

Moreover, to invert the coefficients matrix one just has to invert the $2\times 2$ blocks. Additionally, if $K$ is symmetric; that is, $\eta(y)=\eta(-y)$, we get that
\begin{equation*}
\int_{\mathbb{T}^d} \eta(y) \sin (2\pi \alpha\cdot y) dy=0,~\forall \alpha \in \mathbb{Z}^d.
\end{equation*}
Hence, the coefficient matrices are simply diagonal. Therefore, we have proved the following proposition.
\begin{proposition}
If $K$ is translation invariant then all partial Fourier sums of $K$ and their linear combinations, such as $K_{rr}$ and $\Sigma_{rr}$, contain only $\cos (2\pi \alpha \cdot x) \cos (2\pi \alpha \cdot y), \cos (2\pi \alpha \cdot x) \sin (2\pi \alpha \cdot y), \sin (2\pi \alpha \cdot x) \cos (2\pi \alpha \cdot y), \sin (2\pi \alpha \cdot x) \sin (2\pi \alpha \cdot y)$ expansion terms. Therefore, coefficient matrices of such approximations with respect to trigonometric basis consist of $2\times 2$ blocks that are multiples of $\Delta_\alpha$ in \eqref{eq:Delta_alpha}. If, additionally, $K$ is symmetric these coefficient matrices are diagonal.
\end{proposition}

\begin{remark}
In general, if $\{\phi_1,\phi_2,\cdots,\phi_r,\cdots\}$ is an orthonormal basis consisting of eigenfunctions of Hilbert-Schmidt integral operator $f(\cdot) \mapsto \int_{\mathbb{T}^d} K(\cdot,y) f(y)dy$; that is,
\begin{equation*}
\int_{\mathbb{T}^d} K(x,y) \phi_\alpha(y)dx=\lambda_\alpha \phi_\alpha(x),~x\in \mathbb{T}^d,~\alpha \in \mathbb{N},
\end{equation*}
for some $\{\lambda_\alpha\} \subset \mathbb{R}$. Then, one has that
\begin{equation*}
k_{\alpha \beta}= \int_{\mathbb{T}^d} K(x,y)\phi_\alpha(x)\phi_\beta(y)dxdy= \lambda_\beta \delta_{\alpha \beta}.
\end{equation*}
Consequently, for arbitrary $I \subset \mathbb{N}\times \mathbb{N}$ we have that
\begin{equation*}
K_{I}(x,y)=\sum_{(\alpha,\beta) \in I} k_{\alpha \beta} \phi_{\alpha}(x) \phi_{\beta}(y)=\sum_{(\alpha,\alpha) \in I} \lambda_{\alpha} \phi_{\alpha}(x) \phi_{\alpha}(y).
\end{equation*}
Therefore, all partial Fourier sums of $K$ in basis $\{\phi_\alpha(x)\phi_\beta(y)\}$ contain only terms $\phi_\alpha(x) \phi_\alpha(y)$ and yield diagonal coefficient matrices consisting of corresponding eigenvalues of the Hilbert-Schmidt integral operator.

In general, it is not easy to calculate the eigenfunctions of a given Hilbert-Schmidt integral operator. Nevertheless, as we saw above, for translation invariant symmetric periodic $K$ these eigenfunctions are precisely the trigonometric functions. 
\end{remark}

\section{A numerical method} \label{sec:a_numerical_method}

In this section we propose a numerical method to solve \eqref{eq:main} for a symmetric and positive semi-definite $K$. We assume that an approximation $K_r$ of the form \eqref{eq:K_poly} is already constructed with a symmetric and positive definite $\mathbf{K}$. Thus, we devise an algorithm for the solution of \eqref{eq:main_r}.

By Theorem \ref{thm:equivalence} we have that \eqref{eq:main_r} is equivalent to \eqref{eq:a_equation_optimization}. Therefore, in what follows, we present a suitable discretization of \eqref{eq:a_equation_optimization}. We rewrite latter as
\begin{equation}\label{eq:optim_S}
\inf_{a\in C\left([0,1];\mathbb{R}^r\right)} S(a),
\end{equation}
where
\begin{equation*}
S(a) = \frac{1}{2} \langle \mathbf{J} a, a \rangle - G(a),
\end{equation*}
and $\mathbf{J} = \mathbf{K}^{-1}$.

\subsection{Discretization of the $u_a$}\label{sub:discretization_of_the_lax_formula}

We start with the discretization of $u_a$. For that, we discretize the representation formula \eqref{eq:ua_rep}. We can rewrite latter as
\begin{equation}\label{eq:u_form}
u_a(x,0) = \inf_{\bu} \int_0^1 L_a(\bx(s), \bu(s),s) ds + U(\bx(1)),
\end{equation}
where $\bx$ satisfies the following controlled ODE 
\begin{equation} \label{eq:ODE}
\dot{\bx}(s) = \bu(s), \quad \bx(0) = x,~s\in [0,1].
\end{equation}
Recall that
\begin{equation*}
L_a(x,u,s)=L(x,u)+\sum_{k=1}^r a_k(s)\phi_k(x),~(x,u,s)\in \mathbb{T}^d\times \mathbb{R}^d \times[0,1].
\end{equation*}
We choose a uniform discretization of the time interval:
\[
0=s_0<s_1<s_2<\ldots<s_N=1,
\]
with a step size $h_t = \frac{1}{N}$, hence $s_i=i h_t=\frac{i}{N}$. We denote the values of $\bx$ and $\bu$ at time $s_i$ by $\bx(s_i)=x_i$, $\bu(s_i)=u_i$. Using a backward Euler discretization of \eqref{eq:ODE} we have
\begin{equation*}
u_i = 	 \frac{x_{i} - x_{i-1}}{h_t},~i \in \{1,\ldots, N\}. 
\end{equation*}
Discretizing the integral \eqref{eq:u_form} with a right point quadrature rule and using the above discretization we get
\begin{equation} \label{eq:udisc}
\begin{cases}
\quad [u_a](x,0)&= \inf_{\{x_i\}_0^N} h_t \sum_{i=1}^{N}L_a\left(x_i,\frac{x_{i}-x_{i-1}}{h_t},s_i\right)+U(x_N),\\
\mbox{subject to:}& x_0=x.
\end{cases}
\end{equation}

\subsection{Discretization of $G$}\label{sub:discretization_of_the_functional_g}

We start by discretizing the initial measure $M$ using a convex combination of Dirac $\delta$ distributions. Denoting the discretized measure $[M]$, we have
\begin{equation*}
[M]=\sum\limits_{\alpha=1}^{Q} c_{\alpha} \delta_{y_{\alpha}}
\end{equation*}
or, in the distributional sense,
\begin{equation}\label{eq:Mdiscrete}
\int\limits_{\Tt^d} \psi(y) d[M](y)=\sum\limits_{\alpha=1}^{Q} c_{\alpha} \psi(y_{\alpha}),\quad \psi \in C(\Tt^d),
\end{equation}
for some $\{y_{\alpha}\}_{\alpha=1}^Q \subset \Tt^d$ and $\{c_{\alpha}\geq 0\}_{\alpha=1}^Q$ such that $\sum_{\alpha=1}^Q c_{\alpha}=1$.  Then, $G$ is discretized as follows
\begin{equation}\label{eq:[G]}
[G](a)=\sum\limits_{\alpha=1}^{Q} c_{\alpha}~[u_a](y_{\alpha},0).
\end{equation}

\subsection{Discretization of $S$}\label{sub:discretization_of_s}
Now, we discretize \eqref{eq:optim_S}. We first discretize $a_k$-s by taking their values at times $s_i$, that we denote by:
\begin{equation*}
[a]_k=(a_k(s_0),\cdots, a_k(s_N))=(a_{k0},\cdots, a_{kN}),~ k=1,2,\cdots,r.
\end{equation*}
Recall that
\begin{equation*}
\langle {\bf J} a, a\rangle = \sum\limits_{k,l=1}^{r} {\bf J}_{kl} \int\limits_{0}^{1} a_k(s) a_l(s) ds.
\end{equation*}
We discretize this previous quadratic form by a simple right point quadrature rule.
\begin{equation*}
[\langle {\bf J} a, a\rangle] = h_t \sum\limits_{k,l=1}^{r} {\bf J}_{kl} \sum\limits_{i=1}^{N} a_{ki}a_{li}.
\end{equation*}			
So the discretization of $S$ is
\begin{equation} \label{eq:S}
\begin{split}
[S](a)&= \frac{h_t}{2} \sum\limits_{k,l=1}^{r} {\bf J}_{kl} \sum\limits_{i=1}^{N} a_{ki}a_{li}-[G](a)\\
&= \frac{h_t}{2} \sum\limits_{k,l=1}^{r} {\bf J}_{kl} \sum\limits_{i=1}^{N} a_{ki}a_{li}- \sum\limits_{\alpha=1}^{Q} c_{\alpha}~[u_a](y_{\alpha},0),
\end{split}
\end{equation}
where we used \eqref{eq:[G]}. Therefore, the discretization of \eqref{eq:optim_S} is
\begin{equation}\label{eq:infsupDirectB_gen}
\begin{split}
&\inf_{\{a_{ki}\}}[S](a)\\
=&\inf_{\{a_{ki}\}}\sup_{\{x_{\alpha i}:~x_{\alpha 0}=y_{\alpha}\}} \frac{h_t}{2}\sum_{k,l=1}^r\mathbf{J}_{kl}\sum_{i=1}^{N}a_{ki}a_{li}-h_t\sum_{\alpha=1}^Q\sum_{i=1}^{N}c_{\alpha}L\left(x_{\alpha i},\frac{x_{\alpha i}-x_{\alpha (i-1)}}{h_t}\right)\\
&-h_t \sum_{\alpha=1}^Q \sum_{i=1}^{N}\sum_{k=1}^r c_{\alpha}a_{ki} \phi_k(x_{\alpha i})-\sum_{\alpha=1}^Q c_{\alpha} U(x_{\alpha N}).
\end{split}
\end{equation}

\subsection{Primal-dual hybrid-gradient method}\label{sub:primal_dual_hybrid_gradient_method}

Now, we specify the Lagrangian to be quadratic and devise a primal-dual hybrid-gradient algorithm \cite{chapock'11} to solve \eqref{eq:optim_S}. More precisely, we assume that
\begin{equation*}
L(x,u)=\frac{|u|^2}{2},~(x,u)\in \mathbb{T}^d\times \mathbb{R}^d,
\end{equation*}
and therefore %\eqref{eq:udisc} becomes
%\begin{equation}\label{eq:mindiscretebackward}
%\begin{cases}
%\quad [u_a](x,0)  & = \inf_{\{x_i\}_0^N}  \sum_{i=1}^{N}\frac{|x_{i}-x_{i-1}|^2}{2h_t}+ h_t \sum_{i=1}^{N}\sum_{k=1}^r a_{ki} \phi_k(x_i)+ U(x_N),\\
%\mbox{subject to:}  & x_0=x.\\
%\end{cases}
%\end{equation}
%Consequently,
\eqref{eq:infsupDirectB_gen} becomes
\begin{equation}\label{eq:infsupDirectB}
\begin{split}
&\inf_{\{a_{ki}\}}[S](a)\\
=&\inf_{\{a_{ki}\}}\sup_{\{x_{\alpha i}~:~x_{\alpha 0}=y_{\alpha}\}} \frac{h_t}{2}\sum_{k,l=1}^r\mathbf{J}_{kl}\sum_{i=1}^{N}a_{ki}a_{li}-\sum_{\alpha=1}^Q\sum_{i=1}^{N}c_{\alpha}\frac{|x_{\alpha i}-x_{\alpha (i-1)}|^2}{2h_t}\\
&-h_t \sum_{\alpha=1}^Q \sum_{i=1}^{N}\sum_{k=1}^r c_{\alpha}a_{ki} \phi_k(x_{\alpha i})-\sum_{\alpha=1}^Q c_{\alpha} U(x_{\alpha N}).
\end{split}
\end{equation}

Now, we describe the algorithm. For each iteration time $\nu\geq 0$ we have four groups of variables: 
$a^\nu=\{a^\nu_{ki}\}_{k,i=1,1}^{r,N},~x^\nu=\{x^\nu_{\alpha i}\}_{\alpha,i=1,0}^{Q,N}$, and $z^\nu=\{z^\nu_{\alpha i}\}_{\alpha,i=1,0}^{Q,N}$. Furthermore, we fix $\lambda, \omega>0$ that are proximal step parameters for variables $a$ and $x$, respectively. Additionally, we take $0\leq \theta \leq 1$.

\textbf{Step 1.} Given $a^\nu,x^\nu,z^\nu$ the first step of the algorithm is to solve the proximal problem
\begin{equation*}
\begin{split}
&\inf_{\{a_{ki}\}} \frac{h_t}{2}\sum_{k,l=1}^r\mathbf{J}_{kl}\sum_{i=1}^{N}a_{ki}a_{li}-\sum_{\alpha=1}^Q\sum_{i=1}^{N}c_{\alpha}\frac{|z_{\alpha i}^\nu-z_{\alpha (i-1)}^\nu|^2}{2h_t}\\
&-h_t \sum_{\alpha=1}^Q \sum_{i=1}^{N}\sum_{k=1}^r c_{\alpha}a_{ki} \phi_k(z_{\alpha i}^\nu)-\sum_{\alpha=1}^Q c_{\alpha} U(z_{\alpha N}^\nu)
+\frac{1}{2\lambda} \sum_{k=1}^r\sum_{i=1}^N (a_{ki}-a^{\nu}_{ki})^2,
\end{split}
\end{equation*}
that is equivalent to
\begin{equation*}
\begin{split}
&\inf_{\{a_{ki}\}} \frac{h_t}{2}\sum_{k,l=1}^r\mathbf{J}_{kl}\sum_{i=1}^{N}a_{ki}a_{li}- h_t \sum_{\alpha=1}^Q \sum_{i=1}^{N}\sum_{k=1}^r c_{\alpha}a_{ki} \phi_k(z_{\alpha i}^\nu)+\frac{1}{2\lambda} \sum_{k=1}^r\sum_{i=1}^N (a_{ki}-a^{\nu}_{ki})^2.
\end{split}
\end{equation*}
Thus, we obtain the following update of the $a$-variable.
\begin{equation}\label{eq:a_iterationDirect}
\begin{pmatrix}
a^{\nu+1}_{1i}\\
a^{\nu+1}_{2i}\\
\vdots\\
a^{\nu+1}_{ri}
\end{pmatrix}= (\lambda h_t {\bf J}+\mathrm{Id}_r)^{-1} 
\begin{pmatrix}
a^\nu_{1i}+\lambda h_t \sum_{\alpha=1}^Q c_{\alpha} \phi_1(z^\nu_{\alpha i})\\
a^\nu_{2i}+\lambda h_t\sum_{\alpha=1}^Q c_{\alpha} \phi_2(z^\nu_{\alpha i})\\
\vdots\\
a^\nu_{ri}+\lambda h_t\sum_{\alpha=1}^Q c_{\alpha} \phi_r(z^\nu_{\alpha i})
\end{pmatrix},\quad 1\leq i \leq N.
\end{equation}
\begin{remark}\label{rem:time_par}
	Note that although the number of variables $\{a_{ki}\}_{k,i=1,1}^{r,N}$ is $r\times N$, the calculations of $\{a_{ki}\}$ for different $i$-s are mutually independent. Therefore, the only complexity is in the inversion of an $r\times r$ matrix $\lambda \sigma {\bf J}+\mathrm{Id}_r$ that can be computed beforehand and used throughout the scheme. Moreover, as seen in Section \ref{sec:approx}, translation invariant symmetric kernels yield diagonal matrices that extremely simplify the calculations.
\end{remark}

\textbf{Step 2.} Given $a^{\nu+1},x^\nu,z^\nu$ we update $x$-variable by solving the proximal problem
\begin{equation*}
\begin{split}
&\inf_{\{x_{\alpha i}:~x_{\alpha0}=y_{\alpha}\}} \sum_{\alpha=1}^Q\sum_{i=1}^{N}c_{\alpha}\frac{|x_{\alpha i}-x_{\alpha (i-1)}|^2}{2h_t}+ h_t \sum_{\alpha=1}^Q \sum_{i=1}^{N}\sum_{k=1}^r c_{\alpha}a^{\nu+1}_{ki} \phi_k(x_{\alpha i})\\ 
&+\sum_{\alpha=1}^Q c_{\alpha} U(x_{\alpha N}) +\frac{1}{2\omega} \sum_{\alpha=1}^Q\sum_{i=1}^N |x_{\alpha i}-x^{\nu}_{\alpha i}|^2.
\end{split}
\end{equation*}
Solving this previous problem may be a costly operation. Hence, we just perform a one step gradient descent. Therefore, we obtain
\begin{equation}\label{eq:x_iterationDirect}
\begin{split}
x^{\nu+1}_{\alpha 1} &= x^{\nu}_{\alpha 1}-\frac{\omega c_\alpha}{h_t}(x_{\alpha 1}-y_{\alpha})-\frac{\omega c_\alpha}{h_t}(x_{\alpha 1}-x_{\alpha 2})-\omega c_\alpha h_t \sum_{k=1}^r a^{\nu+1}_{k1}\nabla \phi_k(x_{\alpha 1}),\\
x^{\nu+1}_{\alpha i} &= x^{\nu}_{\alpha i}-\frac{\omega c_\alpha}{h_t}(x_{\alpha i}-x_{\alpha (i-1)})-\frac{\omega c_\alpha}{h_t}(x_{\alpha i}-x_{\alpha (i+1)}),\\
&-\omega c_\alpha h_t \sum_{k=1}^r a^{\nu+1}_{k i}\nabla \phi_k(x_{\alpha i}), \quad 1\leq i \leq N-1,\\
x^{\nu+1}_{\alpha N} &= x^{\nu}_{\alpha N}-\frac{\omega c_\alpha}{h_t}(x_{\alpha N}-x_{\alpha (N-1)})-\omega c_\alpha \nabla U (x_{\alpha N})-\omega c_\alpha h_t \sum_{k=1}^r a^{\nu+1}_{k N}\nabla \phi_k(x_{\alpha N}).
\end{split}
\end{equation}

\textbf{Step 3.} In the final step we update the $z$-variable by
\begin{equation}\label{eq:z_updateDirect}
z^{\nu+1}_{\alpha i}=x^{\nu+1}_{\alpha i}+\theta (x_{\alpha i}^{\nu+1}-x^\nu_{\alpha i}), \quad 1\leq \alpha \leq Q, ~1\leq i \leq N.
\end{equation}

\begin{remark}\label{rem:space_par}
	Note that the updates for $\{x_{\alpha i}\},\{z_{\alpha i}\}$ variables are mutually independent for different $\alpha$-s. Therefore, our $a$-updates are parallel in time, and $x,z$-updates are parallel in space.
\end{remark}

\begin{remark}\label{rem:PDHGnonst}
Strictly speaking, one cannot simply apply the primal-dual hybrid gradient method to \eqref{eq:infsupDirectB} because the coupling between $a$ and $x$ is not bilinear, and there is no concavity in $x$. Nevertheless, our calculations always yield solid results. Therefore, there is a natural problem of rigorously understanding the convergence properties of the aforementioned algorithm. We plan to address this problem in our future work.
\end{remark}

\section{Numerical Examples}\label{sec:numerical_examples}

In this section, we present several numerical experiments. We first look into one-dimensional case, in Section \ref{sub:1_dimensional_examples}, and after we consider the two-dimensional case, in Section \ref{sub:2_dimensional_examples}.

For our calculations, we choose the periodic Gaussian kernel that is given by
\begin{equation}\label{eq:K_smud}
K^d_{\sigma,\mu}(x,y)=\prod_{i=1}^{d} K^1_{\sigma,\mu}(x_i,y_i),~x,y\in \mathbb{T}^d,
\end{equation} 
where
\begin{equation}\label{eq:kernel1}
K^1_{\sigma,\mu}(x,y)=\frac{\mu}{\sqrt{2\pi (\frac{\sigma}{2})^2}}\sum_{k=-\infty}^{\infty}e^{-\frac{(x-y-k)^2}{2 (\frac{\sigma}{2})^2}},~x,y\in \mathbb{T},
\end{equation}
and $\sigma, \mu >0$ are given parameters. Here, $\sigma$ models how spread is the kernel. The smaller $\sigma$ the more weight agents assign to their immediate neighbors -- this translates into crowd-aversion in the close neighborhood only. Furthermore, $\mu$ is the total weight of the agents. Therefore, $\mu$ measures how sensitive is a generic agent to the total population, the bigger the more averse is the agent to others. As we observe in the numerical experiments, the less $\sigma$ and the larger $\mu$ the more separated are the agents. This phenomenon was also observed in \cite{aurelldjehiche'18}.

Throughout the section we denote by
	\begin{equation}\label{eq:phisystem}
		\phi_k(x)=
		\begin{cases}
			1,~\mbox{if}~k=1,\\
			\sqrt{2}\cos \pi (k-1) x,~\mbox{if}~k~\mbox{is odd,}~\mbox{and}~k>0,\\
			\sqrt{2}\sin \pi k x,~\mbox{if}~k~\mbox{is even},~x\in \mathbb{T}.
		\end{cases}
	\end{equation}
	Therefore, we have
	\begin{equation*}
		\{\phi_1,\phi_2,\phi_3,\cdots\}=\{1,\sqrt{2}\sin 2\pi x, \sqrt{2} \cos 2\pi x,\cdots \}.
	\end{equation*}

\subsection{One-dimensional examples}\label{sub:1_dimensional_examples}
For all simulations we use the same initial-terminal conditions
\begin{equation*}
	M(x) = \frac{1}{6} + \frac{5}{3} \sin^2 \pi x,\quad U(x) = 1+\sin \left(4\pi x+ \frac{\pi}{2}\right),~x\in \mathbb{T},
\end{equation*}
that are depicted in Figure \ref{fig:itc_1}.
		\begin{figure}[h] 
		    \centering
		    \begin{subfigure}[t]{0.5\textwidth}
		        \centering
		        \includegraphics[width=1\textwidth]{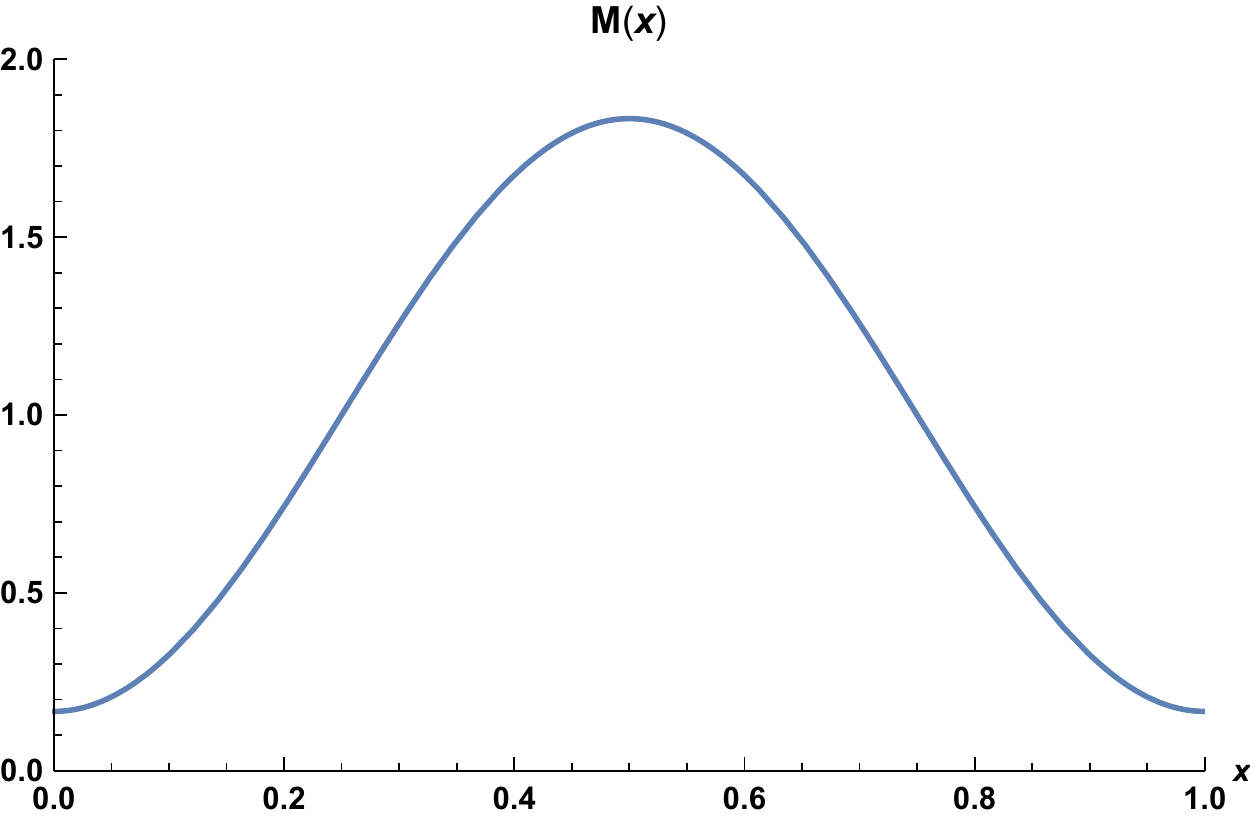}
		        \caption{Initial distribution of agents, $M(x)$.}
		    \end{subfigure}%
		    ~ 
		    \begin{subfigure}[t]{0.5\textwidth}
		        \centering
		        \includegraphics[width=1\textwidth]{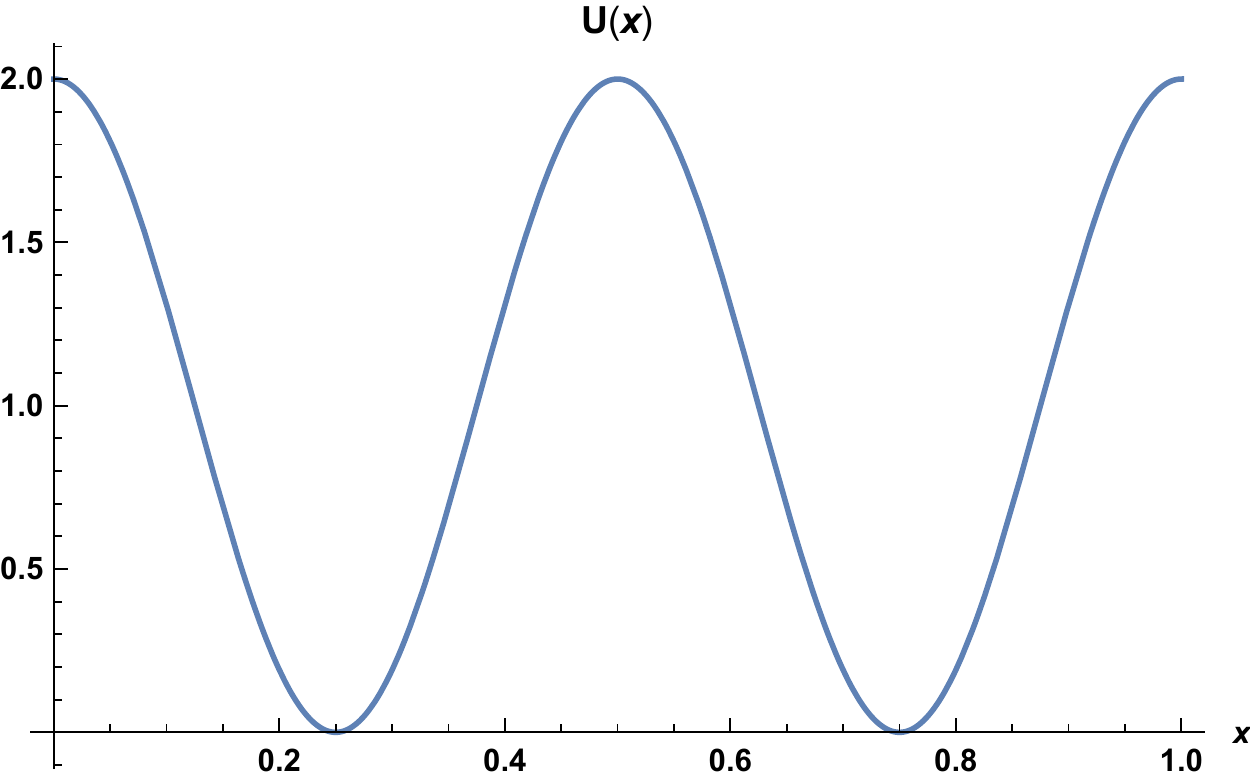}
		        \caption{Terminal cost function, $U(x)$.}
		    \end{subfigure}
			\caption{Initial-terminal conditions.}
			\label{fig:itc_1}
		\end{figure}
		We also use the same time and space discretization for all one dimensional experiments, and the same parameters for the numerical scheme. We discretize the time using a step size $\Delta t = \frac{1}{N}$. For the discretization of $M$ we use
		\begin{equation*}
		y_\alpha=\frac{\alpha}{Q+1},\quad c_{\alpha}=\frac{M(y_\alpha)}{\sum_{\beta=1}^Q M(y_\beta)} ~1\leq \alpha \leq Q.
		\end{equation*}
		We choose $N=20,~Q=50$ and use eight basis functions, $r=8$. Additionally, we set the numerical scheme parameters to $\lambda = 3,~\omega = \frac{1}{12}$ and $\theta=1$.
\begin{remark}
For the standard primal-dual hybrid gradient method, one must have $\omega \lambda <\frac{1}{A^2}$, where $A$ is the norm of the bilinear-form matrix. As we mentioned in Remark \ref{rem:PDHGnonst}, here we do not have a bilinear coupling between $a$ and $x$. Thus, we estimate $A$ by an upper bound on the $(l_2,l_2)$ Lipschitz norm of the mapping
\begin{equation*}
F_{ki}(x)=h_t \sum_{\alpha=1}^Q c_\alpha \phi_k(x_{\alpha i}),~1\leq l \leq r,~1\leq i \leq N.
\end{equation*}
More precisely, we have that
\begin{equation*}
\begin{split}
\mathrm{Lip}(F)^2=&\sup_{\{x_{\beta j}\}} \sup_{\|w_{\beta j}\|_2\leq 1} \sum_{k,i} \left(\sum_{\beta,j} \frac{\partial F_{ki}}{\partial x_{\beta j}}w_{\beta j}\right)^2\\
=&\sup_{\{x_{\beta j}\}} \sup_{\|w_{\beta j}\|_2\leq 1} \sum_{k,i} \left(\sum_{\beta} h_t c_\beta \nabla \phi_k(x_{\beta i}) w_{\beta i}\right)^2\\
\leq &h_t^2 \sup_{\{x_{\beta j}\}} \sup_{\|w_{\beta j}\|_2\leq 1} \sum_{k,i} \left( \sum_{\beta} c_\beta^2 \|\nabla \phi_k(x_{\beta i})\|_2^2 \cdot  \sum_\beta w_{\beta i}^2\right)\\
\leq & h_t^2 \sup_{\|w_{\beta j}\|_2\leq 1} \sum_{k,i} \mathrm{Lip}(\phi_k)^2 \left( \sum_{\beta} c_\beta^2  \cdot  \sum_\beta w_{\beta i}^2\right)\\
= & h_t^2 \sup_{\|w_{\beta j}\|_2\leq 1} \sum_{k} \mathrm{Lip}(\phi_k)^2  \sum_{\beta} c_\beta^2  \sum_{\beta,i} w_{\beta i}^2\\
= & h_t^2  \sum_{k} \mathrm{Lip}(\phi_k)^2  \sum_{\beta} c_\beta^2.  
\end{split}
\end{equation*}
Thus, we take
\begin{equation*}
A^2= h_t^2  \sum_{k=1}^r \mathrm{Lip}(\phi_k)^2  \sum_{\beta=1}^Q c_\beta^2.  
\end{equation*}
\end{remark}
		
The trigonometric expansion of $K^1_{\sigma,\mu}$ is given by
			\begin{equation}\label{eq:kernel1expansioncos}
				K^1_{\sigma,\mu}(x,y)=\mu \left(1+2\sum_{n=1}^\infty e^{-\frac{(\pi n \sigma)^2}{2}} \cos 2\pi n (x-y)\right),~x,y \in \mathbb{T},
			\end{equation}
or
			\begin{equation}\label{eq:kernel1expansionphi}
				K^1_{\sigma,\mu}(x,y)=\sum_{k=1}^\infty \mu e^{-\frac{1}{2}\left(\pi \sigma \left[\frac{k}{2}\right]\right)^2} \phi_k(x)\phi_k(y),~x,y \in \mathbb{T},
			\end{equation}
in our notation. Therefore, for a given $r$, the matrices ${\bf K, J}$ are given by
			\begin{equation}\label{eq:KJmats1}
				\begin{split}
					{\bf K}=& \mbox{diag} \left(\mu e^{-\frac{1}{2}\left(\pi \sigma \left[\frac{k}{2}\right]\right)^2}\right)_{k=1}^r,\\
					{\bf J}=& \mbox{diag} \left(\mu^{-1} e^{\frac{1}{2}\left(\pi \sigma \left[\frac{k}{2}\right]\right)^2}\right)_{k=1}^r.
				\end{split}
			\end{equation}
		In Figure \ref{1d_kernels} we plot the Gaussian kernels we used, for $r=8$ and different values of $\mu$ and $\sigma$.
		We see the influence of these values in Figure \ref{1D_comparison}.
		In the first column of Figure \ref{1D_comparison} we compare the results regarding for different values of $\mu$ and $\sigma$.
			
		Comparing the first and the second columns of Figure \ref{1D_comparison}, we see that the trajectories of the agents in the first column are closer than in the second one.
		This is due to the fact that $\mu = 0.5$ in the first kernel and $\mu=1.5$ in the second one, hence the second kernel (higher value of $\mu$) penalizes more high density of agents.
		Therefore, the agents spread out more before the final time when they converge to the points of low-cost near minima of the terminal cost function, $U$, see Figure \ref{fig:itc_1} (b).
		
		In the last column the value of $\sigma = 0.8$ is higher, this means that agents are indifferent to the distances between them -- they just feel the total mass. Hence, they minimize the travel distances from initial positions to low-cost locations of $U$ ignoring the population density. In fact, in this case $K^1_{\sigma,\mu}\approx \mu$, and therefore $\int_{\mathbb{T}} K^1_{\mu,\sigma}(x,y)m(y,t)dy \approx \mu$. Thus, in this case \eqref{eq:main} approximates a decoupled system of Hamilton-Jacobi and Fokker-Planck equations. But the optimal trajectories of the decoupled system are straight lines by Hopf-Lax formula. As we can see in Figure \ref{1D_comparison} (d), this fact is consistent with the straight-line trajectories that we obtain.
			\begin{figure}[h]
			    \centering
			    \begin{subfigure}[t]{0.5\textwidth}
			        \centering
			        \includegraphics[width=1\textwidth]{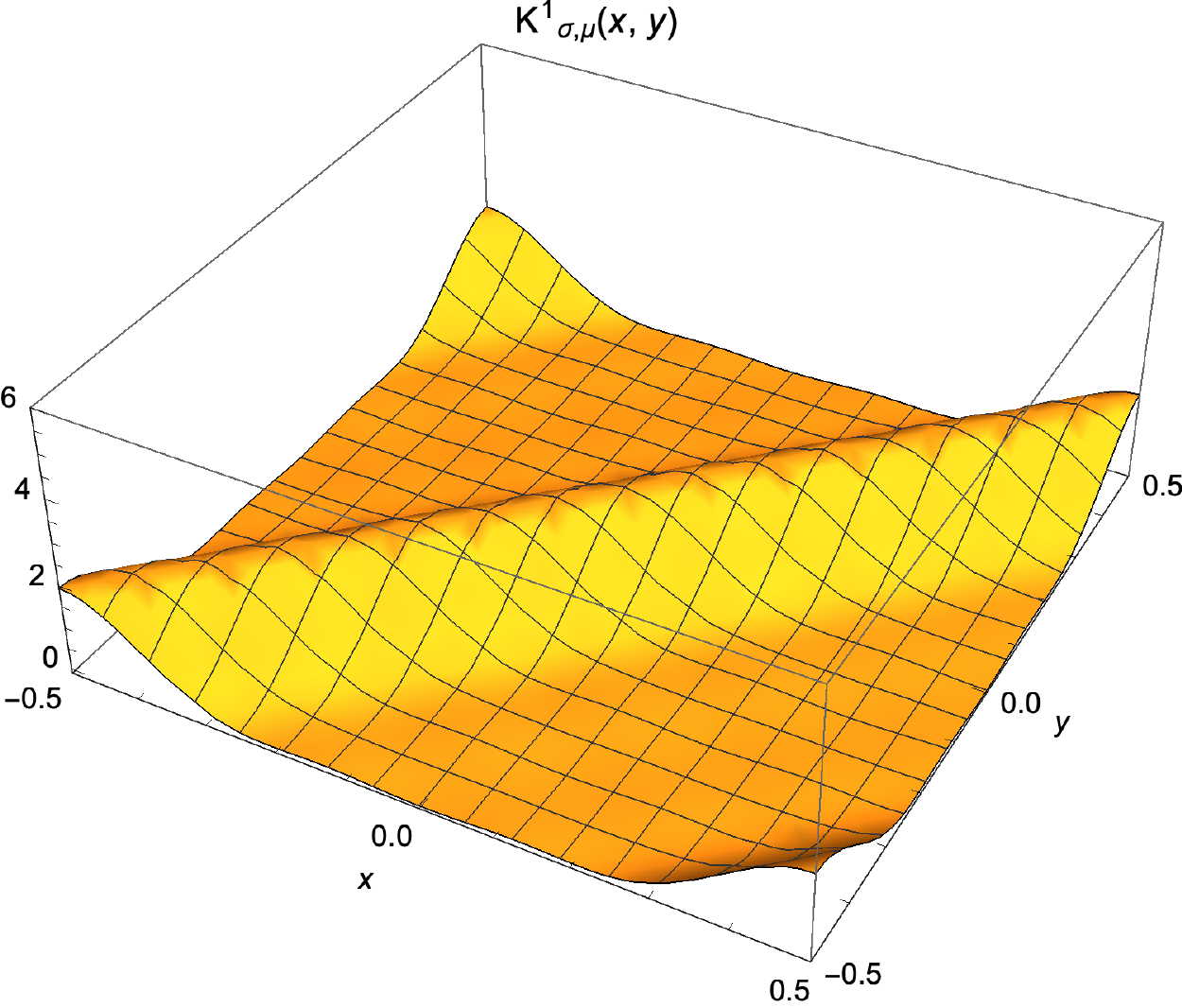}
			        \caption{Gaussian kernel, $K^1_{0.2, 0.5}(x,y)$.}
			    \end{subfigure}%
			    ~ 
			    \begin{subfigure}[t]{0.5\textwidth}
			        \centering
			        \includegraphics[width=1\textwidth]{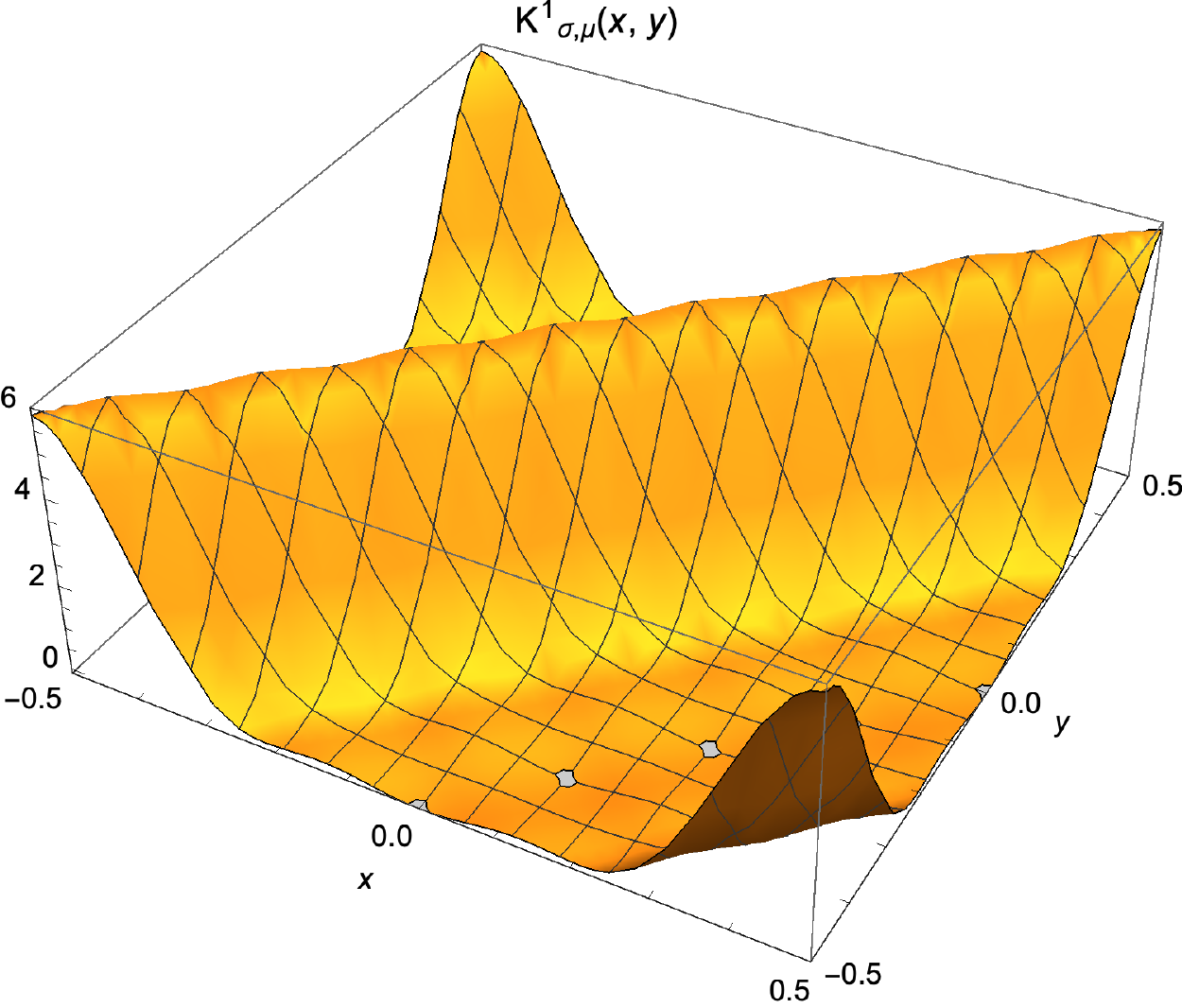}
			        \caption{Gaussian kernel, $K^1_{0.2, 1.5}(x,y)$.}
			    \end{subfigure}
		
			    \begin{subfigure}[t]{0.5\textwidth}
			        \centering
			        \includegraphics[width=1\textwidth]{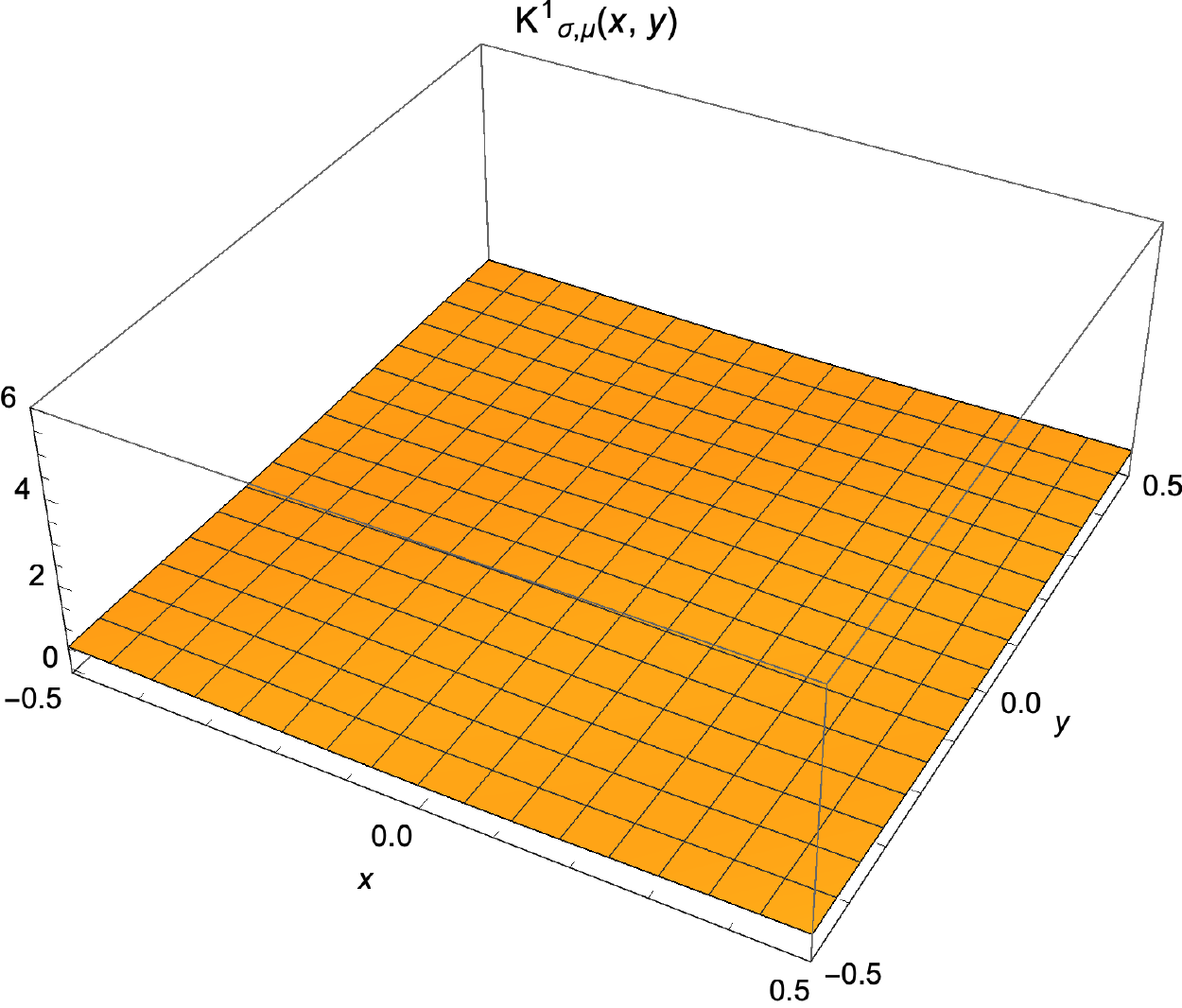}
			        \caption{Gaussian kernel, $K^1_{0.8, 0.5}(x,y)$.}
			    \end{subfigure}%
			    ~ 
			    \begin{subfigure}[t]{0.5\textwidth}
			        \centering
			        \includegraphics[width=1\textwidth]{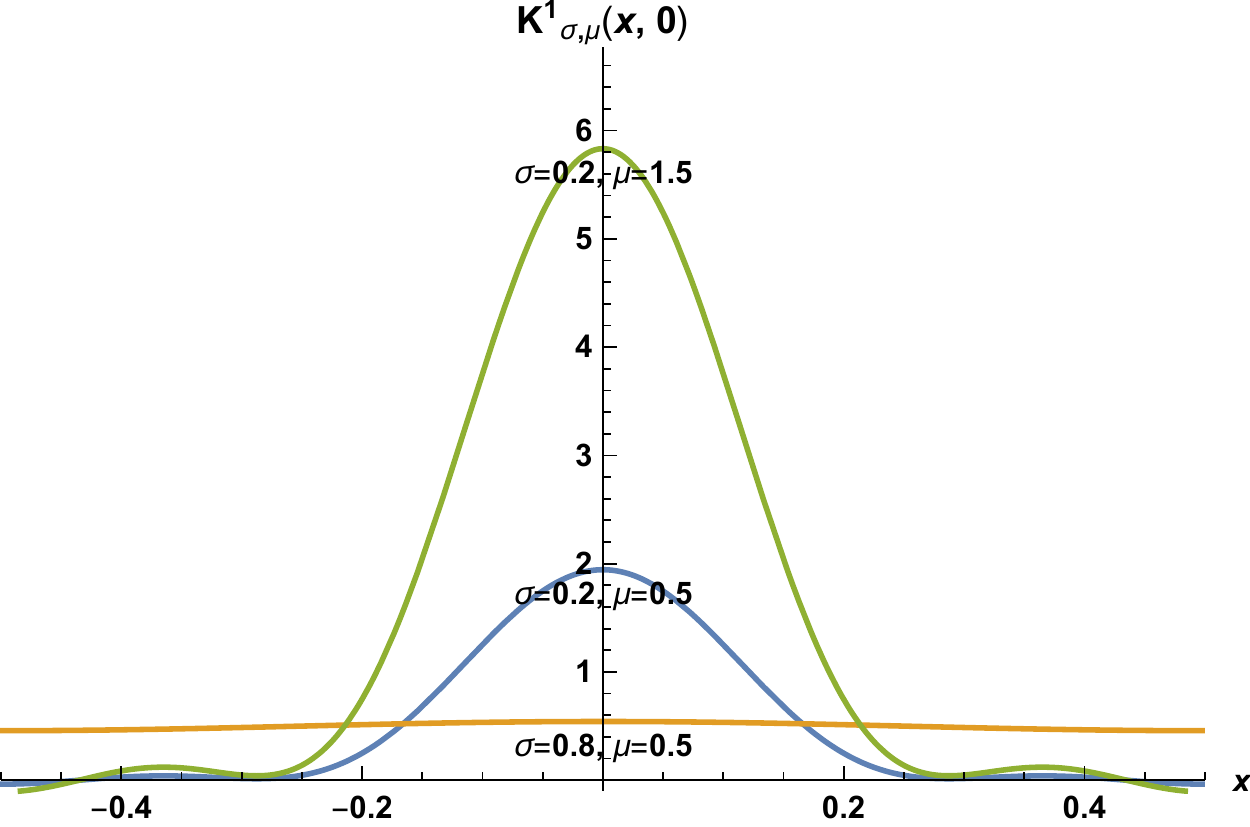}
			        \caption{Comparison of kernels on $K^1_{\sigma, \mu}(x,0)$.}
			    \end{subfigure}
				\caption{Plots of the three Gaussian kernels in (a)-(c), and a comparison of their sections in (d).}
				\label{1d_kernels}
			\end{figure}
		
			\begin{figure}[h]
			    \centering
			    \begin{subfigure}[t]{0.3\textwidth}
			        \centering
			        \includegraphics[width=1\textwidth]{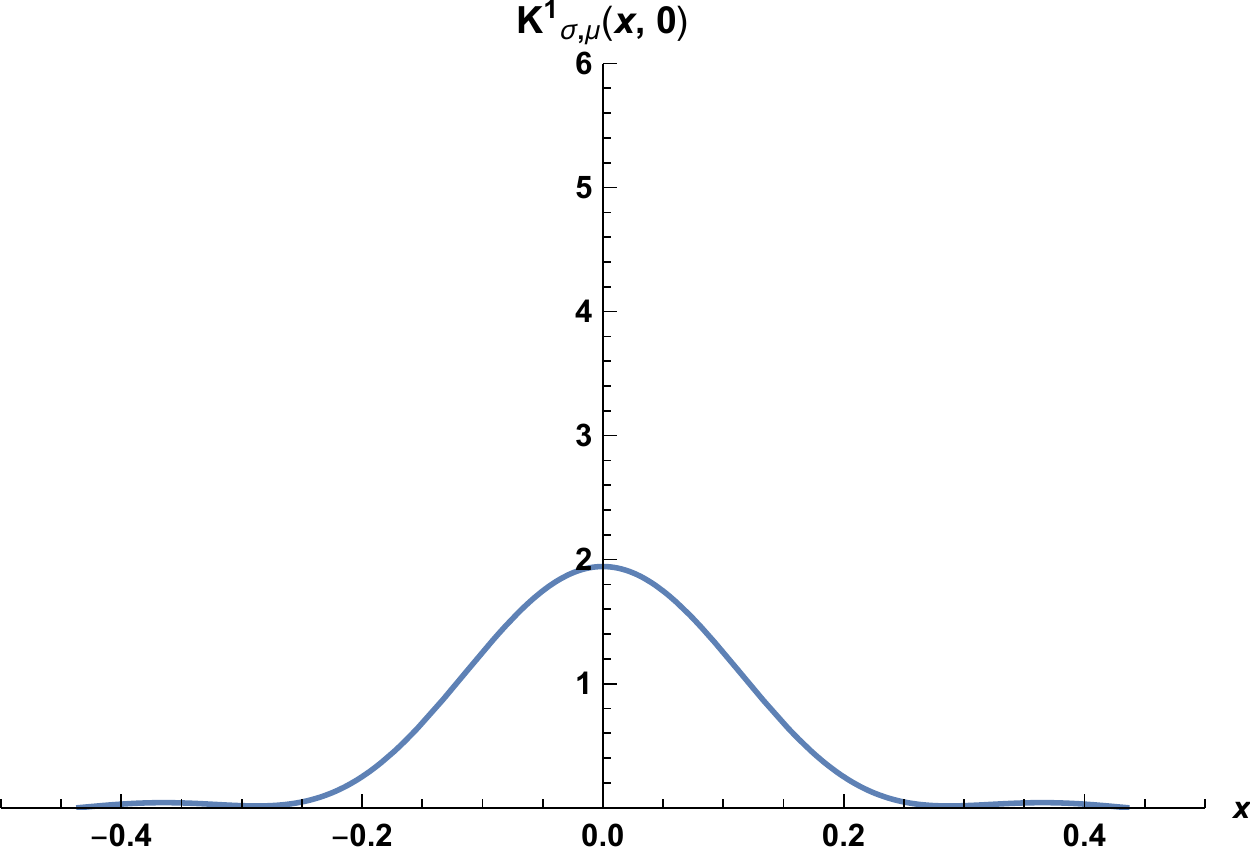}
			        \caption{$K^1_{0.2, 0.5}(x,0)$.}
			    \end{subfigure}%
			    ~
			    \begin{subfigure}[t]{0.3\textwidth}
			        \centering
			        \includegraphics[width=1\textwidth]{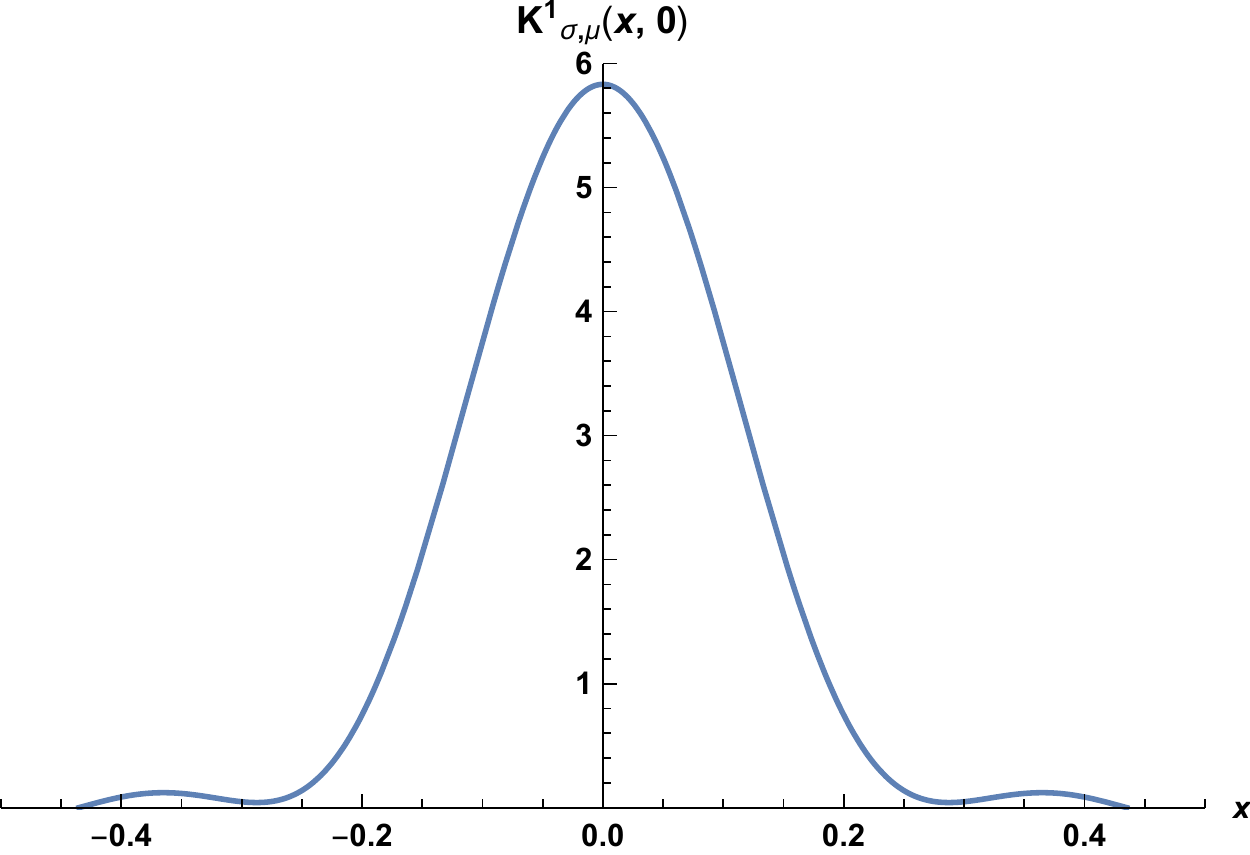}
			        \caption{$K^1_{0.2, 1.5}(x,0)$.}
			    \end{subfigure}
				 ~
			    \begin{subfigure}[t]{0.3\textwidth}
			        \centering
			        \includegraphics[width=1\textwidth]{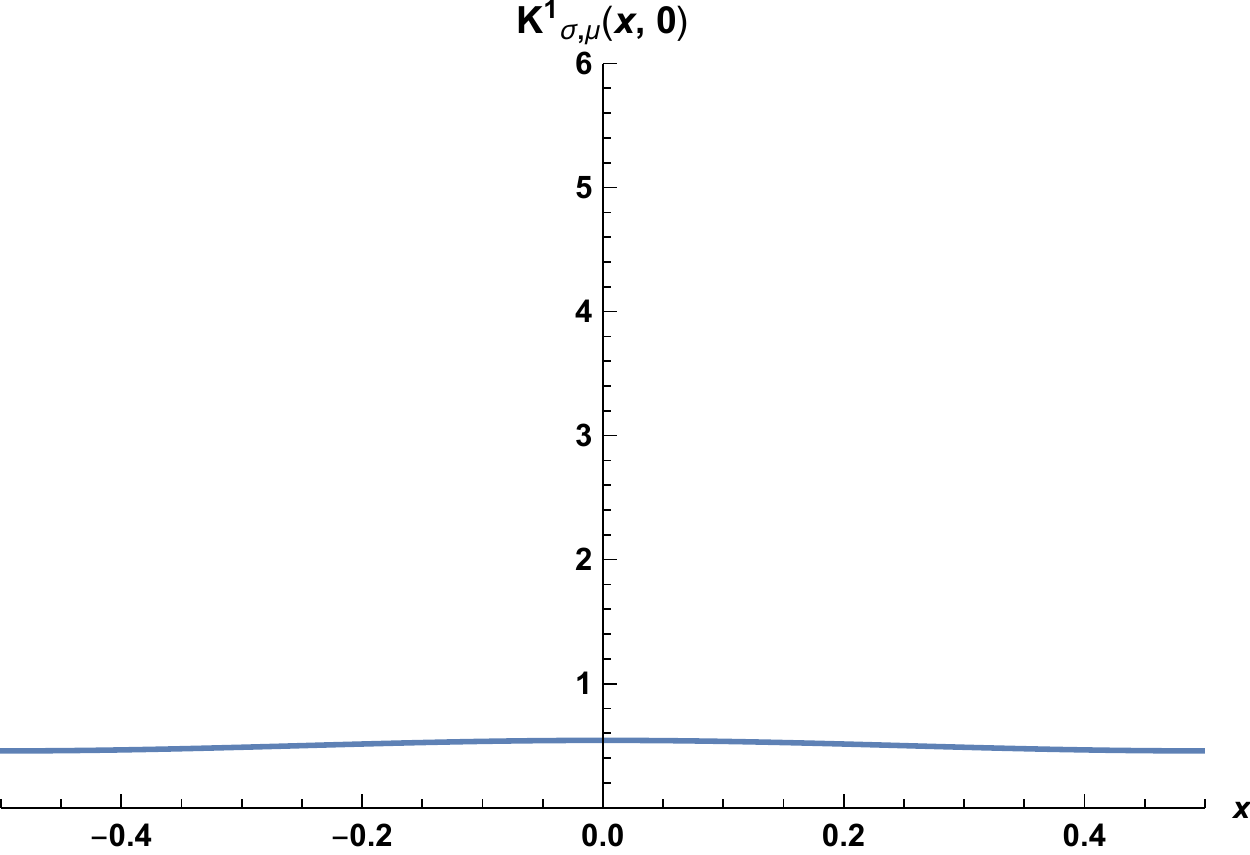}
			        \caption{$K^1_{0.8, 0.5}(x,0)$.}
			    \end{subfigure}%
		    
			    \begin{subfigure}[t]{0.3\textwidth}
			        \centering
			        \includegraphics[width=1\textwidth]{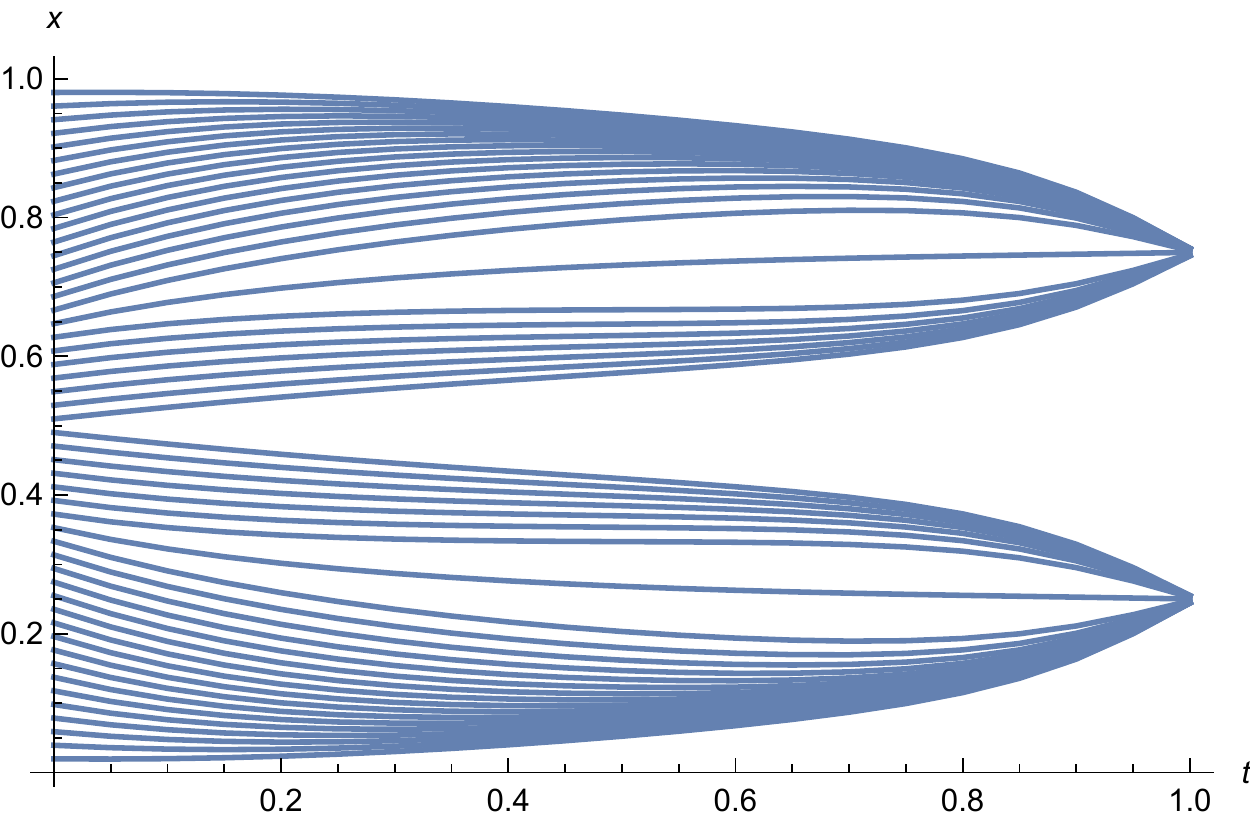}
			        %\caption{Trajectories, $\bx(t,x)$.}
			    \end{subfigure}%
			    ~
			    \begin{subfigure}[t]{0.3\textwidth}
			        \centering
			        \includegraphics[width=1\textwidth]{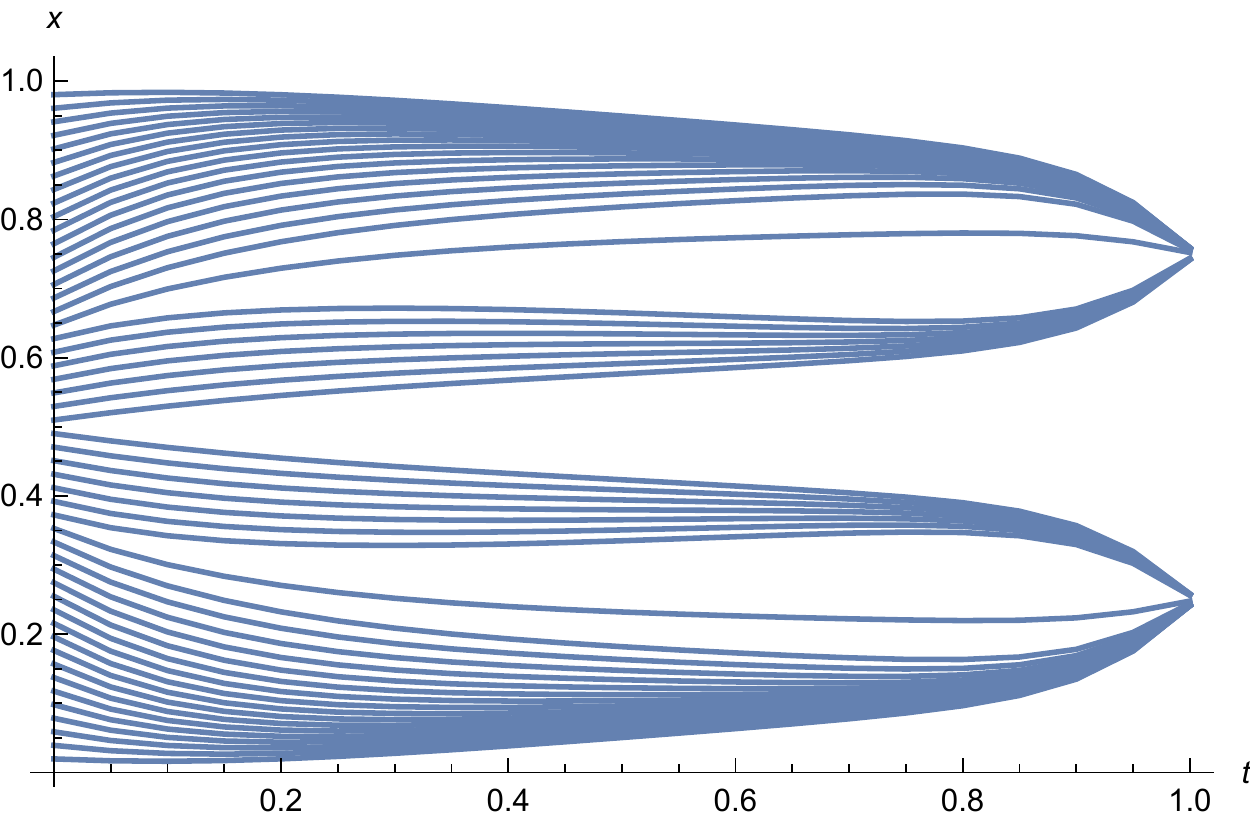}
			        \caption{Trajectories, $\bx(t,y_\alpha)$.}
			    \end{subfigure}
				 ~
			    \begin{subfigure}[t]{0.3\textwidth}
			        \centering
			        \includegraphics[width=1\textwidth]{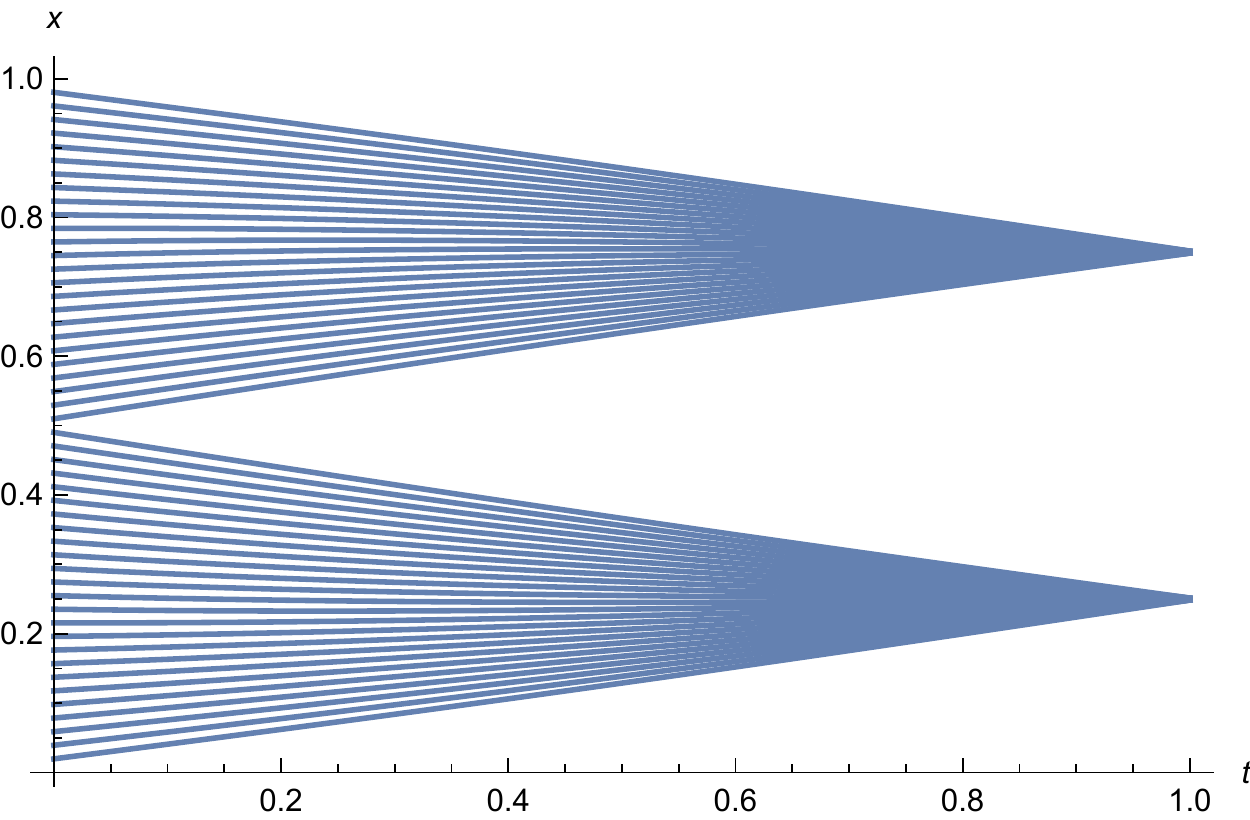}
			        %\caption{Trajectories, $\bx(t,x)$.}
			    \end{subfigure}%
			 
			    \begin{subfigure}[t]{0.3\textwidth}
			        \centering
			        \includegraphics[width=1\textwidth]{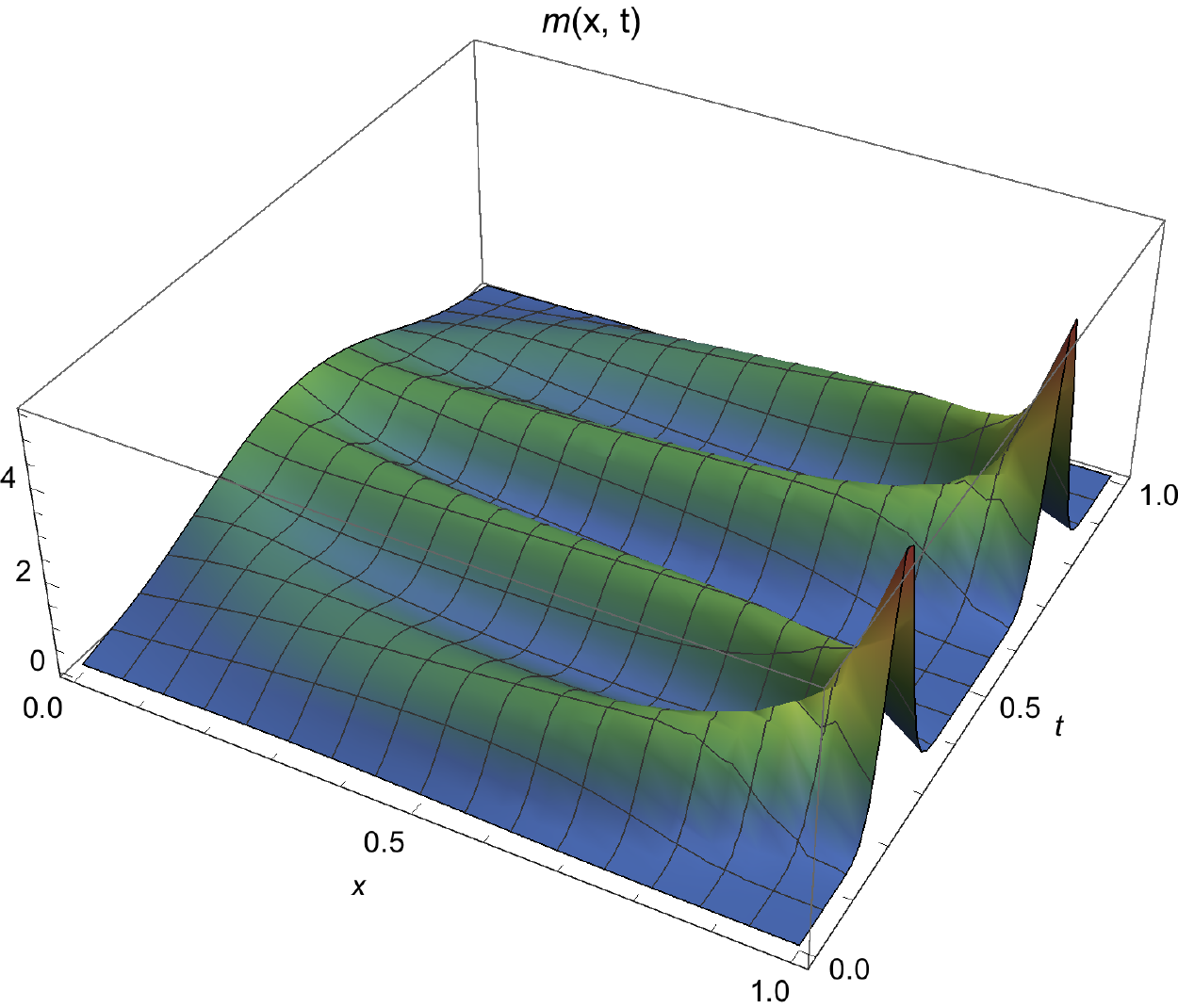}
			        %\caption{Density, $m(t,x)$.}
			    \end{subfigure}%
			    ~
			    \begin{subfigure}[t]{0.3\textwidth}
			        \centering
			        \includegraphics[width=1\textwidth]{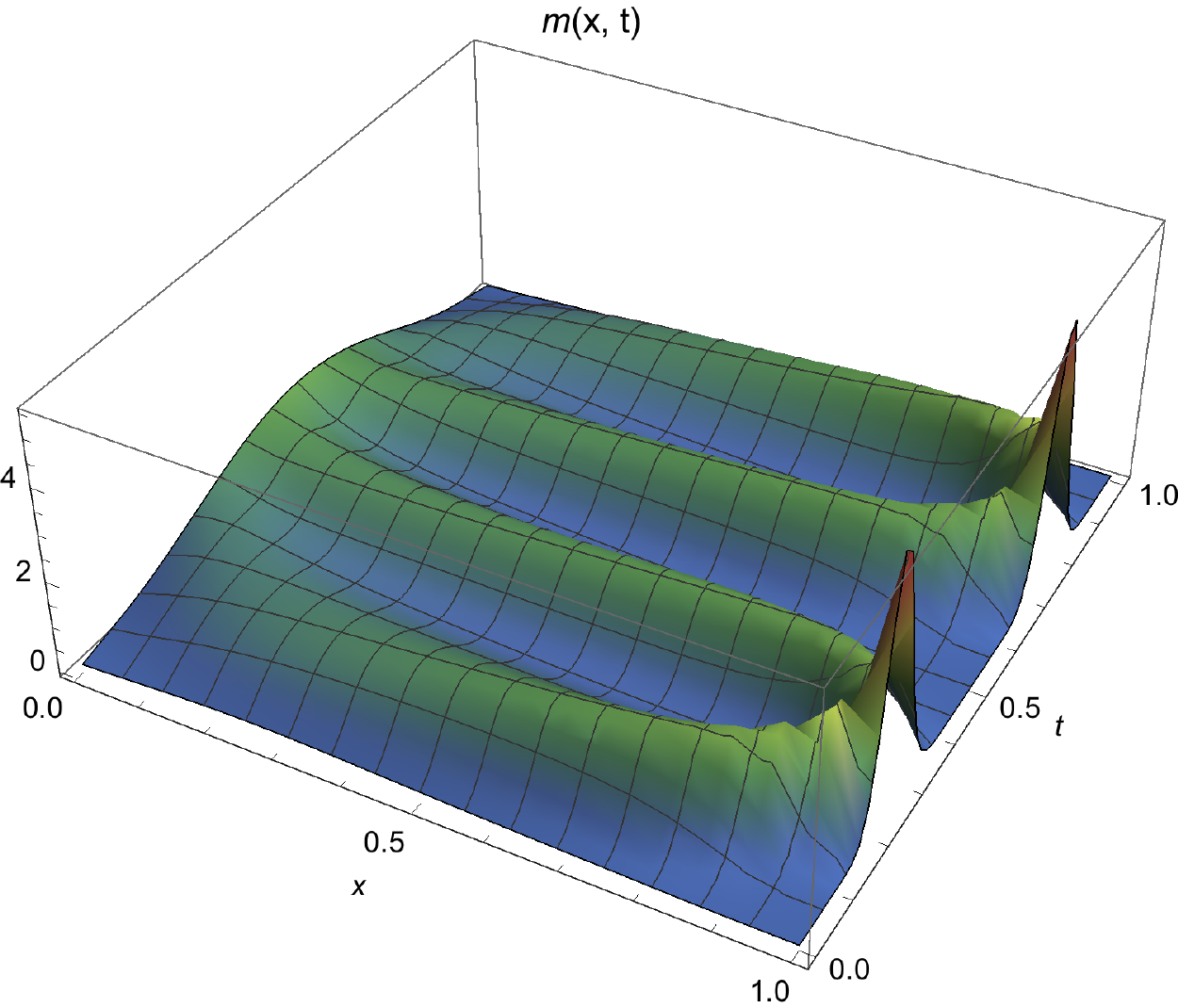}
			        \caption{Density, $m(x,t)$.}
			    \end{subfigure}
				 ~
			    \begin{subfigure}[t]{0.3\textwidth}
			        \centering
			        \includegraphics[width=1\textwidth]{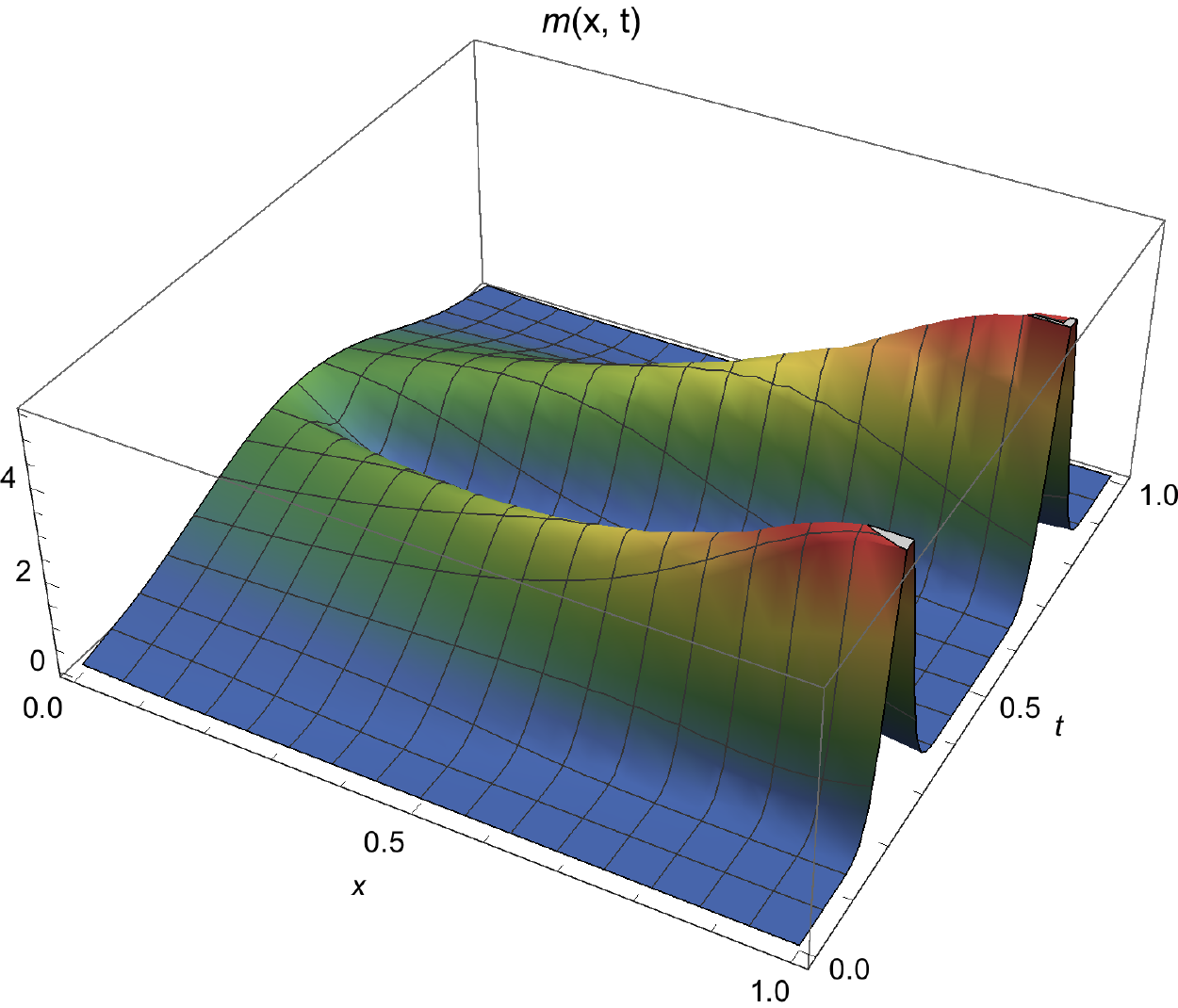}
			        %\caption{Density, $m(t,x)$.}
			    \end{subfigure}%
			
			    \begin{subfigure}[t]{0.3\textwidth}
			        \centering
			        \includegraphics[width=1\textwidth]{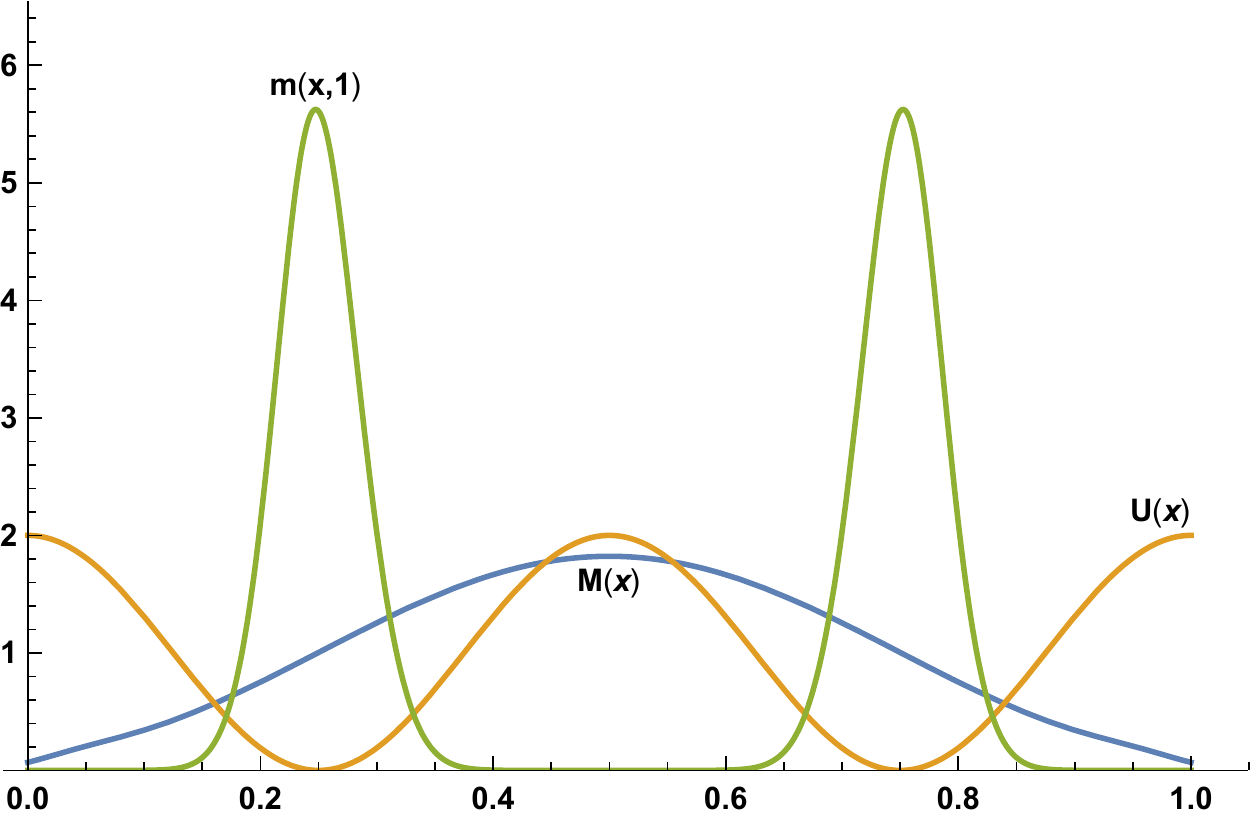}
			        %\caption{$m_0(x), m(T,x), u_t(x)$.}
			    \end{subfigure}%
			    ~
			    \begin{subfigure}[t]{0.3\textwidth}
			        \centering
			        \includegraphics[width=1\textwidth]{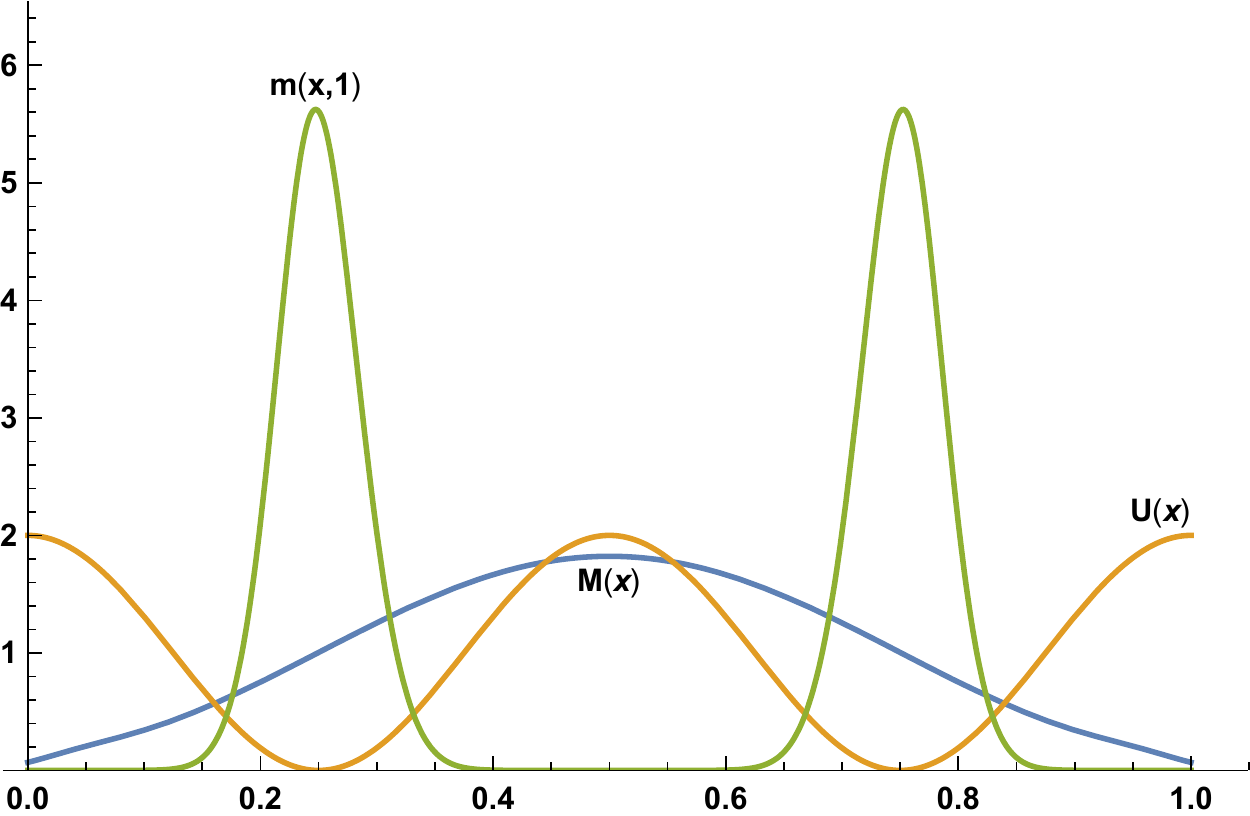}
			        \caption{$M(x)$ -- blue, $m(x,1)$ -- green, $U(x)$ -- orange.}
			    \end{subfigure}
				 ~
			    \begin{subfigure}[t]{0.3\textwidth}
			        \centering
			        \includegraphics[width=1\textwidth]{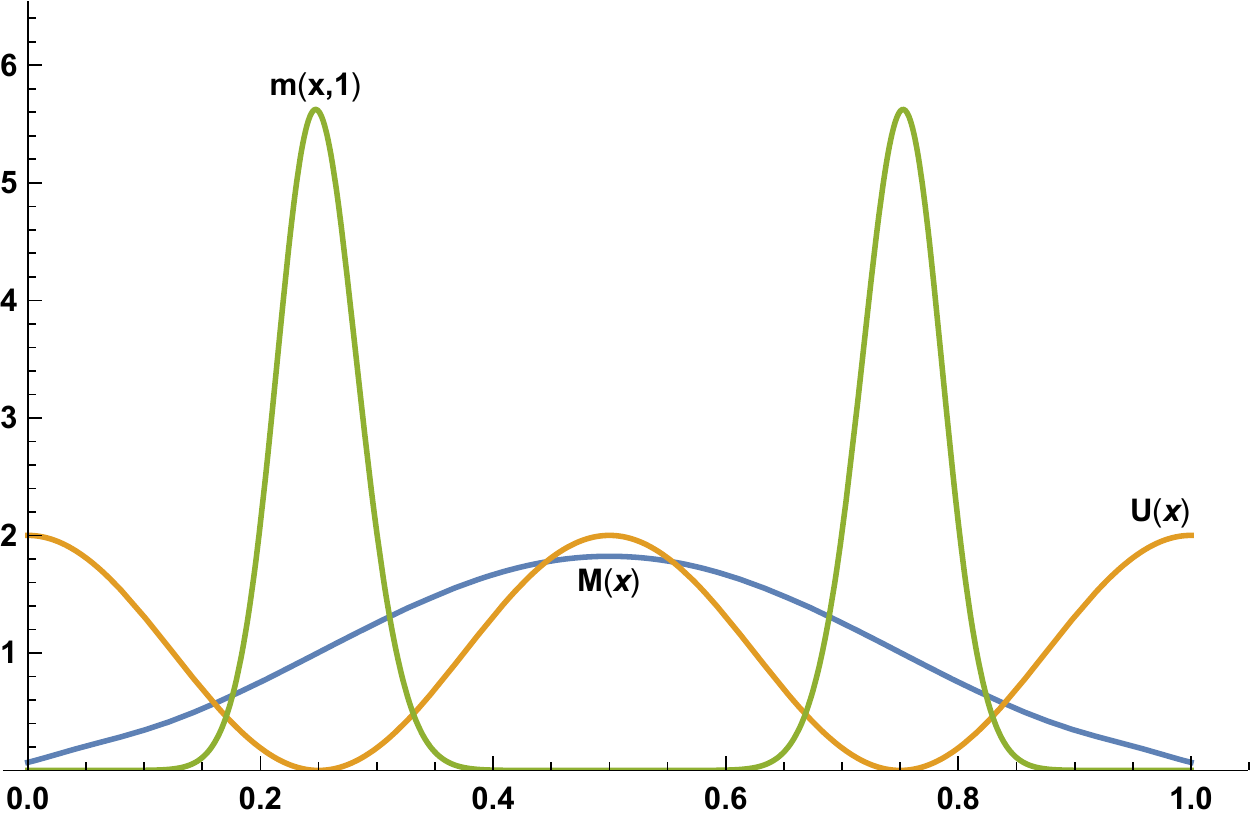}
			        %\caption{$m_0(x), m(T,x), u_t(x)$.}
			    \end{subfigure}%
				\caption{Simulations using Gaussian kernels with different parameters, $(\sigma, \mu) \in \left\{ (0.2,0.5), (0.2,1.5), (0.8,0.5) \right\}$, for each column.
						In the first row, we show a section of each kernel.
						In the second row, we plot the trajectories of the agents, $\{\bx(t,y_\alpha)\}_{\alpha=1}^Q$, at time $t\in[0,1]$ and initial positions $\{y_\alpha\}_{\alpha=1}^Q\subset \Tt$. 
						In the third row, we plot the time evolution of the distribution of players, $m(t,x)$.
						Each plot of the last row displays the initial-terminal conditions, $M(x)$ and $U(x)$, and the final distribution, $m(x,1)$. 
						}
				\label{1D_comparison}
			\end{figure}
\subsection{Two-dimensional examples} \label{sub:2_dimensional_examples}

Here, we consider the case of two-dimensional state space. The initial distribution of players and the terminal cost function are given by
\begin{equation*}
\begin{split}
M(x_1,x_2) =& 1 + \frac 1 2 \cos \lb \pi + 2 \pi \lb x_1 - x_2 \rb \rb + \frac 1 2 \sin \lb \frac \pi 2 + 2 \pi \lb x_1 + x_2 \rb \rb,\\
U(x_1,x_2) =& \frac 3 2 + \frac 1 2 \lb \cos \lb 6 \pi x_1 \rb + \cos \lb 2 \pi x_2 \rb \rb,~(x_1,x_2)\in \Tt^2,			
\end{split}
\end{equation*}
that are depicted in Figure \ref{fig:itc_2D}.
		\begin{figure}[h] 
		    \centering
		    \begin{subfigure}[t]{0.5\textwidth}
		        \centering
		        \includegraphics[width=1\textwidth]{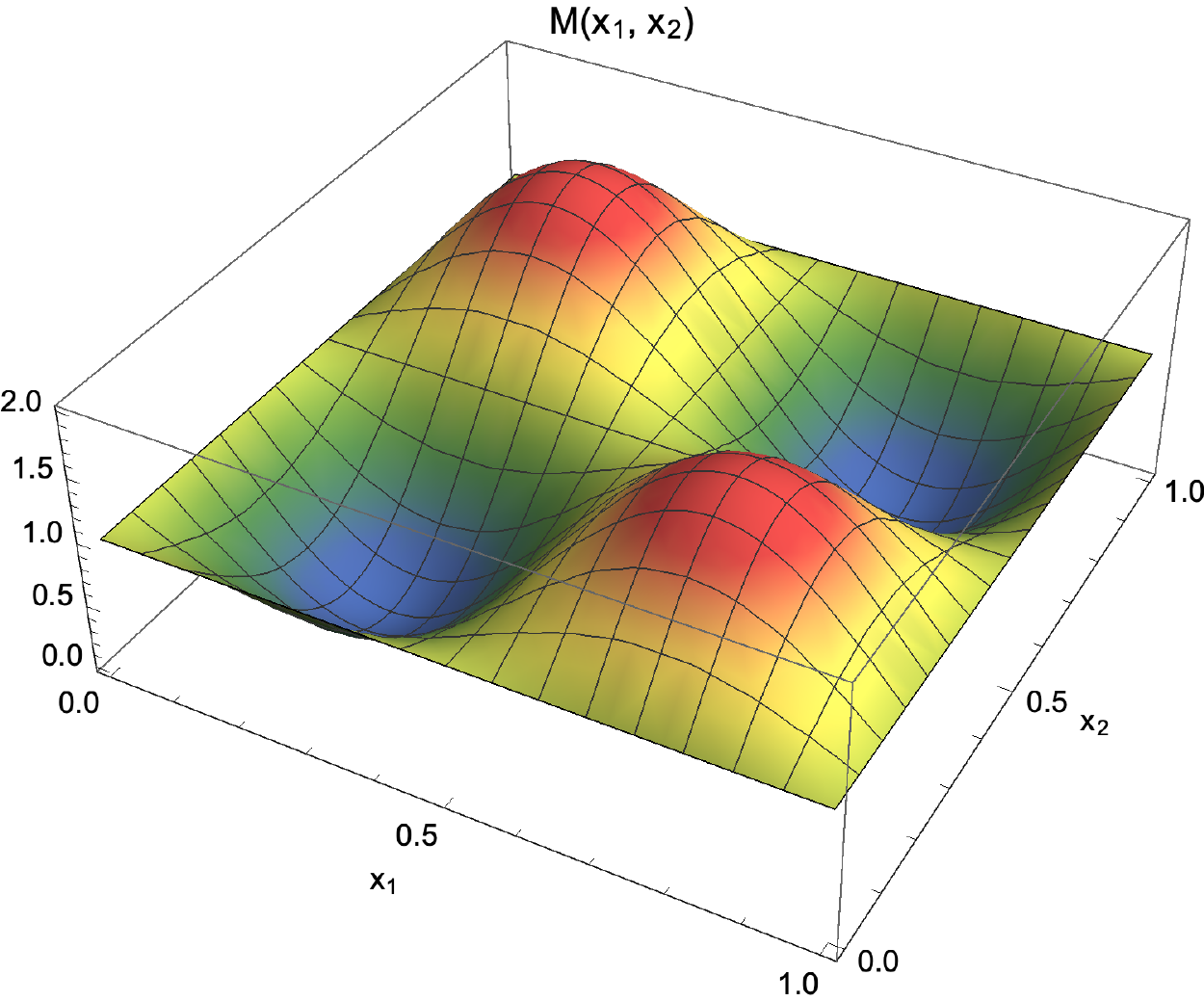} %{Gaussian2D/m_0.pdf}
		        \caption{Initial distribution of agents, $M(x_1,x_2)$.}
		    \end{subfigure}%
		    ~ 
		    \begin{subfigure}[t]{0.5\textwidth}
		        \centering
		        \includegraphics[width=1\textwidth]{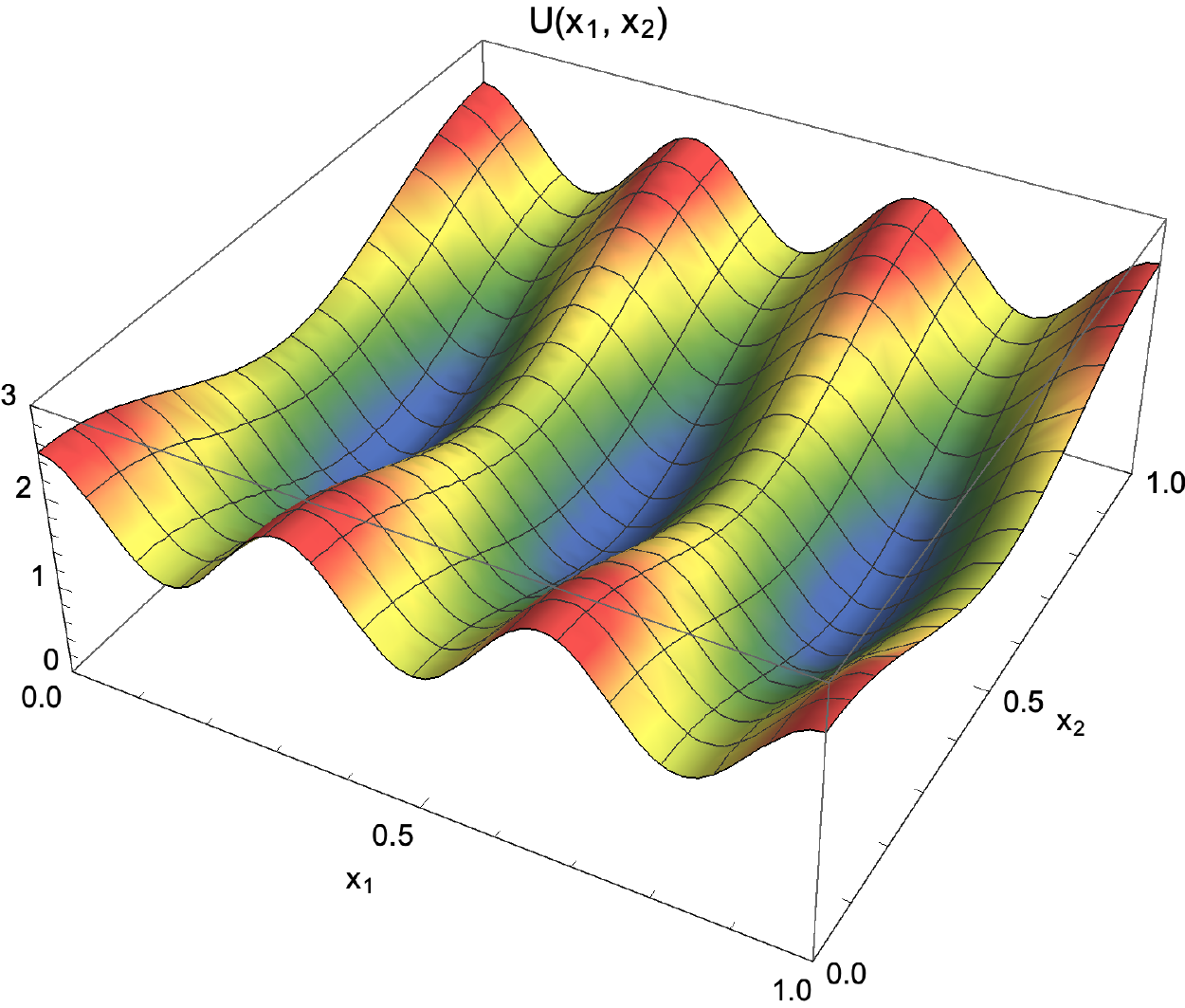}
		        \caption{Terminal cost function, $U(x_1,x_2)$.}
		    \end{subfigure}
			\caption{Initial-terminal conditions.}
			\label{fig:itc_2D}
		\end{figure}
		The corresponding expansion of the kernel is given by
			\begin{equation}\label{eq:kernel12dexpansionphi}
				\begin{split}
					&K^2_{\sigma,\mu}(x_1,x_2;y_1,y_2)\\
					=&\sum_{k,k'=1}^{\infty} \mu^2 e^{-\frac{\pi^2\sigma^2}{2}\left(\left[\frac{k}{2}\right]^2+\left[\frac{k'}{2}\right]^2\right)}\phi_k(x_1) \phi_k(y_1) \phi_{k'}(x_2) \phi_{k'}(y_2)\\
					=&\sum_{k,k'=1}^{\infty} \mu^2 e^{-\frac{\pi^2\sigma^2}{2}\left(\left[\frac{k}{2}\right]^2+\left[\frac{k'}{2}\right]^2\right)}\phi_{k,k'}(x_1,x_2) \phi_{k,k'}(y_1,y_2),
				\end{split}
			\end{equation}
		where
			\begin{equation}\label{eq:tensorphi}
				\phi_{k,k'}(x_1,x_2)=\phi_{k}(x_1)\phi_{k'}(x_2),~x_1,x_2 \in \mathbb{T},~k,k' \in \mathbb{N}.
			\end{equation}
		Thus, for a fixed $r$ we take as a basis functions the set:
			\begin{equation*}
				\begin{split}
					\{\phi_{1,1}, \phi_{1,2},\cdots,\phi_{1,r-1},\phi_{2,1},\cdots,\phi_{2,r-2},\cdots,\phi_{r-1,1}  \}
					= \{\psi_1, \psi_2,\cdots,\psi_{\frac{r(r-1)}{2}}  \}.
				\end{split}
			\end{equation*}
		Therefore, we take all functions $\phi_{k,k'}$ such that $k+k'\leq r$ and order them in the lexicographic order. 
		The corresponding matrices will be of size $\frac{r(r-1)}{2}\times \frac{r(r-1)}{2}$:
			\begin{equation}\label{eq:KJmats12d}
				\begin{split}
					{\bf K}=& \mbox{diag} \left(\mu^2 e^{-\frac{\pi^2\sigma^2}{2}\left(\left[\frac{k}{2}\right]^2+\left[\frac{k'}{2}\right]^2\right)}\right)_{k+k'\leq r},\\
					{\bf J}=& \mbox{diag} \left(\mu^{-2} e^{\frac{\pi^2\sigma^2}{2}\left(\left[\frac{k}{2}\right]^2+\left[\frac{k'}{2}\right]^2\right)}\right)_{k+k'\leq r},
				\end{split}
			\end{equation}
		where the order is again lexicographic.
		
		To compare the results, we use the same time and space discretization throughout all our $2-$dimensional experiments, as well as the same parameters for the numerical scheme. We discretize the time using a step size $\Delta t = \frac{1}{N}$. For the discretization of $M$ we use
		\begin{equation*}
		\begin{split}
		y_{\alpha \alpha'}=\lb \frac{\alpha}{Q+1},\frac{\alpha'}{Q+1} \rb, \quad c_{\alpha \alpha'}=\frac{M(y_{\alpha \alpha'})}{\sum_{\beta,\beta'=1}^Q M(y_{\beta \beta'})},~1\leq \alpha,\alpha' \leq Q.
		\end{split}
		\end{equation*}
		We choose $N=20,~Q=20$ and use eight basis functions, $r=8$. Furthermore, we set the numerical scheme parameters to $\lambda = 1,~\omega = \frac{1}{12}$ and $\theta=1$.
		
		In Figure \ref{2d_kernels}, we plot the Gaussian kernels used in the simulations, with different values of $\mu$ and $\sigma$.
		We see that the bigger $\mu$ is the higher the peak of the kernel, see (a) and (b) in Figure \ref{2d_kernels}.
		This means that each agent in (a) is more adverse of being in crowded areas than agents is (b), $\mu=0.75$ and $\mu=0.5$ respectively.
		For higher values of $\sigma$ we see that the kernel becomes flat, compare (b) with (c) in Figure \ref{2d_kernels}, for $\sigma=0.1$ and $\sigma=1$ respectively. As before, this means that the agents penalize others independent of mutual distances.
			\begin{figure}[h]
			    \centering
			    \begin{subfigure}[t]{0.31\textwidth}
			        \centering
			        \includegraphics[width=1\textwidth]{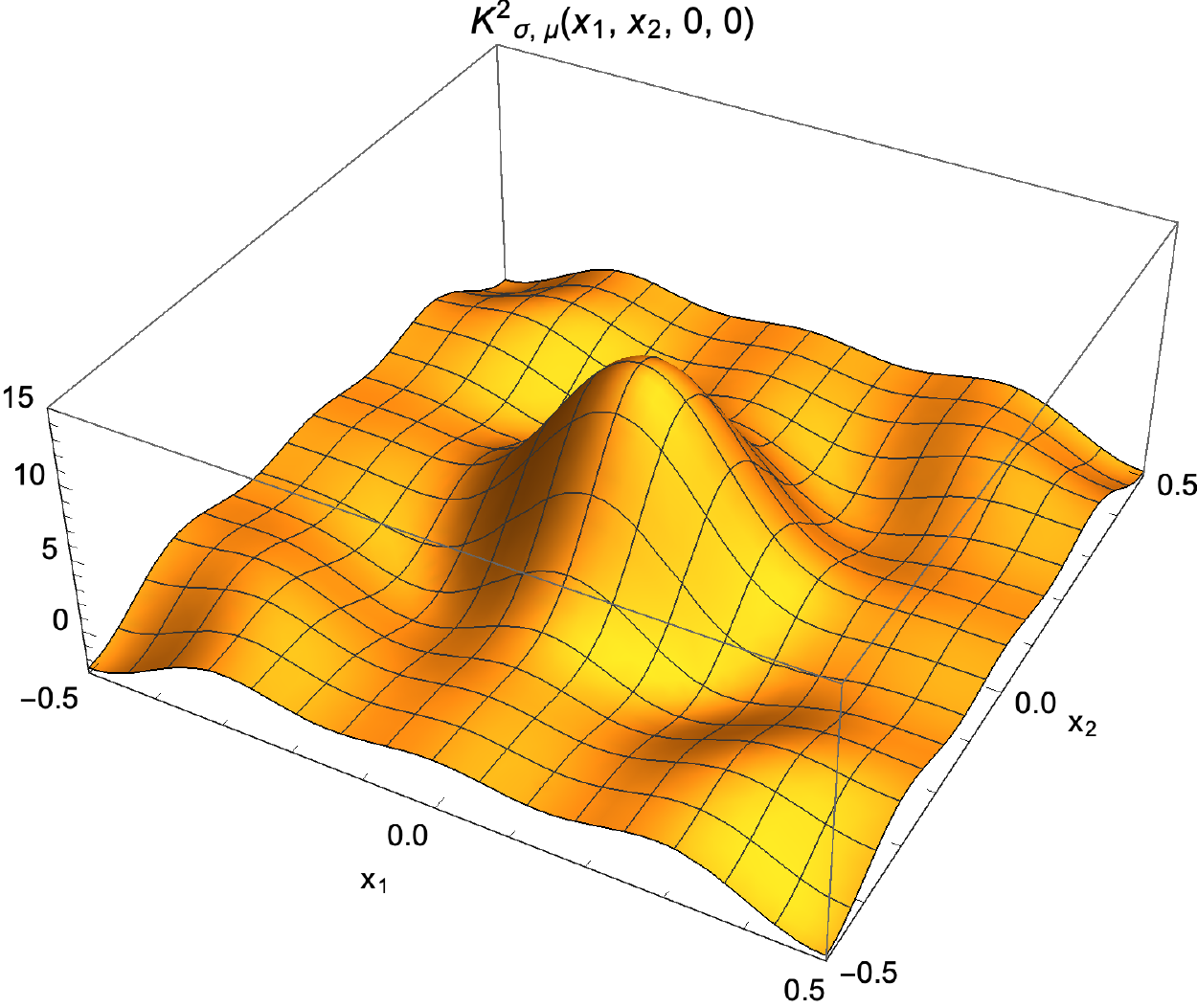} %{Gaussian2D/sim2/Kernel_section.pdf}
			        \caption{$K^2_{0.1, 0.75}(x_1,x_2;0,0)$.}%{Gaussian kernel, $K^1_{0.1, 1.1}(x,y)$.}
			    \end{subfigure}%
				 ~
			    \begin{subfigure}[t]{0.31\textwidth}
			        \centering
			        \includegraphics[width=1\textwidth]{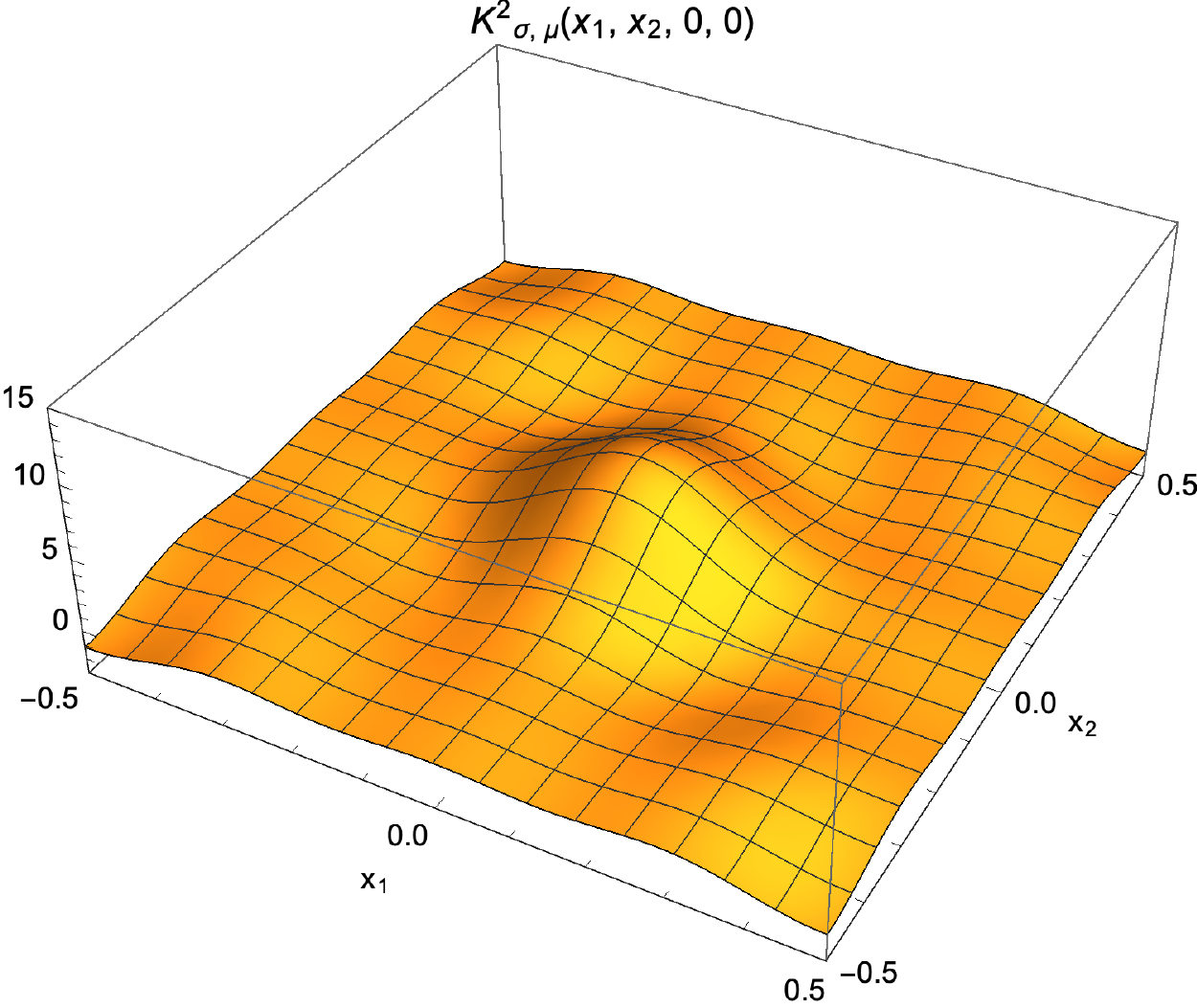}
			        \caption{$K^2_{0.1, 0.5}(x_1,x_2;0,0)$.}%{Gaussian kernel, $K^1_{0.1, 0.8}(x,y)$.}
			    \end{subfigure}%
			    ~ 
			    \begin{subfigure}[t]{0.31\textwidth}
			        \centering
			        \includegraphics[width=1\textwidth]{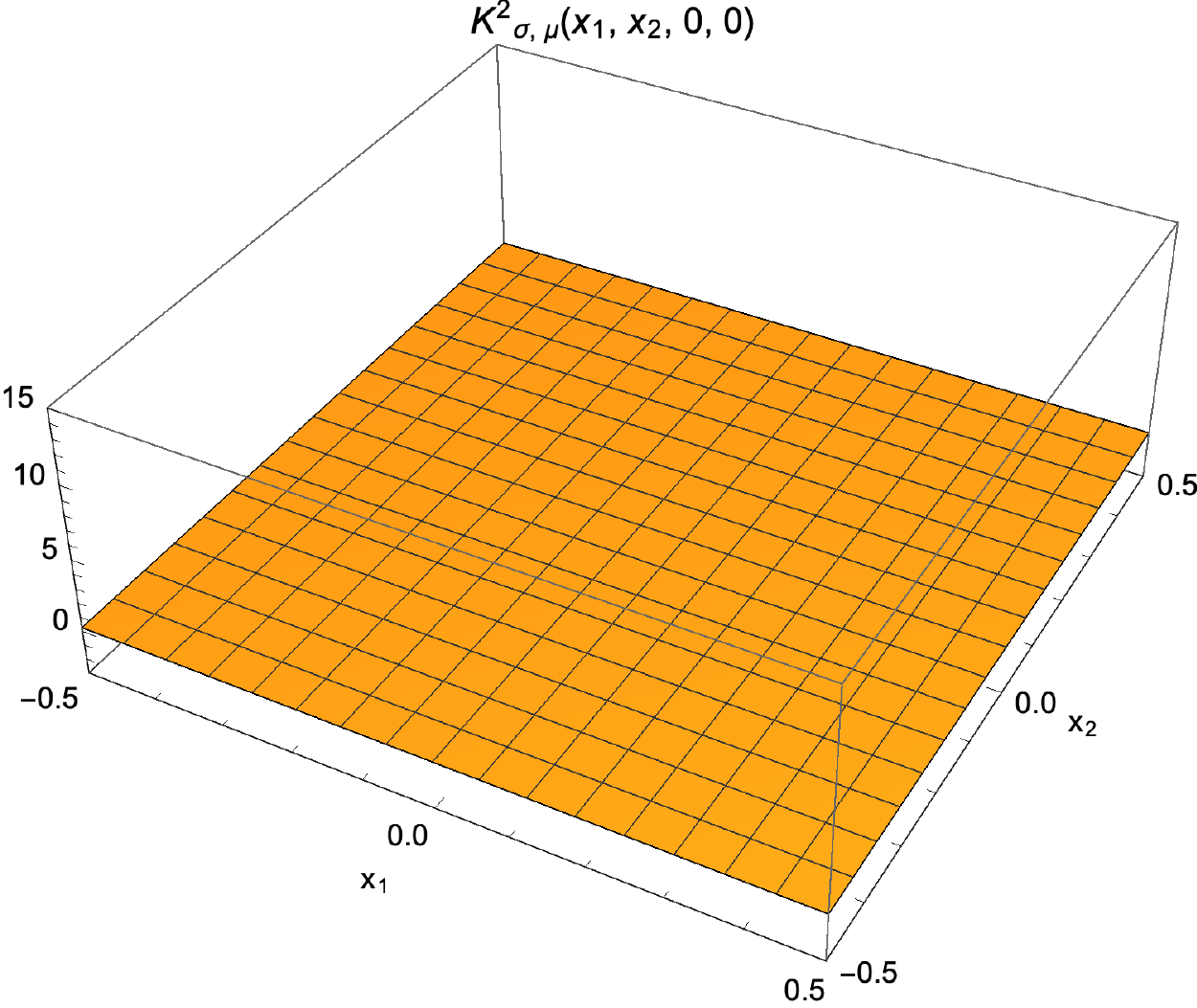}
			        \caption{$K^2_{1, 0.5}(x_1,x_2;0,0)$.}%{Gaussian kernel, $K^1_{0.5, 0.8}(x,y)$.}
			    \end{subfigure}%
				\caption{Plots of the Gaussian kernels for $(\sigma,\mu) \in \{(0.1,0.75),(0.1,0.5),(1,0.5)\}$.}
				\label{2d_kernels}
			\end{figure}
		
			In Figure \ref{fig:2d_comp}, we compare the simulation results using the same initial-terminal conditions, see Figure \ref{fig:itc_2D}, but different kernel functions (plotted in the first row of Figure \ref{fig:2d_comp}). In the last row of Figure \ref{fig:2d_comp} we have the final distribution of agents.
			
			We see that for larger values of $\mu$, left column compared with the middle one, the agents' concentration near low-cost regions of terminal cost, $U$, is less dense. We also see that when $\sigma$ is bigger the the agents become more indifferent to the density of the crowd, and concentrate more densely near low-cost values of $U$ -- see the right column in Figure 6 (f).
			
			As in the 1-dimensional case, looking to the projected trajectories in the 2-dimensional plane we observe that for flat kernel agents follow straight lines from the initial positions to closest low-cost regions of the terminal cost function. 
			\begin{figure}[h]
			    \centering
			    \begin{subfigure}[t]{0.3\textwidth}
			        \centering
			        \includegraphics[width=1\textwidth]{Gaussian2D_32Kernel_section.pdf}
			        \caption{$K^2_{0.1, 0.75}(x_1,x_2;0,0)$.}
			    \end{subfigure}
				 ~
			    \begin{subfigure}[t]{0.3\textwidth}
			        \centering
			        \includegraphics[width=1\textwidth]{Gaussian2D_31Kernel_section.pdf}
			        \caption{$K^2_{0.1, 0.5}(x_1,x_2;0,0)$.}
			    \end{subfigure}%
			    ~
			    \begin{subfigure}[t]{0.3\textwidth}
			        \centering
			        \includegraphics[width=1\textwidth]{Gaussian2D_33Kernel_section.pdf}
			        \caption{$K^2_{1, 0.5}(x_1,x_2;0,0)$.}
			    \end{subfigure}
		    
			    \begin{subfigure}[t]{0.3\textwidth}
			        \centering
			        \includegraphics[width=1\textwidth]{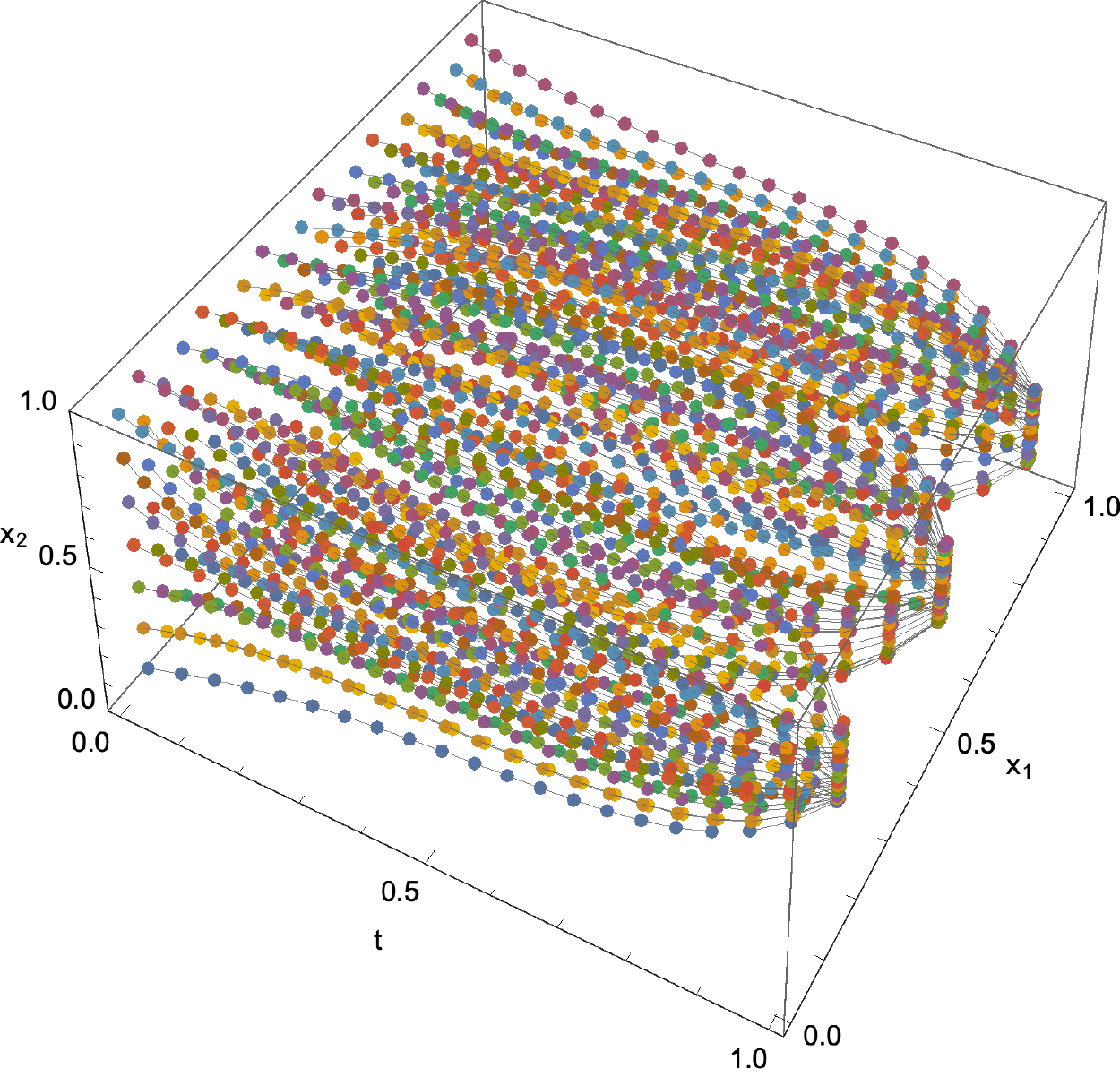}
			        %\caption{Trajectories, $\bx(t,x)$.}
			    \end{subfigure}
				 ~
			    \begin{subfigure}[t]{0.3\textwidth}
			        \centering
			        \includegraphics[width=1\textwidth]{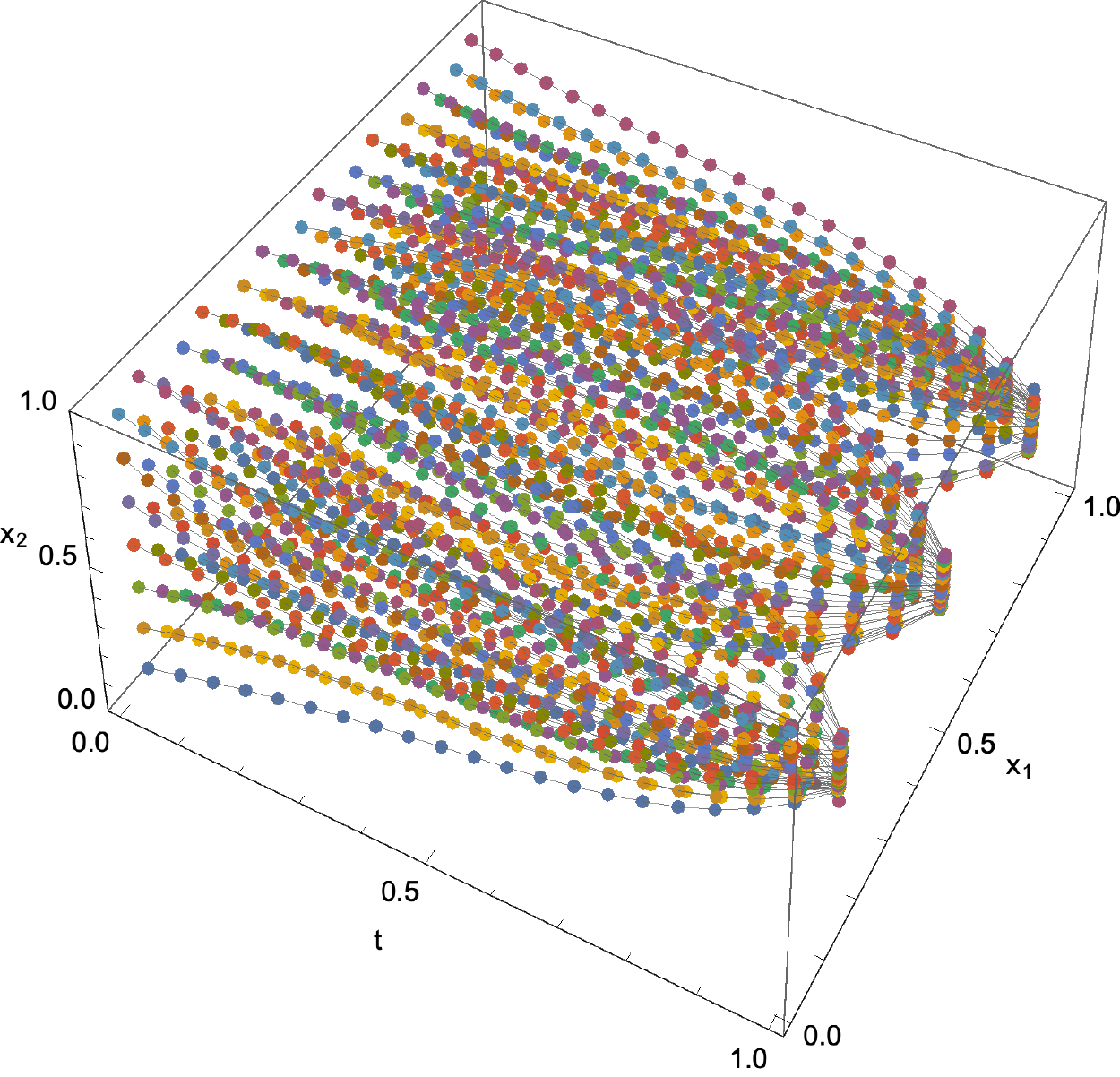}
			        \caption{Trajectories, $\bx(t,y_{\alpha})$.}
			    \end{subfigure}%
			    ~
			    \begin{subfigure}[t]{0.3\textwidth}
			        \centering
			        \includegraphics[width=1\textwidth]{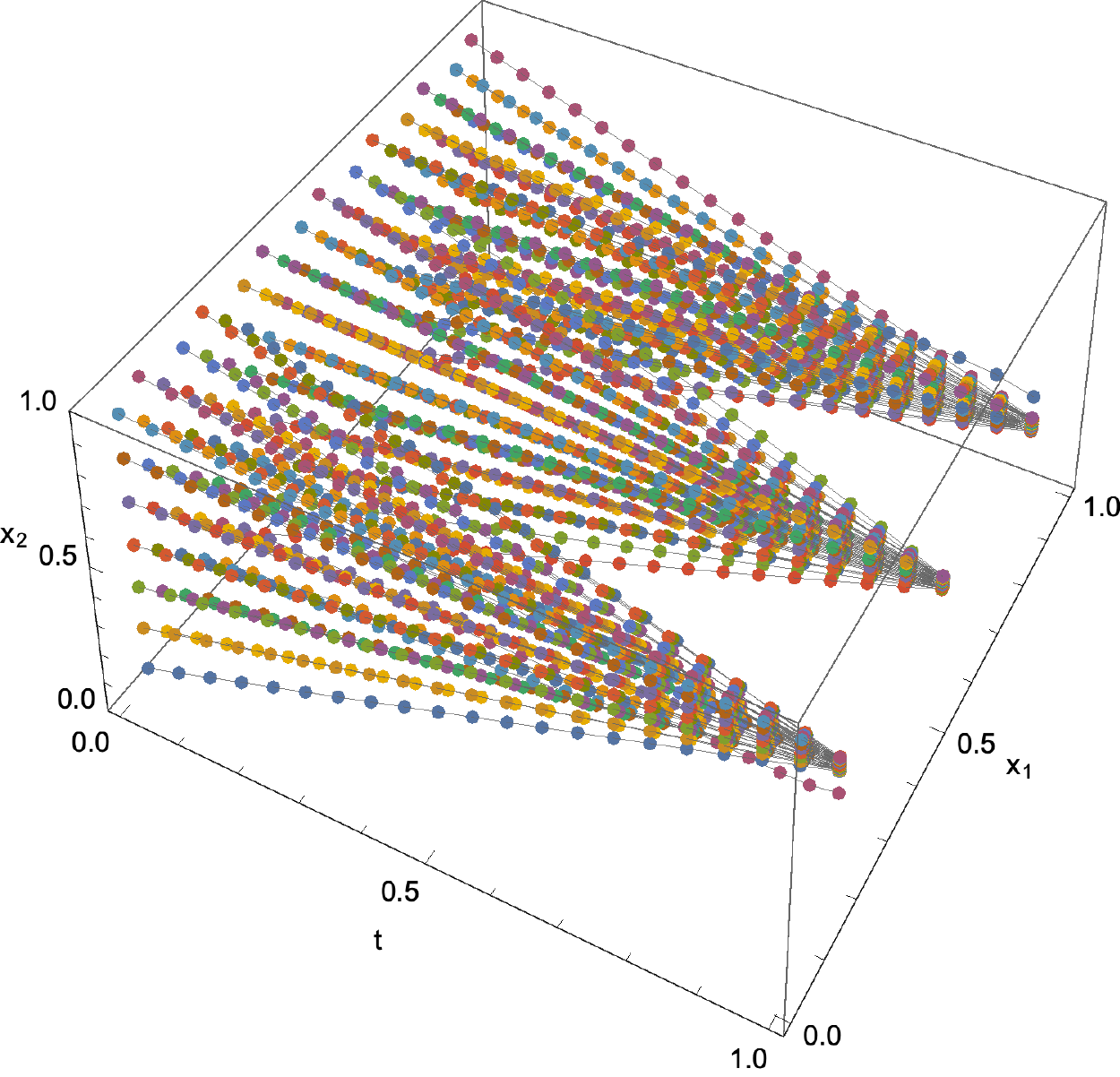}
			        %\caption{Trajectories, $\bx(t,x)$.}
			    \end{subfigure}%
				 
			    \begin{subfigure}[t]{0.3\textwidth}
			        \centering
			        \includegraphics[width=1\textwidth]{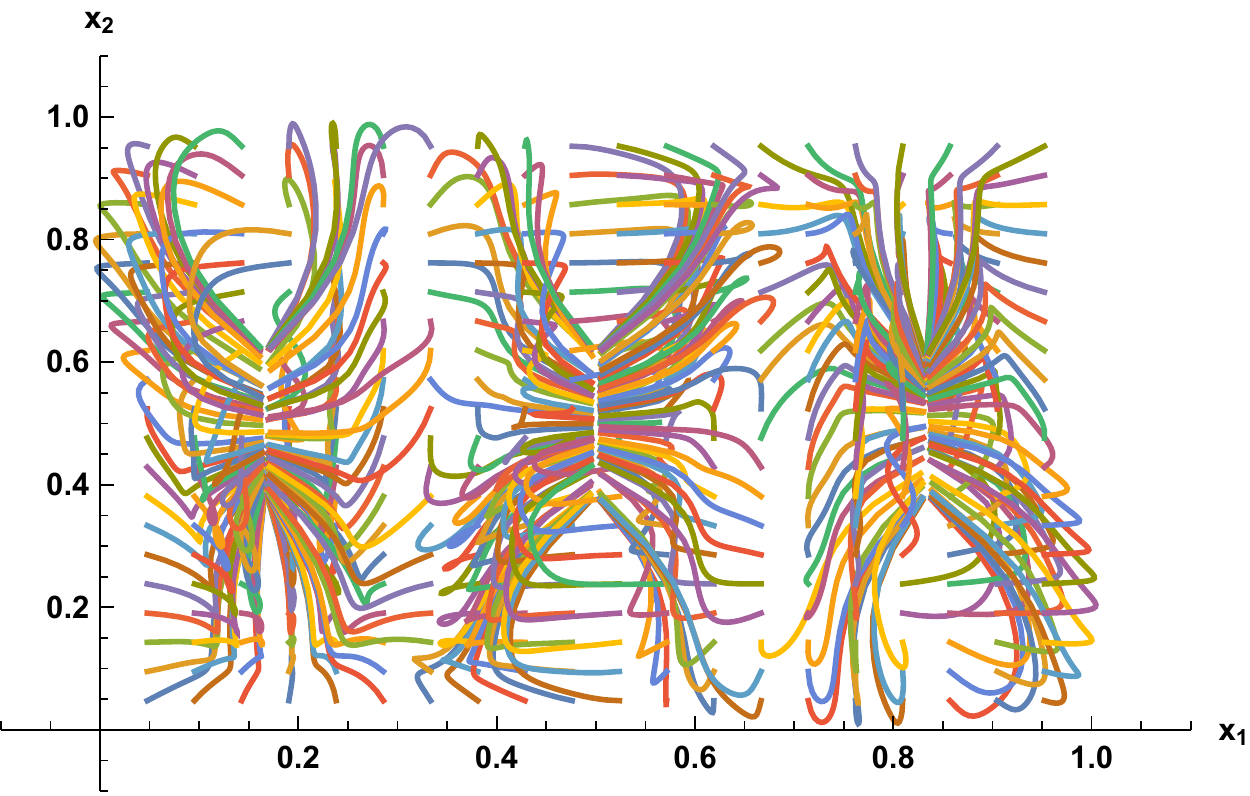}
			        %\caption{Trajectories, $\bx(t,x)$.}
			    \end{subfigure}
				 ~
			    \begin{subfigure}[t]{0.3\textwidth}
			        \centering
			        \includegraphics[width=1\textwidth]{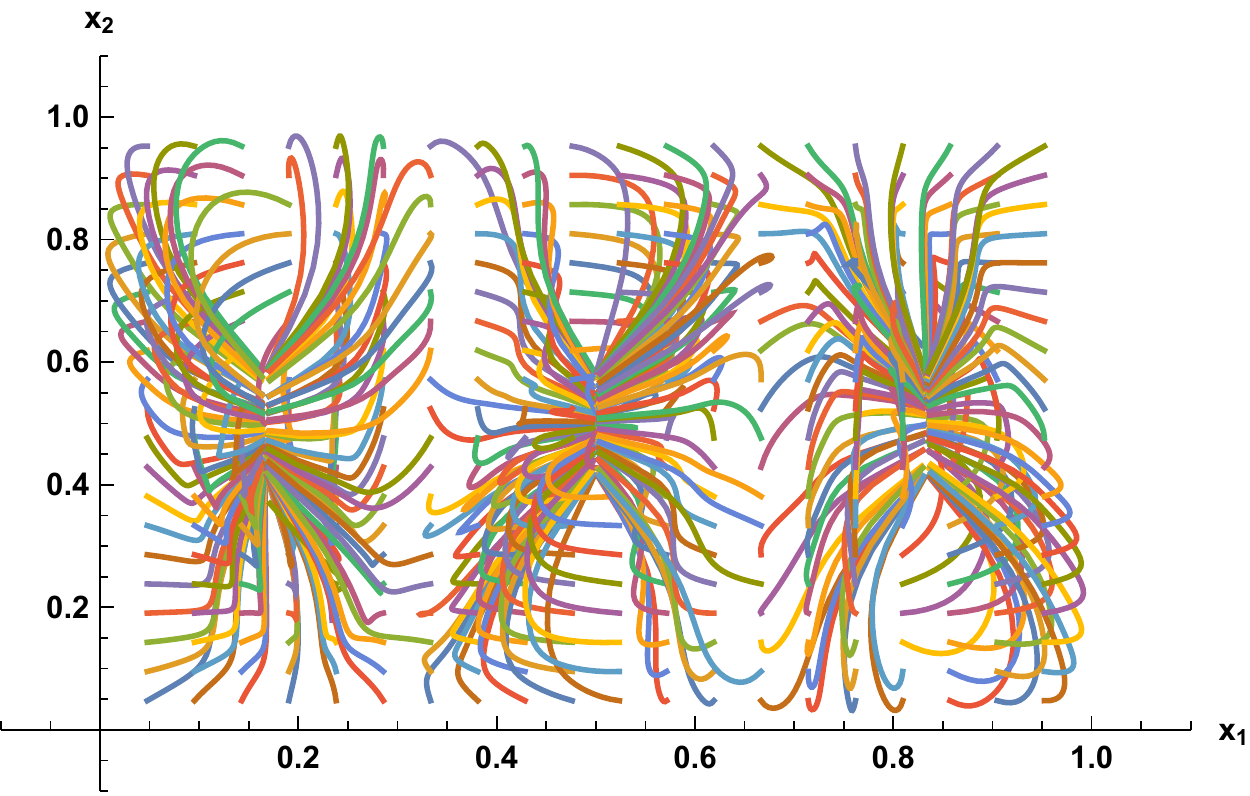}
			        \caption{Projected trajectories.}
			    \end{subfigure}%
			    ~
			    \begin{subfigure}[t]{0.3\textwidth}
			        \centering
			        \includegraphics[width=1\textwidth]{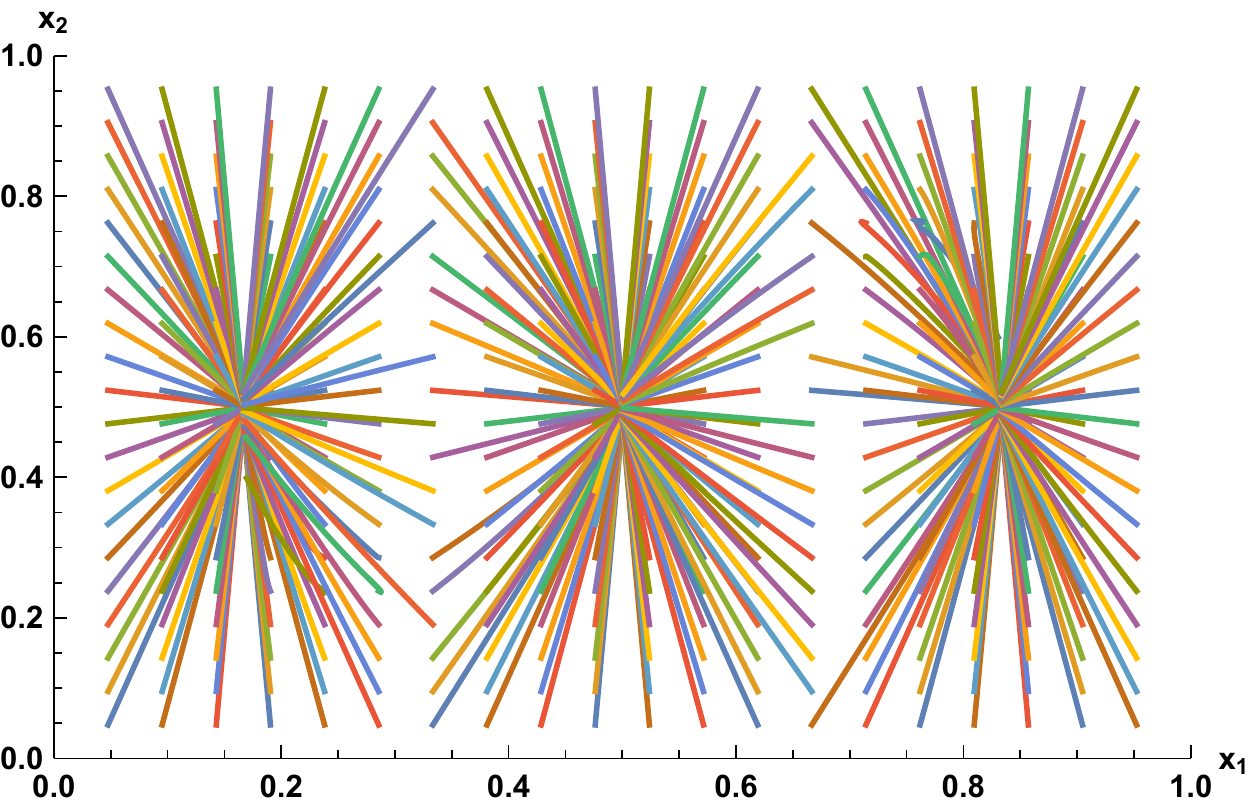}
			        %\caption{Trajectories, $\bx(t,x)$.}
			    \end{subfigure}%
				 
			    \begin{subfigure}[t]{0.3\textwidth}
			        \centering
			        \includegraphics[width=1\textwidth]{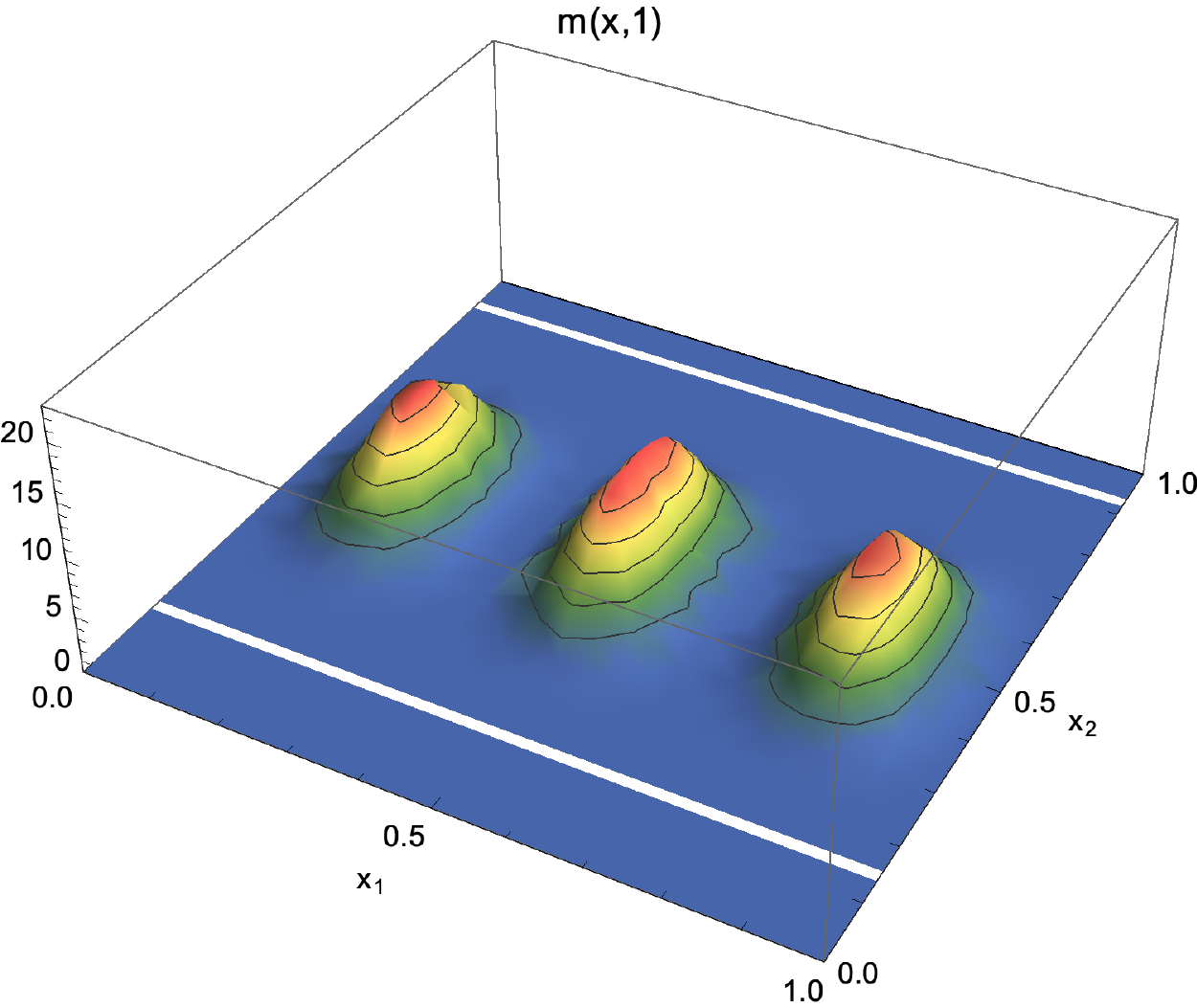}
			        %\caption{Density, $m(t,x)$.}
			    \end{subfigure}
				 ~
			    \begin{subfigure}[t]{0.3\textwidth}
			        \centering
			        \includegraphics[width=1\textwidth]{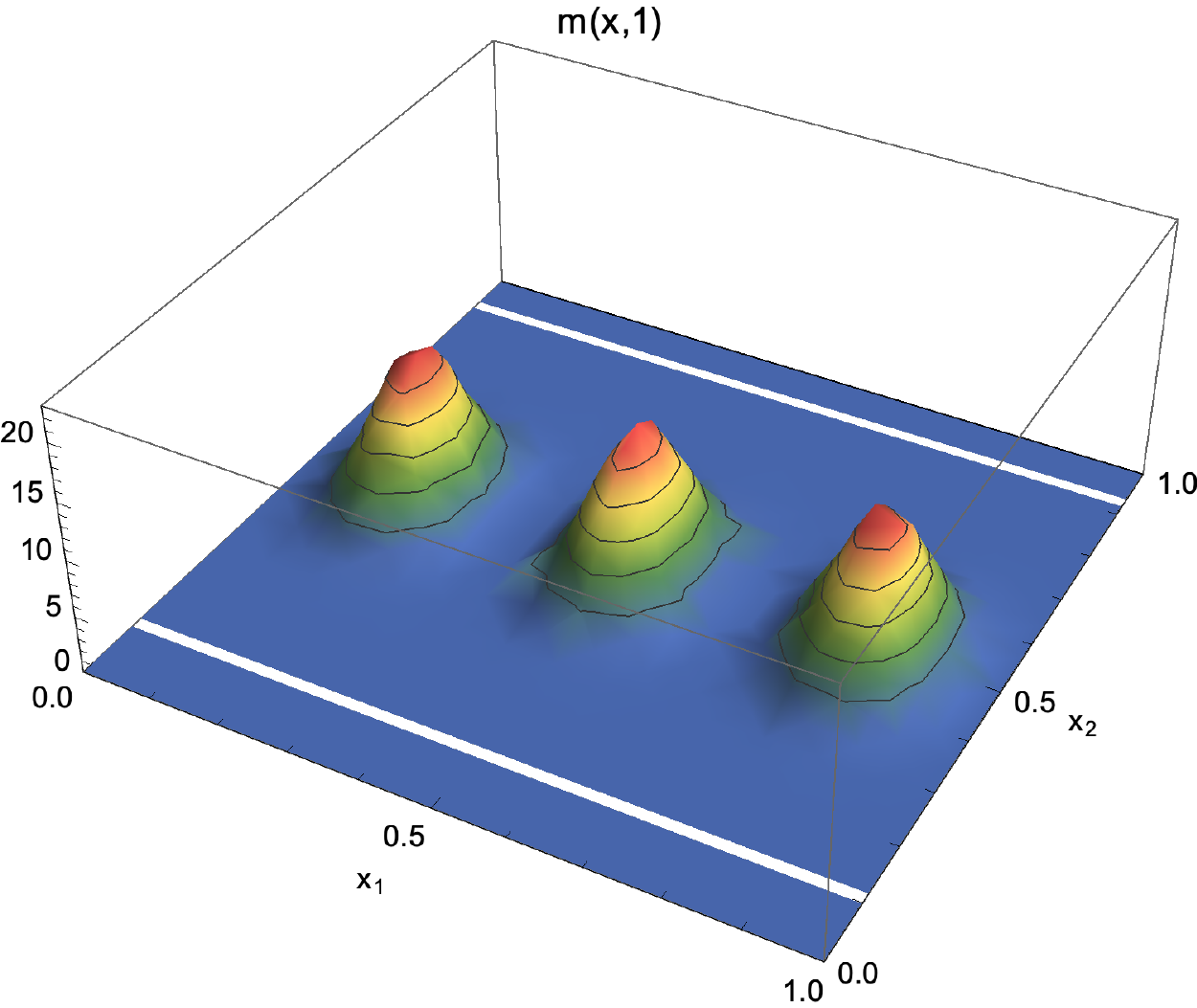}
			        \caption{Final density, $m(\bx,1)$.}
			    \end{subfigure}%
			    ~
			    \begin{subfigure}[t]{0.3\textwidth}
			        \centering
			        \includegraphics[width=1\textwidth]{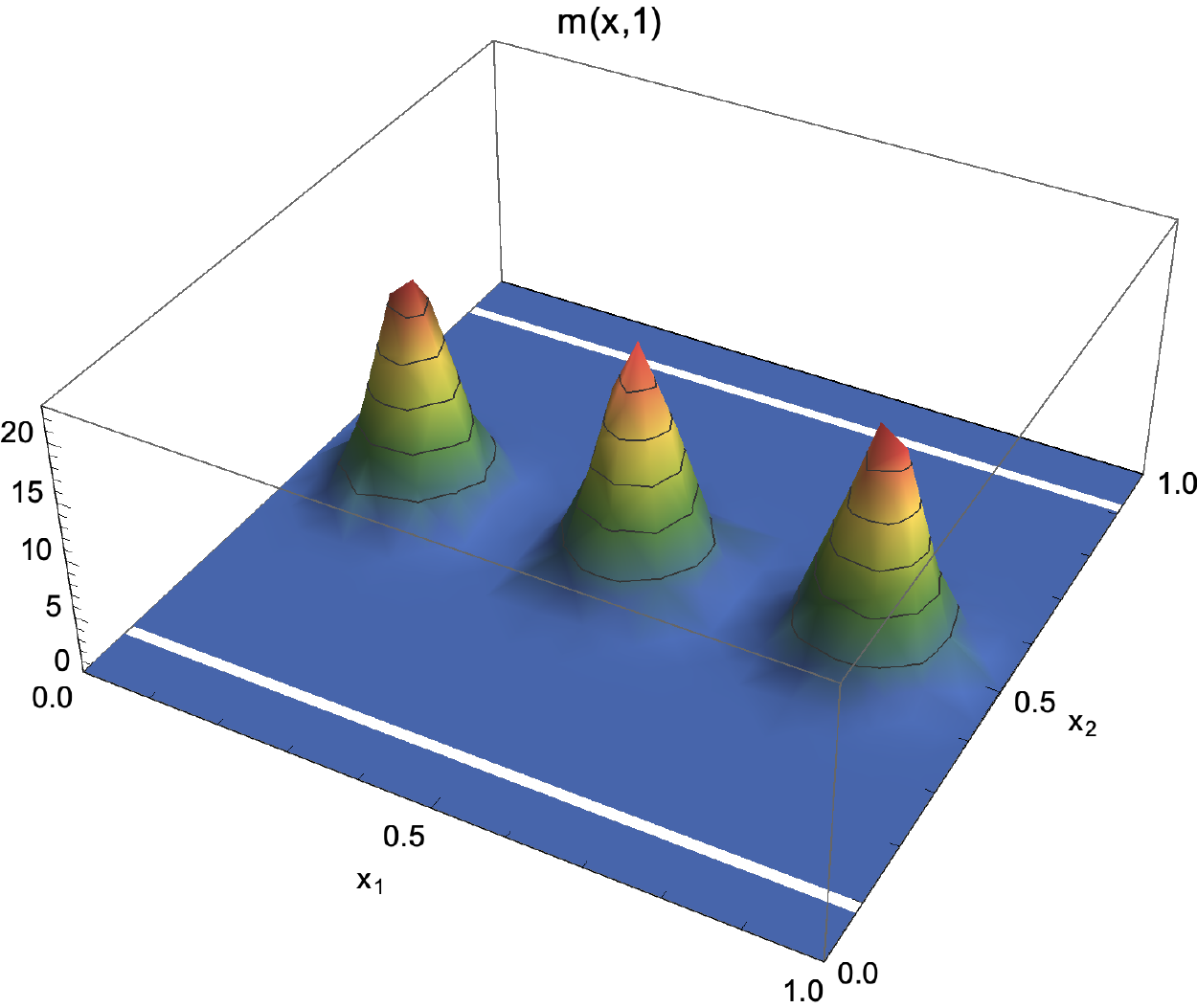}
			        %\caption{Density, $m(t,x)$.}
			    \end{subfigure}%
				 %
				 % 			    \begin{subfigure}[t]{0.3\textwidth}
				 % 			        \centering
				 % 			        \includegraphics[width=1\textwidth]{Gaussian/sim1/Densities_cost.pdf}
				 % 			        \caption{$m_0(x), m(T,x), u_t(x)$.}
				 % 			    \end{subfigure}%
				 % 			    ~
				 % 			    \begin{subfigure}[t]{0.3\textwidth}
				 % 			        \centering
				 % 			        \includegraphics[width=1\textwidth]{Gaussian/sim3/Densities_cost.pdf}
				 % 			        \caption{$m_0(x), m(T,x), u_t(x)$.}
				 % 			    \end{subfigure}
				 % ~
				 % 			    \begin{subfigure}[t]{0.3\textwidth}
				 % 			        \centering
				 % 			        \includegraphics[width=1\textwidth]{Gaussian/sim2/Densities_cost.pdf}
				 % 			        \caption{$m_0(x), m(T,x), u_t(x)$.}
				 % 			    \end{subfigure}%
				\caption{Simulations using Gaussian kernels with different parameters, $(\sigma, \mu) \in \left\{ (0.1,0.75), (0.1,0.5), (1,0.5) \right\}$, for each column.
						In the first row we show a section of each kernel.
						In the second row we show the trajectories of the agents, $\{\bx(t,y_{\alpha \alpha'})\}_{\alpha,\alpha'=1}^Q,~t\in[0,1],$ with initial positions $\{y_{\alpha \alpha'}\}_{\alpha,\alpha'=1}^Q\in \Tt^2$. 
						In the third row, we plot the 2D projection of the trajectories. And in the last row, we plot the final distribution of the agents, $m(\bx,1)$. 
						}
				\label{fig:2d_comp}
			\end{figure}
		% subsubsection gaussian_kernel (end)
		% \subsubsection{Piece-wise linear kernels} % (fold)
		% \label{ssub:piece_wise_linear_kernels}
		% 	For the piecewise linear kernel we have
		% 	\begin{equation}\label{eq:kernel22d}
		% 		B^2_{\sigma,\mu}(x_1,x_2;y_1,y_2)=B^1_{\sigma,\mu}(x_1,y_1)~B^1_{\sigma,\mu}(x_2,y_2),~x_i,y_i \in \mathbb{T}.
		% 	\end{equation}
		% 	The corresponding expansion with respect to $\phi$ basis is
		% 	\begin{equation}
		% 		\begin{split}
		% 			&B^2_{\sigma,\mu}(x_1,x_2;y_1,y_2)\\
		% 			=&\sum_{k,k'=1}^{\infty} \mu^2 \left(\frac{\sin \pi \sigma \left[\frac{k}{2}\right] \sin \pi \sigma \left[\frac{k'}{2}\right]}{\pi^2 \sigma^2 \left[\frac{k}{2}\right] \left[\frac{k'}{2}\right]}\right)^2\phi_{k,k'}(x_1,x_2) \phi_{k,k'}(y_1,y_2).
		% 		\end{split}
		% 	\end{equation}
		% 	Consequently, the ${\bf K, J}$ matrices are
		% 	\begin{equation}\label{eq:KJmats22d}
		% 		\begin{split}
		% 			{\bf K}=& \mbox{diag} \left(\mu^2 \left(\frac{\sin \pi \sigma \left[\frac{k}{2}\right] \sin \pi \sigma \left[\frac{k'}{2}\right]}{\pi^2 \sigma^2 \left[\frac{k}{2}\right] \left[\frac{k'}{2}\right]}\right)^2\right)_{k+k'\leq r},\\
		% 			{\bf J}=& \mbox{diag} \left(\mu^{-2} \left(\frac{\pi^2 \sigma^2 \left[\frac{k}{2}\right] \left[\frac{k'}{2}\right]}{\sin \pi \sigma \left[\frac{k}{2}\right] \sin \pi \sigma \left[\frac{k'}{2}\right]}\right)^2\right)_{k+k'\leq r}.
		% 		\end{split}
		% 	\end{equation}
		% % subsubsection piece_wise_linear_kernels (end)
	% subsection 2_dimensional_examples (end)
% section numerical_examples (end)

\frenchspacing
\bibliographystyle{plain}
\bibliography{nonlocalPM.bib}

\begin{thebibliography}{10}

\bibitem{achdou'13}
Y.~Achdou.
\newblock Finite difference methods for mean field games.
\newblock In {\em Hamilton-{J}acobi equations: approximations, numerical
  analysis and applications}, volume 2074 of {\em Lecture Notes in Math.},
  pages 1--47. Springer, Heidelberg, 2013.

\bibitem{achdcamdolcetta'13}
Y.~Achdou, F.~Camilli, and I.~Capuzzo-Dolcetta.
\newblock Mean field games: convergence of a finite difference method.
\newblock {\em SIAM J. Numer. Anal.}, 51(5):2585--2612, 2013.

\bibitem{achddolcetta'10}
Y.~Achdou and I.~Capuzzo-Dolcetta.
\newblock Mean field games: numerical methods.
\newblock {\em SIAM J. Numer. Anal.}, 48(3):1136--1162, 2010.

\bibitem{achporretta'16}
Y.~Achdou and A.~Porretta.
\newblock Convergence of a finite difference scheme to weak solutions of the
  system of partial differential equations arising in mean field games.
\newblock {\em SIAM J. Numer. Anal.}, 54(1):161--186, 2016.

\bibitem{alfego'17}
N.~Almulla, R.~Ferreira, and D.~A. Gomes.
\newblock Two numerical approaches to stationary mean-field games.
\newblock {\em Dyn. Games Appl.}, 7(4):657--682, 2017.

\bibitem{andreev'17}
R.~Andreev.
\newblock Preconditioning the augmented {L}agrangian method for instationary
  mean field games with diffusion.
\newblock {\em SIAM J. Sci. Comput.}, 39(6):A2763--A2783, 2017.

\bibitem{aurelldjehiche'18}
A.~Aurell and B.~Djehiche.
\newblock Mean-field type modeling of nonlocal crowd aversion in pedestrian
  crowd dynamics.
\newblock {\em SIAM J. Control Optim.}, 56(1):434--455, 2018.

\bibitem{bardidolcetta'97}
M.~Bardi and I.~Capuzzo-Dolcetta.
\newblock {\em Optimal control and viscosity solutions of
  {H}amilton-{J}acobi-{B}ellman equations}.
\newblock Systems \& Control: Foundations \& Applications. Birkh\"{a}user
  Boston, Inc., Boston, MA, 1997.
\newblock With appendices by Maurizio Falcone and Pierpaolo Soravia.

\bibitem{bencar'15}
J.-D. Benamou and G.~Carlier.
\newblock Augmented {L}agrangian methods for transport optimization, mean field
  games and degenerate elliptic equations.
\newblock {\em J. Optim. Theory Appl.}, 167(1):1--26, 2015.

\bibitem{bencarmarnen'18}
J.-D. Benamou, G.~Carlier, S.~Di~Marino, and L.~Nenna.
\newblock An entropy minimization approach to second-order variational
  mean-field games.
\newblock {\em Preprint}, 2018.
\newblock arXiv:1807.09078v1 [math.OC].

\bibitem{bencarsan'17}
J.-D. Benamou, G.~Carlier, and F.~Santambrogio.
\newblock Variational mean field games.
\newblock In {\em Active particles. {V}ol. 1. {A}dvances in theory, models, and
  applications}, Model. Simul. Sci. Eng. Technol., pages 141--171.
  Birkh\"auser/Springer, Cham, 2017.

\bibitem{befreyam'13}
A.~Bensoussan, J.~Frehse, and P.~Yam.
\newblock {\em Mean field games and mean field type control theory}.
\newblock SpringerBriefs in Mathematics. Springer, New York, 2013.

\bibitem{silva'18}
L.~M. Brice\~{n}o Arias, D.~Kalise, and F.~J. Silva.
\newblock Proximal methods for stationary mean field games with local
  couplings.
\newblock {\em SIAM J. Control Optim.}, 56(2):801--836, 2018.

\bibitem{carda'13}
P.~Cardaliaguet.
\newblock Long time average of first order mean field games and weak {KAM}
  theory.
\newblock {\em Dyn. Games Appl.}, 3(4):473--488, 2013.

\bibitem{CardaNotes}
P.~Cardaliaguet.
\newblock Notes on mean field games, 2013.
\newblock https://www.ceremade.dauphine.fr/~cardaliaguet/.

\bibitem{carsilva'14}
E.~Carlini and F.~J. Silva.
\newblock A fully discrete semi-{L}agrangian scheme for a first order mean
  field game problem.
\newblock {\em SIAM J. Numer. Anal.}, 52(1):45--67, 2014.

\bibitem{carsilva'15}
E.~Carlini and F.~J. Silva.
\newblock A semi-{L}agrangian scheme for a degenerate second order mean field
  game system.
\newblock {\em Discrete Contin. Dyn. Syst.}, 35(9):4269--4292, 2015.

\bibitem{cardela'18}
R.~Carmona and F.~Delarue.
\newblock {\em Probabilistic Theory of Mean Field Games with Applications I}.
\newblock Probability Theory and Stochastic Modelling. Springer, Cham, 2018.

\bibitem{notebari}
A.~Cesaroni and M.~Cirant.
\newblock Introduction to variational methods for viscous ergodic mean-field
  games with local coupling.
\newblock 2017.
\newblock Lecture notes in Springer INdAM Series, to appear.

\bibitem{chapock'11}
A.~Chambolle and T.~Pock.
\newblock A first-order primal-dual algorithm for convex problems
  with applications to imaging.
\newblock {\em Journal of Mathematical Imaging and Vision}, 40(1):120--145, May
  2011.

\bibitem{osher'17}
Y.~T. Chow, J.~Darbon, S.~Osher, and W.~Yin.
\newblock Algorithm for overcoming the curse of dimensionality for
  time-dependent non-convex {H}amilton-{J}acobi equations arising from optimal
  control and differential games problems.
\newblock {\em J. Sci. Comput.}, 73(2-3):617--643, 2017.

\bibitem{osher'18b}
Y.~T. Chow, W.~Li, S.~Osher, and W.~Yin.
\newblock Algorithm for {H}amilton-{J}acobi equations in density space via a
  generalized {H}opf formula.
\newblock {\em Preprint}, 2018.
\newblock arXiv:1805.01636v1 [math.NA].

\bibitem{fleson'93}
W.~H. Fleming and H.~M. Soner.
\newblock {\em Controlled {M}arkov processes and viscosity solutions},
  volume~25 of {\em Applications of Mathematics (New York)}.
\newblock Springer-Verlag, New York, 1993.

\bibitem{GoBook}
D.~A. Gomes, E.~A. Pimentel, and V.~Voskanyan.
\newblock {\em Regularity theory for mean-field game systems}.
\newblock SpringerBriefs in Mathematics. Springer, 2016.

\bibitem{Gomes2014}
D.~A. Gomes and J.~Sa{\'u}de.
\newblock Mean field games models---a brief survey.
\newblock {\em Dynamic Games and Applications}, 4(2):110--154, Jun 2014.

\bibitem{GJS2}
D.~A. Gomes and J.~Sa{\'u}de.
\newblock Numerical methods for finite-state mean-field games satisfying a
  monotonicity condition.
\newblock {\em Applied Mathematics {\&} Optimization}, Aug 2018.

\bibitem{gll'11}
O.~Gu{\'e}ant, J.-M. Lasry, and P.-L. Lions.
\newblock {\em Mean Field Games and Applications}, pages 205--266.
\newblock Springer Berlin Heidelberg, Berlin, Heidelberg, 2011.

\bibitem{HCM07}
M.~Huang, P.~E. Caines, and R.~P. Malham\'e.
\newblock Large-population cost-coupled {LQG} problems with nonuniform agents:
  individual-mass behavior and decentralized {$\epsilon$}-{N}ash equilibria.
\newblock {\em IEEE Trans. Automat. Control}, 52(9):1560--1571, 2007.

\bibitem{HCM06}
M.~Huang, R.~P. Malham\'e, and P.~E. Caines.
\newblock Large population stochastic dynamic games: closed-loop
  {M}c{K}ean-{V}lasov systems and the {N}ash certainty equivalence principle.
\newblock {\em Commun. Inf. Syst.}, 6(3):221--251, 2006.

\bibitem{LL061}
J.-M. Lasry and P.-L. Lions.
\newblock Jeux \`a champ moyen. {I}. {L}e cas stationnaire.
\newblock {\em C. R. Math. Acad. Sci. Paris}, 343(9):619--625, 2006.

\bibitem{LL062}
J.-M. Lasry and P.-L. Lions.
\newblock Jeux \`a champ moyen. {II}. {H}orizon fini et contr\^ole optimal.
\newblock {\em C. R. Math. Acad. Sci. Paris}, 343(10):679--684, 2006.

\bibitem{LL07}
J.-M. Lasry and P.-L. Lions.
\newblock Mean field games.
\newblock {\em Jpn. J. Math.}, 2(1):229--260, 2007.

\bibitem{osher'18}
A.~T. Lin, Y.~T. Chow, and S.~Osher.
\newblock A splitting method for overcoming the curse of dimensionality in
  {H}amilton-{J}acobi equations arising from nonlinear optimal control and
  differential games with applications to trajectory generation.
\newblock {\em Preprint}, 2018.
\newblock arXiv:1803.01215v1 [math.OC].

\bibitem{LCDF}
P.-L. Lions.
\newblock Coll\'ege de {F}rance course on mean-field games.
\newblock 2007-2011.
\newblock
  https://www.college-de-france.fr/site/en-pierre-louis-lions/index.htm.

\bibitem{nurbe'18}
L.~Nurbekyan.
\newblock One-dimensional, non-local, first-order stationary mean-field games
  with congestion: a {F}ourier approach.
\newblock {\em Discrete Contin. Dyn. Syst. Ser. S}, 11(5):963--990, 2018.

\bibitem{Yegorov2018}
I.~Yegorov and P.~M. Dower.
\newblock Perspectives on characteristics based curse-of-dimensionality-free
  numerical approaches for solving {H}amilton--{J}acobi equations.
\newblock {\em Applied Mathematics {\&} Optimization}, Jul 2018.

\end{thebibliography}

\end{document}